\documentclass[%
 aip,
 jmp,%
 amsmath,amssymb,
reprint,%
author-year,%
]{revtex4-2}

\usepackage{graphicx}
\usepackage{dcolumn}
\usepackage{bm}
\usepackage{subcaption}
\usepackage{amsthm}
\usepackage{booktabs}
\usepackage{makecell}
\usepackage{multirow}
\usepackage{array}
\usepackage{mathtools}
\newtheorem{theorem}{Theorem}
\DeclarePairedDelimiter{\abs}{\lvert}{\rvert}
\newcolumntype{C}[1]{>{\centering\arraybackslash}p{#1}}

\usepackage{siunitx}
\usepackage{style}
\usepackage{dsfont}
\usepackage{cases}
\usepackage[]{appendix}
\usepackage{caption}
\usepackage{bbold}

\newcommand{\bs}{\boldsymbol} 
\newtheorem{proposition}{Proposition}
\newtheorem{assumption}{Assumption}
\newtheorem{lemma}{Lemma}
\newtheorem*{example}{Example}

\begin{document}

\preprint{AIP/123-QED}

\title[]{The multivariate ARMA/CARMA transformation relation}
\thanks{The PhD grant of M.\ D.\ Eggen is funded by NORSAR.}

\author{Mari Dahl Eggen}
 \email{marideg@math.uio.no.}
\affiliation{ 
Department of Mathematics, University of Oslo, P.O.\ Box 1053 Blindern, 0316 Oslo, Norway.
}

\date{\today}

\begin{abstract}
A transformation relation between multivariate ARMA and CARMA processes is derived through a discretization procedure. This gives a direct relationship between the discrete time and continuous time analogues, serving as the basis for an estimation method for multivariate CARMA models. We will see that the autoregressive coefficients, making up the deterministic part of a multivariate CARMA model, are entirely given by the transformation relation. An Euler discretization convergence rate of jump diffusions is found for the case of small jumps of infinite variation. This substantiates applying the transformation relation for estimation of multivariate CARMA models driven by NIG-L{\'e}vy processes. A two-dimensional CAR model is fit to stratospheric temperature and wind data, as an example of how to apply the transformation relation in estimation methods. 
\end{abstract}

\maketitle


\section{\label{sect:introduction}Introduction}


The widely known family of autoregressive moving average (ARMA) processes holds properties suitable for discrete time series modelling (see, e.g., \cite{brockwell_davis_book91}). Another family of processes that has become important in representing time series is the family of continuous time ARMA (CARMA) processes, the continuous time analogue to ARMA processes. Continuous time models allow for irregularly spaced time series and provide the opportunity to derive explicit formulas describing events and properties relying on dynamical systems. Univariate CARMA processes and generalizations are used to model, for example, variables in finance and energy markets, weather variables and turbulence (see, e.g., \cite{todorov_etal06}, \cite{garcia_etal11}, \cite{brockwell_etal13_wind}, \cite{barndorff-nielsen07}). For an extensive overview of developments and applications of CARMA processes, see \cite{brockwell2014recent} and references therein. \\

Most time series arising from dynamical systems in fields like natural sciences, finance and economics will be more accurately represented by multidimensional models describing dependencies between two or more variables within a system. Discrete time multivariate ARMA processes (also called vector ARMA (VARMA) processes), see \cite{brockwell_davis_book91}, are used as models for this purpose, see, e.g. \cite{gomez19} and \cite{wei_19} for practical examples. As argued above, modelling continuous time series is often useful, with no exception in the multidimensional case. A particularly useful property of a multivariate continuous time version of VARMA models in applications, would be the possibility of deriving explicit formulas for the crosscorrelation matrix between modelled time series. The first derivation and proper representation of multivariate CARMA (MCARMA) processes is derived in \cite{marquardt07}. They are shown to be the obvious continuous analogue of VARMA processes. \\

As a natural extension of the work in \cite{marquardt07}, where proper analogy between VARMA and MCARMA processes is concluded, \cite{schlemm12} derive fundamental results for developing estimation theory for non-Gaussian MCARMA processes based on (discrete) equidistant observations. As stated in \cite{brockwell13}, estimating a CARMA model consists of three tasks: 1) choosing suitable integer values $p$ and $q$, respectively describing the autoregressive and moving average orders; 2) estimating autoregressive and moving average coefficients; 3) suggesting an appropriate stochastic process to drive the model. These tasks hold for MCARMA models as well. Statistical tests for choosing $p$ and $q$ already exist for VARMA processes when normally distributed error terms are assumed, see e.g. \cite{gomez19}, however, a best consensus method does not seem to exist. Results in \cite{schlemm12} are utilized in \cite{stelzer12} to derive an estimation procedure for autoregressive and moving average model coefficients in the multivariate case. Some further efforts have been made to enrich model coefficient estimation theory for MCARMA processes, see, e.g., \cite{mayer21} and \cite{fasenhartmann21}. Finally, as stated in \cite{brockwell13}, there are several well-established methods to estimate an appropriate driving L{\'e}vy process. The work in \cite{brockwell13} focuses on a parametric estimation method of discretely observed MCARMA processes. \\

Let $\bs{Y}(t)$ represent the unique causal stationary solution, see \cite{schlemm12}, of an MCARMA process, and define its discrete time $h$-sampled analogue as $\bs{Y}^h(t)\triangleq \{\bs{Y}(nh)\}_{n\in\mathbb{Z}}$. As shown in \cite{fasenhartmann21}, discretely sampled MCARMA processes admits a VARMA representation with a weak white noise. In \cite{brockwell19}, determining the coefficients of $\bs{Y}(t)$ from those of $\bs{Y}^h(t)$ and vice versa is referred to as the sampling problem and the embedding problem respectively. In particular, \cite{fasenhartmann21}, \cite{chambers12} and \cite{thornton17} focus on the sampling problem, where discrete time representations of sampled MCARMA processes are estimated. To the best of our knowledge, the embedding problem for MCARMA processes is not yet considered from the view of discretization methods. However, discretization of CAR processes leading to a transformation relation between AR and CAR processes is derived in \cite{benth2008}. Work leading up to such discretization transformations is found in, e.g., \cite{pham91} and \cite{soderstrom97}. \\

The aim of this study is to derive a transformation relation between VARMA and MCARMA processes through discretization of the MCARMA state space representation. Define a L{\'e}vy-driven MCRAMA process as the solution, $\bs{Y}(t)$, of the state space model 
\begin{align*}
    d\bs{X}(t) = A\bs{X}(t)dt + \beta d\bs{L}(t),\quad \bs{Y}(t) = C\bs{X}(t), \quad t\in\mathbb{R},
\end{align*}
see Section\,\ref{subsec:varma_mcarma_definition} for a formal definition. The multidimensional L{\'e}vy-driven stochastic differential equation (SDE) of Ornstein--Uhlenbeck type, representing a system of one-dimensional SDEs, is discretized using the Euler scheme. The discretized system of SDEs is further solved in a recursive manner to estimate the state vector process $\bs{X}(t)$. The solution admits a VARMA representation, giving a direct transformation relation between VARMA and MCARMA processes.\\

Inspired by \cite{asmussen01} and \cite{benth11}, convergence rates of (multidimensional) jump diffusions with jumps of finite variance and infinite variations are studied. The final result gives a convergence rate of these processes' Euler scheme, see, e.g., \cite{platen10}. In particular, it is shown that the rate of convergence is controllable by adjusting the discretization step size, as usual, and by the size of the approximated small jumps. The state space representation of the L{\'e}vy-driven MCARMA process is a special case of these jumps diffusions, meaning that the convergence results hold for such models when driven by L{\'e}vy processes of finite variance and infinite variation. \\

VARMA models have been used extensively in modelling of economic and financial variables, see, e.g., \cite{james85} and \cite{gomez19}. In modelling and prediction of climate and weather variables, VARMA models have been utilized to some degree, e.g., \cite{parlange00}, \cite{love08} and \cite{bs21}. The continuous time counterpart would be useful in applications as well, see \cite{eggen2021}. The VAR model is a simplified version of the VARMA model, with the MCAR process as its continuous counterpart. The multivariate ARMA/CARMA transformation relation is considerably simplified in the VAR/MCAR case, and is easy to use in model estimation of MCAR processes. To demonstrate how the transformation relation is used as a tool in model estimation, a case study is performed where a two-dimensional MCAR model is fit to weather data. \\

The atmospheric layer closest to the surface of the Earth is called the troposphere. Above the troposphere lays the stratosphere, reaching up to about $50$\;km above the surface of the Earth. These two atmospheric layers are said to interact through stratosphere-troposphere coupling, as weather conditions in the stratosphere affect weather conditions in the troposphere. As explained in, e.g., \cite{karpechko16} and \cite{scaife22}, probing and proper representation of the stratosphere, combined with a thorough understanding of stratosphere-troposphere coupling, has the potential to enhance long-term numerical surface weather prediction. In this regard, stochastic modelling of stratospheric weather dynamics will enlarge the ensemble of existing methods. Inspired by prior work, e.g., \cite{benth2008} and \cite{eggen2021}, an MCAR process is successfully fit to stratospheric temperature and wind data, giving a representation of the two-dimensional dynamical system of dependent variables. \\

The structure of the paper is as follows. Section\,\ref{sect:varma_intro} defines the VARMA and MCARMA representations, and introduces useful
notation. In Section\,\ref{sect:transformation_relation}, the multivariate ARMA/CARMA transformation relation is found through discretization of the MCARMA state space representation. Further, convergence rates substantiating the transformation relation are derived. A case study demonstrating how to use the transformation relation in MCARMA model estimation is performed in Section\,\ref{sect:the_mcar_model}. \\


\begin{flushleft}
\textbf{Notation}\\
\end{flushleft}
Assume that a complete filtered probability space $(\Omega,\mathcal{F},\{\mathcal{F}_t\}_{t\geq 0},P)$ is given as usual, and that all stochastic (vector) processes $\bs{X}(t) \triangleq \{\bs{X}(t)\}_{t\geq 0}$ are defined on that space. \\

Block-matrices are denoted as $M\in \R^{m\times n}$, having associated matrices $M_j$ with matrix elements $\mu_i^{(j)}$. Vectors are denoted as $\bs{V}\in\R^{m}$, with associated sub-vectors $\bs{V}_j$ and elements $V_i$. The $d$-dimensional identity and zero matrices are given by $\mathbb{1}_d,\mathbb{0}_d\in\R^{d\times d}$. Denote by $I_B(\cdot)$ the indicator function of some set $B$, and let $\det(\cdot)$ denote the determinant of matrices. Further, we work under the norm $\|\cdot\|_2 = (E[\abs{\cdot}^2])^{1/2}$ throughout. \\

For convenience, let us recall the general representation of L{\'e}vy processes through It{\^o}-L{\'e}vy decomposition. Let $\bs{L}(t)$ be a L{\'e}vy process with values in $\R^{m}$. Then, for each $t\geq 0$, there exist constant-valued functions $\bs{\alpha}\in\R^{m}$ and $\sigma\in \R^{m\times n}$, referred to as drift and diffusion respectively, such that 
\begin{align}
    \label{eq:ito-levy_decomposition}
    \bs{L}(t) = \bs{\alpha} t + \sigma W(t) + \int_{\abs{\bs{z}}<1}\bs{z}\tilde{N}(t,d\bs{z}) + \int_{\abs{\bs{z}}\geq 1}\bs{z}N(t,d\bs{z}),
\end{align}
where $\tilde{N}(dt,d\bs{z})=N(dt,d\bs{z})-\nu(d\bs{z})dt$, $\bs{z}\in \R^{m}$, $\nu(U)$ is a Borel measure on $\R^{m}\backslash \{0\}$, $\bs{W}(t)$ is a standard Brownian motion process in $\R^{n}$ and $N(t,U)$ is a Poisson random measure on $\R^{+}\times (\R^{m}\backslash \{0\})$. Further, $\bs{W}(t)$ and $N(t,U)$ are independent, and $\nu(U)$ is a L{\'e}vy measure, meaning $\int_{\R^{m}\backslash \{0\}}(\abs{\bs{z}}^2\wedge 1)\nu(d\bs{z})< \infty$ is satisfied. We will further assume that $\int_{\abs{\bs{z}}\geq \epsilon}\abs{\bs{z}}^2\nu(d\bs{z})< \infty$ for some $0<\epsilon\leq 1$, such that the L{\'e}vy process has finite second moments. See, e.g., \cite{applebaum04} for a thorough introduction of L{\'e}vy processes. \\


\section{\label{sect:varma_intro}Multivariate ARMA and CARMA models}


This section introduces the modelling framework of discrete time VARMA and continuous time MCARMA processes. With the intention of deriving a transformation relation between VARMA and MCARMA processes in Section\,\ref{sect:transformation_relation}, the the MCARMA model structure is studied more explicitly, and a recursive parameter connected to the notation of the defined MCARMA framework is defined. \\ 


\subsection{Model representations \label{subsec:varma_mcarma_definition}}


A short introduction to VARMA and MCARMA processes is presented in the following. The notation and definitions are inspired by, \cite{gomez19}, \cite{marquardt07} and \cite{schlemm12}. \\


Define the autoregressive and moving average matrix polynomials as
\begin{align}
    \label{eq:varma_matrix_polynomials}
    \phi(z) = \left( I - \phi_1z - \phi_2z^2 - \cdots - \phi_pz^p \right)\quad\text{and}\quad
    \theta(z) = \left( I + \theta_1z + \theta_2z^2 + \cdots + \theta_qz^q \right), 
\end{align}
respectively, and a backshift operator S as $S\bs{x}(t) = \bs{x}(t-1)$. Then, the VARMA process might be stated as 
\begin{align}
    \label{eq:def_of_varma_mode}
    \phi(S)\bs{x}(t) = \theta(S)\bs{\mathcal{E}}(t),
\end{align}
where $\bs{x}(t)\in\R^d$ is a sequence of random vectors and $\bs{\mathcal{E}}(t)\in\R^m$ is a sequence of serially uncorrelated i.i.d. random vectors with zero mean and common covariance matrix $\Sigma$. Stationarity and invertibility of the model is secured as long as all roots of $\det(\phi(z))$ and $\det(\theta(z))$ are outside the unit circle, respectively, see \cite{Levendis18} and \cite{gomez16}. As seen in Eq.\,\eqref{eq:varma_matrix_polynomials} and \eqref{eq:def_of_varma_mode}, the VARMA model, sometimes denoted as VARMA($p,q$), is determined by an autoregressive parameter, $p\in\N$, a moving average parameter, $q\in\N$, and the dimensionalities $d,m\in\N$. \\

In \cite{marquardt07}, the need for MCARMA processes were addressed, and further derived, such that the joint behaviour of $d$ different time series can be modelled continuously. As far as we know, this is the first mentioning of MCARMA processes in literature. We adapt the notation from \cite{marquardt07} and \cite{schlemm12}, and define the MCARMA process as follows. \\

Let $\bs{L}(t)$ be a L{\'e}vy process with values in $\R^m$, as defined in Eq.\,\eqref{eq:ito-levy_decomposition}. Then the $d$-dimensional L{\'e}vy-driven MCARMA process, $\bs{Y}(t)$, is given by the state space representation 
\begin{align}
    \label{eq:the_mcarma_model}
    d\bs{X}(t) = A\bs{X}(t)dt + \beta d\bs{L}(t),\quad \bs{Y}(t) = C\bs{X}(t).
\end{align}
Here, $\bs{X}(t)$ is required to be a unique stationary solution of the given Ornstein--Uhlenbeck type state space model, with $A$, $\beta$ and $C$ given as
\begin{align*}
     A = \begin{bmatrix}
       \mathbb{0}_d & \mathbb{1}_d & \mathbb{0}_d & \cdots & \mathbb{0}_d       \\
       \mathbb{0}_d & \mathbb{0}_d & \mathbb{1}_d &\cdots & \mathbb{0}_d        \\
       \vdots & \vdots & \vdots &\vdots &\vdots \\
       \mathbb{0}_d & \mathbb{0}_d & \mathbb{0}_d & \cdots & \mathbb{1}_d \\
       -A_p & -A_{p-1} & -A_{p-2} & \cdots & -A_1 \\
     \end{bmatrix}\in \R^{pd\times pd}, \quad \beta = \left(\beta_1^T\cdots \beta_p^T\right)^T\in \R^{pd\times m},
\end{align*}
\begin{align*}
    \beta_{p-\kappa} = -I_{\{0,\ldots ,q\}}(\kappa)\left[\sum_{i=1}^{p-\kappa-1} A_i\beta_{p-\kappa-i} - B_{q-\kappa}\right] \quad\text{and}\quad 
        C = \left(\mathbb{1}_d,\mathbb{0}_d,\ldots , \mathbb{0}_d\right)\in \R^{d\times pd},
\end{align*}
where $\mathbb{1}_d,\mathbb{0}_d\in\R^{d\times d}$ represents the $d$-dimensional identity and zero matrices respectively, and $I_{\{0,\ldots ,q\}}(\cdot)$ is the indicator function of the set $\{0,\ldots ,q\}$. Further, we have that $A_j\in \R^{d\times d}$, and $B_{q-\kappa}\in \R^{d\times m}$ for $j\in\{1,\ldots ,p\}$, $\kappa\in\{0,\ldots ,q\}$ and positive integers $p>q$. Notice also that $\beta_{p-\kappa}\in \R^{d\times m}$ for $0\leq \kappa \leq p-1$, which is a zero matrix if $\kappa\notin \{0,\ldots ,q\}$. As long as the driving process, $\bs{L}(t)$, admits finite variance, the MCARMA process is assured to have a unique causal stationary solution when the eigenvalues of $A$ have strictly negative real parts.  \\


\subsection{The MCARMA model structure\label{subsect:the_model_structure}}


The aim of this work is to derive a transformation relation between discrete time VARMA processes and continuous time MCARMA processes through discretization of the MCARMA representation in Eq.\,\eqref{eq:the_mcarma_model}. To prepare for this the state space model, $\bs{X}(t)$, is stated and inspected on a more explicit level. That is, the block-matrix $A$ is written explicitly with conveniently defined notation, and each matrix $A_j$ is structured into collections. The defined notation is used to construct a recursive parameter in Section\,\ref{subsect:the_recursive_parameter}, which is further utilized to solve the discretized system of SDEs making up the MCARMA model. The mentioned collection-structure is made to ease readability and understanding throughout this process. \\

As indicated in Section\,\ref{subsec:varma_mcarma_definition}, the state vector process $\bs{X}(t)$ and L{\'e}vy vector process $\bs{L}(t)$ takes values in $\R^{pd}$ and $\R^{m}$ respectively. Denote each element of these vector processes as $X_{i_X}(t)$, $i_X\in\{1,\ldots ,pd\}$, and $L_{i_L}(t)$, $i_L\in\{1,\ldots ,m\}$, and further denote each $d$-dimensional sub-vector of $\bs{X}(t)$ as $\bs{X}_l(t)\in\R^d$, $l\in\{1,\ldots ,p\}$. Then the state space model in Eq.\,\eqref{eq:the_mcarma_model} might be considered as a system of $p$ equations blocks
\begin{numcases}{}
    \label{eq:sde_recursive_equations}d\bs{X}_{l}(t) = \mathbb{1}_d\bs{X}_{l+1}(t)dt + \beta_{p-\kappa}d\bs{L}(t), \\
    \label{eq:sde_lag_dependence_structure} d\bs{X}_{p}(t) = [-A_p\cdots -A_1]\bs{X}(t)dt + \beta_{p}d\bs{L}(t),
\end{numcases}
with $(l,\kappa)\in\{(1,p-1),(2,p-2),\ldots ,(p-1,1)\}$. Note that the equation block in Eq.\,\eqref{eq:sde_lag_dependence_structure} corresponds to $l=p$ and $\kappa=0$. The index-dependence between the block number index, $l$, and the moving average index, $\kappa$, is summarised in Table\,\ref{tab:index_collection_dependence}, together with the corresponding index number of each one-dimensional SDE in the system. \\

Further, group the $p$ equation blocks into three disjoint collections. That is, define the solution-collection $\mathcal{C}^{S}$ as holding SDEs in block $1$ ($l=1$), the recursive-collection $\mathcal{C}^{R}$ as holding SDEs in blocks $2$ to $p-1$ ($l=2,\ldots ,p-1$) and the AR-collection $\mathcal{C}^{AR}$ as holding SDEs in block $p$. The collection-structure is presented in Table\,\ref{tab:index_collection_dependence}. A $d$-dimensional dynamical system known to follow an MCARMA process, $\bs{Y}(t)$, is given by the vector process $\bs{X}_1(t)$ (see the model setup in Eq.\,\eqref{eq:the_mcarma_model} to understand why). This is why $\mathcal{C}^{S}$ is called the solution-collection. The time lag dependence structure of the dynamical system is entirely described through the SDEs in Eq.\,\eqref{eq:sde_lag_dependence_structure}, giving the name to the AR-collection, $\mathcal{C}^{AR}$. Finally, the system of one-dimensional SDEs making up the MCARMA model has to be solved recursively to obtain $\bs{X}_1(t)$, with the SDEs in $\mathcal{C}^{AR}$ as a starting point. The SDEs in the recursive-collection, $\mathcal{C}^{R}$, are used for this purpose. Notice that an increasing number of lags, $p$, in the MCARMA model requires a larger collection $\mathcal{C}^{R}$ to make the system of equations solvable. \\

\begin{table}[bt]
\centering
\caption{Dependence structure between collections $\mathcal{C}^{S}$, $\mathcal{C}^{R}$ and $\mathcal{C}^{AR}$, block number index, $l$, moving average index, $\kappa$, and SDE index numbers. \label{tab:index_collection_dependence}}
\begin{tabular}{l c c l}
\toprule \\ 
Collection \hspace{1mm} & \hspace{1mm} Block number index, $l$ \hspace{1mm} & \hspace{1mm} Moving average index, $\kappa$ \hspace{1mm} & \hspace{1mm} SDE index number \\
\midrule
$\mathcal{C}^{S}$ & 1  & p-1 & $\{1,\ldots ,d\}$  \\
\multirow{4}{1cm}{$\mathcal{C}^{R}$} &  2 & p-2 & \multirow{4}{3cm}{$\{d+1,\ldots ,pd-d\}$}\\
 &  \vdots & \vdots & \\
 & p-2  & 2 & \\
 &  p-1 & 1 &  \\
$\mathcal{C}^{AR}$ &  p & 0 & $\{pd-d+1,\ldots ,pd\}$ \\
\bottomrule  
\end{tabular}
\end{table}

A general solution, $\bs{X}_1(t)$, of the state space model representing the MCARMA process will be derived through the Euler scheme of $d\bs{X}(t)$. Note that the derived solution will be an approximation with accuracy depending on the incremental value defining the discretization scheme, see Section\,\ref{sect:transformation_relation}. To prepare for element-wise discretization, define each matrix $A_j$, \,$j\in\{1,\ldots ,p\}$, and $\beta_{p-\kappa}$, \,$\kappa\in\{0,\ldots ,p-1\}$, in the respective block-matrices $A$ and $\beta$, as
\begin{align}
    \label{eq:matrix_element_definition}
     A_j = \begin{bmatrix}
       \alpha_1^{(
       j)} & \alpha_2^{(j)} & \cdots & \alpha_d^{(j)}       \\
       \alpha_{d+1}^{(j)} &  \cdots &\cdots & \alpha_{2d}^{(j)}        \\
        \vdots & \vdots &\vdots &\vdots \\
        \alpha_{(d-1)d+1}^{(j)} & \cdots & \cdots & \alpha_{dd}^{(j)} 
     \end{bmatrix}\quad \text{and} \quad \beta_{p-\kappa} = \begin{bmatrix}
       \beta_1^{(\kappa)} & \beta_2^{(\kappa)} & \cdots & \beta_m^{(\kappa)}       \\
       \beta_{m+1}^{(\kappa)} &  \cdots &\cdots & \beta_{2m}^{(\kappa)}        \\
        \vdots & \vdots &\vdots &\vdots \\
        \beta_{(d-1)m+1}^{(\kappa)} & \cdots & \cdots & \beta_{dm}^{(\kappa)} 
     \end{bmatrix}.
\end{align}
This gives elements 
\begin{align}
    \label{eq:elementwise_p_and_q_terms}
        -\sum_{l=1}^{p}\sum_{s=1}^{d}\alpha_{(k-1)d+s}^{(p-l+1)}X_{(l-1)d+s}(t)dt \quad\text{and}\quad \sum_{r=1}^{m}\beta_{(k-1)m+r}^{(\kappa)}dL_r(t),
\end{align}
of the respective matrix-vector products $[-A_p\cdots -A_1]\bs{X}(t)dt$ and $\beta_{p-\kappa} d\bs{L}(t)$, for all $\kappa$. Here, $\kappa$ depends on the equation block, $l$, under consideration and $k\in\{1,\ldots ,d\}$ is set based on the dimension of interest. For example, to state the SDE of dimension $2$ in equation block $p$, set $\kappa=0$ (confer with Table\,\ref{tab:index_collection_dependence}) and $k=2$. Also, remember that the element-wise subscripts of $\bs{X}(t)$ and $\bs{L}(t)$ are defined as $i_X\in\{1,\ldots ,pd\}$ and $i_L\in\{1,\ldots ,m\}$ respectively, corresponding to the respective elements in Eq.\eqref{eq:elementwise_p_and_q_terms}. \\


\subsection{The recursive parameter\label{subsect:the_recursive_parameter}}


The solution, $\bs{Y}(t)$, of the $d$-dimensional MCARMA($p,q$) process is given by the state space model in Eq.\,\eqref{eq:the_mcarma_model}. In this section, a recursive parameter is defined to solve a discretized version of the state space model recursively. We will see that the recursive parameter can be used to express all the one-dimensional SDEs in Eq.\,\eqref{eq:sde_recursive_equations}. \\

It is important to realize that the $d$-dimensions of an MCARMA model each represents a dynamical variable, each dependent on $p$ previous time steps, also across dimensions, with a moving average of degree $q$. The $p$ times lagged $d$ dimensions are represented through $d$ SDEs in each of the $p$ equation blocks introduced in Section\,\ref{subsect:the_model_structure} (see Eq.\,\eqref{eq:sde_recursive_equations} and \eqref{eq:sde_lag_dependence_structure}). That is, the collections $\mathcal{C}^{S}$, $\mathcal{C}^{R}$ and $\mathcal{C}^{AR}$ holds $d$, $(p-2)d$ and $d$ SDEs, respectively (see Table\,\ref{tab:index_collection_dependence}). Remember that the $(p-2)d$ SDEs in $\mathcal{C}^{R}$ are distributed into $p-2$ equation blocks (that is block $2$ to block $p-1$), where block $p-1$ represents the first recursive step, block $p-2$ the second recursive step, and so on. To keep track of all the one-dimensional SDEs in a recursive procedure from block $p$ to block $1$, a recursive parameter is introduced as
\begin{align}
    \label{eq:the_recursive_parameter}
    Q_i^{(l)} \triangleq (Q_i^{(l)}\mid k) = (l-i)d + k,
\end{align}
where $k\in\{1,\ldots , d\}$ and $l\in\{1,\ldots ,p\}$, such that $1\leq i\leq l$ for every $l$, and $i>1$ when $l=p$. Note that $d$ is the fixed dimension parameter of the MCARMA process, $i$ and $l$ are dynamical indexes holding track of equation blocks $1$ to $p-1$ during the recursive procedure, and that $k$ is set depending on the dimension of interest. \\

To understand how the recursive parameter is utilized, consider the following. Let the recursive parameter represent element index, $i_X$, of $\bs{X}(t)$, meaning that elements of $\bs{X}(t)$ are given as $X_{Q_i^{(l)}}(t)$. By the definition in Eq.\,\eqref{eq:the_recursive_parameter}, the recursive parameter takes values in $\{1,\ldots ,pd-d\}$, corresponding to SDEs in $\mathcal{C}^S$ and $\mathcal{C}^R$ (see Table\,\ref{tab:index_collection_dependence}). Further, the recursive parameter holds two properties that are utilized in the recursive procedure:
\begin{enumerate}
    \item[Property 1.] SDEs in an equation block with lower block number index, $l$, might be expressed in more ways using $X_{Q_i^{(l)}}(t)$ than SDEs in an equation block with higher block number index. This property is useful because the recursive procedure iterates through SDEs in an equation block with lower block number index more times. That is,
    in the first recursive step SDEs from block $p$ are substituted into SDEs in block $p-1$, in the second recursive step these SDEs are further substituted into SDEs in block $p-2$. This procedure continues until the last recursive step where the result is substituted into SDEs in block $1$. Substituted SDEs from block $p-1$ into SDEs in block $p-2$ (in the second recursive step), which now also contain substituted SDEs from block $p$, will be substituted into one less block. A similar argument holds for all further recursive steps starting in blocks $p-2$ to $2$. See an illustrative example below of this recursive parameter property for a case with $p=3$, where SDEs in equation block $2$ can be expressed using $X_{Q_i^{(l)}}(t)$ in two different ways, and SDEs in equation block $1$ can be expressed using $X_{Q_i^{(l)}}(t)$ in three different ways. 
    \item[Property 2.] By definition (Eq.\,\eqref{eq:the_recursive_parameter}), we have that $Q_1^{(l)}=Q_1^{(l)}-d+d = Q_2^{(l)}+d$. Similarly, $Q_2^{(l)} = Q_3^{(l)}+d$, and so on, until $Q_{p-1}^{(l)}=Q_p^{(p)}+d$. This iterative property is used to conveniently express the SDEs in collections $\mathcal{C}^{S}$ and $\mathcal{C}^{R}$ (see Proposition\,\ref{prop:q_stochastic_diff_eq}), and further to recursively solve the discretized $d$-dimensional MCARMA process by the recursive procedure from equation block corresponding to $X_{Q_1^{(l)}}$, for all $l$, to the final solution block corresponding to $X_{Q_p^{(p)}}(t)=X_k(t)\in\bs{X}_1(t)$.
\end{enumerate}

The following proposition states how the recursive parameter might be used to represent the one-dimensional SDEs in collections $\mathcal{C}^{S}$ and $\mathcal{C}^{R}$.
\begin{proposition}
    \label{prop:q_stochastic_diff_eq}
Let the recursive parameter be as given in Eq.\,\eqref{eq:the_recursive_parameter}, where $k\in\{1,\ldots , d\}$ and $l\in\{1,\ldots ,p\}$. For $1\leq l \leq p$ and $1\leq i\leq l$ ($i>1$ when $l=p$), corresponding to SDE numbers in collections $\mathcal{C}^{S}$ and $\mathcal{C}^{R}$ (see Table\,\ref{tab:index_collection_dependence}), we have that 
\begin{align}
    \label{eq:q_stochastic_diff_eq}
     X_{Q_i^{(l)}+d}(t)dt = dX_{Q_i^{(l)}}(t) - \sum_{r=1}^{m}\beta_{(k-1)m+r}^{(p-l+i-1)}dL_r(t).
\end{align}
\end{proposition}
\begin{proof}
See Appendix \ref{app:1}.
\end{proof}

\begin{example}
The MCARMA process dynamics with given parameters $p=4$, $d=2$ and $m=d$ is given by 
\begin{align}
    \begin{cases}
    \label{eq:example_sdes}
        d\bs{X}_1(t) = \bs{X}_2(t)dt +  \beta_1d\bs{L}(t)    \\
        d\bs{X}_2(t) = \bs{X}_3(t)dt +  \beta_2d\bs{L}(t)    \\
        d\bs{X}_3(t) = \bs{X}_4(t)dt +  \beta_3d\bs{L}(t)    \\
        d\bs{X}_4(t) = [-A_4\cdots -A_1]\bs{X}(t)dt +  \beta_4d\bs{L}(t),    
    \end{cases}
\end{align}
where $\bs{X}_1(t) = [X_1(t),X_2(t)],\ldots ,\bs{X}_4(t) = [X_7(t),X_8(t)]$, $\bs{X}(t)=[\bs{X}_1(t),\ldots ,\bs{X}_4(t)]$, $\bs{L}(t) = [L_1(t),L_2(t)]$, and $A_j$ and $\beta_{p-\kappa}$ are given in Eq.\,\eqref{eq:matrix_element_definition}. Notice that Eq.\,\eqref{eq:example_sdes} is a system of eight SDEs in four equation blocks. That is, one equation block in $\mathcal{C}^S$, two in $\mathcal{C}^R$ and one in $\mathcal{C}^{AR}$ (see Table\,\ref{tab:index_collection_dependence}). Proposition\,\ref{prop:q_stochastic_diff_eq} states that SDEs in collections $\mathcal{C}^{S}$ and $\mathcal{C}^{R}$ (corresponding to Eq.\,\eqref{eq:sde_recursive_equations}) can be written in terms of the recursive parameter. The following points illustrate how property 1 of the recursive parameter is utilized to recursively solve the discretized $d$-dimensional MCARMA process. Concentrating on dimension number $2$, corresponding to considering a recursive parameter with $k=2$ (see Eq.\,\eqref{eq:the_recursive_parameter}) we see that $dX_2(t)$ might be expressed by the recursive parameter in three different ways:
\begin{enumerate}
    \item With equation block $3$ as a starting point in the backwards recursive procedure, corresponding to equation number $Q_1^{(3)}$, $dX_2(t)$ is obtained when we hit $Q_3^{(3)}$, after two recursive steps;
    \item With equation block $2$ as starting point in the backwards recursive procedure, corresponding to equation number $Q_1^{(2)}$, $dX_2(t)$ is obtained when we hit $Q_2^{(2)}$, after one recursive step;
    \item With equation block $1$ as starting point in the backwards recursive procedure, $dX_2(t)$ is already obtained by $Q_1^{(1)}$.
\end{enumerate}
Further, $dX_4(t)$ might be expressed by the recursive parameter in two different ways:
\begin{enumerate}
    \item With equation block $3$ as a starting point in the backwards recursive procedure, corresponding to equation number $Q_1^{(3)}$, $dX_2(t)$ is obtained when we hit $Q_2^{(3)}$, after one recursive step;
    \item With equation block $2$ as starting point in the backwards recursive procedure, $dX_4(t)$ is already obtained by $Q_1^{(2)}$.
    \end{enumerate}
Notice that the SDEs in equation block $2$ might be expressed with the recursive parameter when starting the recursive procedure in block $3$ or $2$, however, not when starting it in equation block $1$. SDEs in equation block $1$ might be expressed with the recursive parameter regardless of the starting point of the recursive procedure. This is because the recursive parameter is defined to follow a backwards recursive pattern. All of the starting points, blocks $p$ to $2$, of the backwards recursive procedure of a general MCARMA process ends in block $1$ (SDEs in $\mathcal{C}^S$) with $X_{Q^{(l)}_{l}}(t)$, where $l$ equals the block number index of the recursive procedure starting point.
\end{example}


\section{\label{sect:transformation_relation}The transformation relation}


The L{\'e}vy-driven MCARMA process dynamics is discretized using an Euler scheme. Through the discretized version a transformation relation between the continuous time MCARMA process and the discrete time VARMA process is derived. The convergence rate of the discretized MCARMA process dynamics is finally assessed for a driving L{\'e}vy process with finite variance and infinite variations, which is relevant for the case study in Section\,\ref{sect:the_mcar_model}.


\subsection{From the MCARMA Euler scheme to the transformation relation\label{subsect:mcarma_euler}}


In this section, an approximated solution, $\bs{x}_1(t)$, of the $d$-dimensional MCARMA($p,q$) process with an $m$-dimensional driving L{\'e}vy process is found through discretization of the state space model dynamics, $d\bs{X}(t)$, in Eq.\,\eqref{eq:the_mcarma_model}. A backwards recursive procedure is performed on the discretized process to find the solution, which finally takes the form of a VARMA process with coefficients given by the MCARMA coefficients. This is what we refer to as the multivariate ARMA/CARMA transformation relation, see Theorem\,\ref{thm:the_transformation_relation}. \\

Consider the equidistant time discretization 
\begin{align}
\label{eq:equidistiant_time_discretization}
t^h = \{0=t_0<t_1<\dots <t_N=T\},    
\end{align}
of a given time interval $[0,T]$ with fixed incremental value $(0,1]\ni h = t_{i+1}-t_i$, for $i\in\{0,\ldots N-1\}$. That is, the discretized time interval might be written as $[0,h,2h,\ldots , Nh]$, and a time step from an arbitrary point in time, $t$, is represented by $t+h$. Assume that the law of stationary and independent increments $\Delta L(t)\triangleq L(t+h) - L(t)$ is known. Then, using the Euler scheme (see \cite{kloeden92} and \cite{protter97}), the SDEs in $\mathcal{C}^S\cup\mathcal{C}^R$ (see Proposition\,\ref{prop:q_stochastic_diff_eq}, Eq.\,\eqref{eq:q_stochastic_diff_eq}) might be written as the piecewise constant process 
\begin{align}
    \label{eq:q_stochastic_diff_eq_discretization}
    x_{Q_i^{(l)}+d}(t) = \frac{1}{h}\left(x_{Q_i^{(l)}}(t+h) - x_{Q_i^{(l)}}(t) - \sum_{r=1}^m\beta_{(k-1)m+r}^{(p-l+i-1)}\Delta L_r (t)\right).
\end{align}

Through a backwards recursive procedure of discretized SDEs on this form, an approximated solution of the MCARMA process (Eq.\,\eqref{eq:the_mcarma_model}) is found as $\bs{X}_1(t)\simeq \bs{x}_1(t) = [x_{(Q_{l}^{(l)}\mid k=1)}(t),\ldots ,x_{(Q_{l}^{(l)}\mid k=d)}(t)]= [x_1(t),\ldots ,x_d(t)]$. The backwards recursive procedure has $p-1$ starting points, namely equation blocks $p$ to $2$ (corresponding to SDEs in $\mathcal{C}^{R}\cup\mathcal{C}^{AR}$), where the SDE index number in each of these blocks are given by $Q_{1}^{(l)}$, for $l=2,\ldots ,p$. The final goal is to express all SDEs represented by $Q_{1}^{(l)}$, in terms of $Q_{l}^{(l)}$. See the example in Section\,\ref{subsect:the_recursive_parameter} to understand how the recursive procedure is iterated. As seen in Eq.\,\eqref{eq:sde_recursive_equations} and \eqref{eq:sde_lag_dependence_structure}, the structures of the SDEs in $\mathcal{C}^R$ and $\mathcal{C}^{AR}$ are different. In the following lemma, SDEs represented by $Q_{1}^{(l)}$ is written in terms of $Q_{l}^{(l)}$, for the the collection $\mathcal{C}^R$. This will further be modified in terms of the SDEs in $\mathcal{C}^{AR}$ to find the final approximated solution. \\

\begin{lemma}
\label{lemma:euler_approx_base_equations}
 The approximated solution of SDEs in equation block $l\in\{1,\ldots ,p-1\}$, for $i=1$, corresponding to the starting point of each recursive step in the derivation of the transformation relation in Theorem\,\ref{thm:the_transformation_relation}, is given by  
 \begin{align*}
    x_{Q_1^{(l)}+d}(t) =& \frac{1}{h^{l}}\sum_{n=0}^{l}(-1)^n b_n^{l} x_{k}(t + (l-n)h) \\
    & - \sum_{w=0}^{l-1}\frac{1}{h^{w+1}}\sum_{r=1}^{m}\beta_{(k-1)m+r}^{(p-l+w)}\sum_{v=0}^{w}(-1)^vb_v^w\Delta L_r(t+(w-v)h),
\end{align*}
where $p$ is the number of lags, $d$ is the total number of dimensions, $k\in\{1,\ldots ,d\}$ is the dimension of interest, $m$ is the number of independent driving L{\'e}vy processes, $b_n^{i}$ is defined recursively from Eq.\,\eqref{eq:pascal_triangle} (Appendix\,\ref{app:2}), and $h$ is the Euler discretization step size.
\end{lemma}
\begin{proof}
See Appendix\,\ref{app:2}.
\end{proof}

Lemma\,\ref{lemma:euler_approx_base_equations} gives an implicit formula for the approximated solution of SDEs in $\mathcal{C}^R$. However, the autoregressive behaviour of the MCARMA process is described by SDEs in $\mathcal{C}^{AR}$. This information has to be added to the formula in Lemma \ref{lemma:euler_approx_base_equations}, to find the final approximated solution of the MCARMA process. Theorem\,\ref{thm:the_transformation_relation} states a formula for the $p$ times lagged variable of dimension $k\in\{1,\ldots ,d\}$. \\

\begin{theorem}
\label{thm:the_transformation_relation}
The multivariate ARMA/CARMA transformation relation is given by
\begin{align*}
    x_k(t+ph) = &x_k(t+(p-1)h)  -h^p\sum_{s=1}^{d}\alpha_{(k-1)d+s}^{(p)}x_{s}(t)+ h^{p-1}\sum_{r=1}^{m}\beta_{(k-1)m+r}^{(0)}\Delta L_r(t) \\
    &-\sum_{l=2}^{p}\sum_{s=1}^{d}\alpha_{(k-1)d+s}^{(p-l+1)}\Bigg(h^{p-l+1}\sum_{n=0}^{l-1}(-1)^n b_n^{l-1} x_s(t + (l-1-n)h)  \\ 
    &\qquad\qquad\qquad\qquad - \sum_{w=0}^{l-2}h^{p-w-1}\sum_{r=1}^{m}\beta_{(s-1)m+r}^{(p-l+w+1)}\sum_{v=0}^{w}(-1)^vb_v^w\Delta L_r(t+(w-v)h)\Bigg) \\
    &-\sum_{n=1}^{p-1}(-1)^n b_n^{p-1} \left(x_k(t + (p-n)h) - x_k(t + (p-1-n)h)\right) \\
    & + \sum_{w=0}^{p-2}h^{p-w-2}\sum_{r=1}^{m}\beta_{(k-1)m+r}^{(w+1)}\sum_{v=0}^{w}(-1)^vb_v^w\left(\Delta L_r(t+(w-v+1)h) - \Delta L_r(t+(w-v)h)\right),
\end{align*}
where $p$ is the number of lags, $d$ is the total number of dimensions, $k\in\{1,\ldots ,d\}$ is the dimension of interest, $m$ is the number of independent driving L{\'e}vy processes, $b_n^{i}$ is defined recursively from Eq.\,\eqref{eq:pascal_triangle} (Appendix\,\ref{app:2}), and $h$ is the Euler discretization step size.
\end{theorem}
\begin{proof}
See Appendix\,\ref{app:3}.
\end{proof}
Notice that the multivariate ARMA/CARMA transformation relation expresses the $p$ times lagged variable of dimension $k\in\{1,\ldots ,d\}$ as a linear combination of $0$ to $p-1$ lagged variables of all model dimensions and driving L{\'e}vy processes. That is, the approximated solution of the MCARMA process is represented by a VARMA process. \\


\subsection{\label{subsect:convergence_rates}An analysis of convergence rates}
A multivariate ARMA/CARMA transformation relation was derived in the previous section through an Euler discretization of the L{\'evy}-driven MCARMA process. For any practical application, it is important that the Euler scheme converges. An Euler scheme convergence rate for jump diffusions with jumps of finite variance and infinite variations is derived in this section. In Section\,\ref{subsect:nig_levy_process} we will see that the result holds for an NIG-L{\'e}vy-driven MCARMA process. Note that vector processes and deterministic vector functions are written without bold font in this section. \\

Consider the jump diffusion 
\begin{align}
    \label{eq:multidimensional_sde_with_small_jumps}
    Z(t) =& Z(0) + \int_{0}^{t}\tilde{a}(s,Z(s))ds +\int_0^t b(s,Z(s))dW(s)+ \int_0^t\int_{\R^m\backslash\{0\}}\gamma(s-,Z(s-),z)\tilde{N}(ds,dz),
\end{align}
where $\tilde{a}(t,x)=\left(a(t,x)+\int_{\abs{z}\geq \epsilon}\gamma(t,x,z)\nu(dz)\right)$. We assume that the usual integrability conditions of coefficients are satisfied, and that $W(t)$ and $N(t,U)$ are independent stochastic processes. Further, we assume that $\gamma(t,x,z)=g(z)\eta(t,x)$, where $t\to\eta(t,x)$ is c{\`a}dl{\`a}g, and that the finite variance condition
\begin{align*}
    G^2(\infty) = \int_{\R^m\backslash \{0\}}g^2(z)\nu(dz)<\infty,
\end{align*}
is satisfied. \\

Convergence rates of jump diffusions depend on the behaviour of the L{\'e}vy measure at origin, see, e.g., \cite{kuhn19}. A complete discussion on the Euler scheme (among others) and convergence rates for L{\'e}vy processes with L{\'e}vy measure of finite total mass, $\nu(U)<\infty$, is found in \cite{platen10}. Note that this corresponds to L{\'e}vy processes where the jump part is a compound Poisson process. For L{\'e}vy processes with L{\'e}vy measure of infinite total mass, $\nu(U)=\infty$, the question of Euler scheme convergence rates is more intricate. In \cite{asmussen01}, the cases $\int_{\abs{z}\leq 1}\abs{z}\nu(dz)<\infty$ and $\int_{\abs{z}\leq 1}\abs{z}\nu(dz)=\infty$ are discussed. In the following, inspired by work in \cite{platen10}, \cite{asmussen01} and \cite{benth11}, an Euler scheme convergence rate for jump diffusions of infinite variations, as given in Eq.\,\eqref{eq:multidimensional_sde_with_small_jumps}, is derived for the case $\nu(U)=\infty$. \\

In cases where the jump part of $Z(t)$ is of finite variation, convergence rates might be obtained by replacing the small jumps by their expected value, or by simply removing the small jumps. In the infinite variation case, removing the small jumps would not be appropriate, as the small jumps dominate in that case. Define the function 
\begin{align}
    \label{eq:G_squared_difinition}
    G^2(\epsilon) \triangleq \int_{\abs{z}<\epsilon} g^2(z)\nu(dz), 
\end{align}
with natural property $\lim_{\epsilon \to 0}G^2(\epsilon)= 0$. We propose, as in \cite{benth11}, to approximate the small jumps part of Eq.\,\eqref{eq:multidimensional_sde_with_small_jumps} as 
\begin{align}
    \label{eq:small_jumps_approximation}
    \int_0^t\int_{\abs{z}<\epsilon}g(z)\eta(s-,Z(s-))\tilde{N}(ds,dz) \simeq \int_0^t G(\epsilon)\eta(s,Z_\epsilon(s))dB(s) \quad\leftrightarrow\quad S(t)\simeq S_\epsilon(t),
\end{align}
where $B(t)$ is a Brownian motion process independent of $W(t)$ and $N(t,U)$. That is, $Z(t)$ in Eq.\,\eqref{eq:multidimensional_sde_with_small_jumps} is approximated by the process
\begin{align}
    \label{eq:multidimensional_sde_with_small_jumps_approximation}
    \begin{split}
    Z_\epsilon(t) =& Z(0) +\int_{0}^{t}\tilde{a}(s,Z_\epsilon(s))ds +\int_0^t b(s,Z_\epsilon(s))dW(s)+\int_{0}^{t}G(\epsilon)\eta(s,Z_\epsilon(s))dB(s)\\
    &\quad + \int_0^t\int_{\abs{z}\geq \epsilon}g(z)\eta(s-,Z_\epsilon(s-))\tilde{N}(ds,dz).
    \end{split}
\end{align}
Notice that the jump part of $Z_\epsilon(t)$ is represented by a general (compensated) compound Poisson process. Before stating the convergence rate for the approximation $Z_\epsilon(t)$, we state an intermediate result. The following assumption, inspired by \cite{platen10}, holds throughout this work. \\

\begin{assumption}
\label{ass:assumption1}
    Assume that the jump diffusions $Z(t)$ and $Z_\epsilon(t)$, in Eq.\,\eqref{eq:multidimensional_sde_with_small_jumps} and \eqref{eq:multidimensional_sde_with_small_jumps_approximation} respectively, satisfy 
\begin{align*}
    E(\abs{Z(0)}^2) = E(\abs{Z_\epsilon(0)}^2)<\infty \quad\text{and}\quad E(\abs{Z_\epsilon(0)-Z_\epsilon^h(0)}^2)\leq K_0h,
\end{align*}
where $Z_\epsilon^h(0)$ represents the initial condition of the Euler discretization of $Z_\epsilon(t)$. Further, for $t\in[0,T]$ and $x,y\in\R^m$, the drift and diffusion coefficients, as well as $\eta(t,x)$ in the separable jump coefficient $\gamma(t,x,z)=g(z)\eta(t,x)$, satisfy the Lipschitz and linear growth conditions
\begin{align}
    \label{eq:lipschitz_linear_growth_conditions}
    \abs{f(t,x)-f(t,y)}\leq K_1\abs{x-y}\quad \text{and}\quad \abs{f(t,x)}^2\leq K_2(1+\abs{x}^2).
\end{align}
Define $K=\max(K_1,K_2)$, which is used in the following proofs without notice.
\end{assumption}

\begin{lemma}
\label{lemma:error_convergence}
    Let $S(t)$ and $S_\epsilon(t)$ be given as in Eq.\,\eqref{eq:small_jumps_approximation}, and denote the process $S_\epsilon(t)$ by $S_{\epsilon,Z(t)}(t)$, when it is state-dependent on $Z(t)$ rather than $Z_{\epsilon}(t)$ (see Eq.\,\eqref{eq:multidimensional_sde_with_small_jumps} and \eqref{eq:multidimensional_sde_with_small_jumps_approximation}). The convergence rate of $S_{\epsilon,Z(t)}(t)$ is given by
    \begin{align*}
        \|\sup_{0\leq t\leq T}\abs{S(t) - S_{\epsilon,Z(t)}(t)}\|_2 \leq C^{1/2}G(\epsilon),
    \end{align*}
    where $C=8KT(1+C_1e^{C_1})$, $C_1= 2K(2T + 2 +G^2(\infty))/T$ and $G(\epsilon)$ given in Eq.\,\eqref{eq:G_squared_difinition}. 
\end{lemma}

\begin{proof}
Define $f\triangleq f(t,Z(t))$. First, we show a boundedness result of $Z(t)$. By the triangle inequality, the identity $\abs{f+g}^p\leq 2^{p-1}(\abs{f}^p + \abs{g}^p)$, Cauchy-Schwarz inequality, It{\^o} isometry, monotone convergence and linear growth, we have
\begin{align}
    \label{eq:norm_of_Z_in_proof}
    \begin{split}
    \|\sup_{0\leq t\leq T}\abs{Z(t)}\|_2^2 &= E\Bigg[\abs{\sup_{0\leq t \leq T}\abs{\int_{0}^{t}\tilde{a}ds + \int_{0}^{t}b dW(s) +\int_0^t\int_{\R^m\backslash \{0\}}g(z)\eta\tilde{N}(ds,dz)}}^2\Bigg] \\
    &\leq E\Bigg[\abs{\sup_{0\leq t \leq T}\left(\abs{\int_{0}^{t}\tilde{a} ds} + \abs{\int_{0}^{t}b dW(s)} +\abs{\int_0^t\int_{\R^m\backslash \{0\}}g(z)\eta\tilde{N}(ds,dz)}\right)}^2\Bigg] \\
    &\leq E\Bigg[\sup_{0\leq t \leq T}\left(4\abs{\int_{0}^{t}\tilde{a} ds}^2 + 4\abs{\int_{0}^{t}b dW(s)}^2 + 2\abs{\int_0^t\int_{\R^m\backslash \{0\}}g(z)\eta\tilde{N}(ds,dz)}^2\right)\Bigg] \\
    &\leq E\Bigg[4T\int_{0}^{T}\tilde{a}^2 ds + 4\int_{0}^{T}b^2 ds + 2G^2(\infty)\int_0^T\eta^2 ds\Bigg] \\
    &\leq 2(2T + 2 +G^2(\infty))\int_{0}^{T}E\left[K(1+\abs{Z(s)}^2)\right] ds\\
    &\leq 2K(2T + 2 +G^2(\infty))\left(T + \int_{0}^{T}E\left[\sup_{0\leq t\leq s}\abs{Z(t)}^2\right] ds\right) \\
    &= C_1 + \frac{C_1}{T}\int_{0}^{T}\|\sup_{0\leq t\leq s}\abs{Z(t)}\|_2^2 ds,
    \end{split}
\end{align}
where $C_1= 2K(2T + 2 +G^2(\infty))/T$. Define $F(u)\triangleq \|\sup_{0\leq t\leq u}\abs{Z(t)}\|_2^2$. Then, by Gr{\"o}nwall's inequality, we find that 
\begin{align*}
    F(T) \leq C_1 + \frac{C_1}{T}\int_0^T F(s)ds \quad\leftrightarrow\quad F(T)\leq C_1e^{C_1}.
\end{align*}
Now, continue to prove the stated convergence result. By Doob's maximal inequality, the independence property of processes, the expectation rule of It{\^o} integrals, It{\^o} isometry, monotone convergence and linear growth, we have that
\begin{align*}
    \|\sup_{0\leq t\leq T}\abs{S(t) - S_{\epsilon,Z(t)}(t)}\|_2^2 &\leq 4 E\left[\left(\int_0^T\int_{\abs{z}<\epsilon}g(z)\eta\tilde{N}(ds,dz) - \int_0^T G(\epsilon)\eta dB(s) \right)^2\right]\\
    &=4 E\left[\left(\int_0^T\int_{\abs{z}<\epsilon}g(z)\eta\tilde{N}(ds,dz)\right)^2 + \left(\int_0^T G(\epsilon)\eta dB(s) \right)^2\right]\\
    &= 4 E\left[\int_0^T G^2(\epsilon)\eta^2 ds  + \int_0^T G^2(\epsilon)\eta^2 ds \right] \\
    & =8G^2(\epsilon)\int_0^T E\left[\eta^2 \right]ds \leq 8G^2(\epsilon)\int_0^T E\left[ K(1+\abs{Z(s)}^2) \right]ds \\
    & \leq 8G^2(\epsilon)K\left(T + \int_0^T E\left[ \sup_{0\leq t\leq s} \abs{Z(t)}^2 \right]ds \right) \\
    & = 8G^2(\epsilon)K \left(T + \int_0^T \|\sup_{0\leq t\leq s}\abs{Z(t)}\|_2^2 ds\right).
\end{align*}
Since $\|\sup_{0\leq t\leq s}\abs{Z(t)}\|_2^2 \leq \|\sup_{0\leq t\leq T}\abs{Z(t)}\|_2^2\leq C_1e^{C_1}$, we have that
\begin{align*}
    \|\sup_{0\leq t\leq T}\abs{S(t) - S_{\epsilon,Z(t)}(t)}\|_2^2 &\leq 8G^2(\epsilon)K \left(T + \int_0^T C_1e^{C_1} ds\right),
\end{align*}
and the proof is complete.
\end{proof}
Note that $G(\epsilon)$ is finite by definition, and that the bound in Lemma\,\ref{lemma:error_convergence} converges to zero as $\epsilon\to 0$. \\

The convergence rate of the approximated jump diffusion $Z_\epsilon(t)$ (Eq.\,\eqref{eq:multidimensional_sde_with_small_jumps_approximation}) is derived in the following proposition. \\

\begin{proposition}
    \label{prop:approx_nig-levy_process_convergence_rate}
Let $Z(t)$ and $Z_\epsilon(t)$ be as given in Eq.\,\eqref{eq:multidimensional_sde_with_small_jumps} and \eqref{eq:multidimensional_sde_with_small_jumps_approximation} respectively. The convergence rate of $Z_\epsilon(t)$ is given by
\begin{align*}
    \|\sup_{0\leq t \leq T}\abs{Z(t)-Z_\epsilon(t)}\|_2 \leq C_\epsilon^{1/2}G(\epsilon),
\end{align*}
 where $C_\epsilon=8Ce^{C_2T}$, $C_2= 4K(T+1+G^2(\infty)+G^2(\epsilon))$, and $C$ and $G(\epsilon)$ is given in Lemma\,\ref{lemma:error_convergence} and Eq.\,\eqref{eq:G_squared_difinition} respectively.
\end{proposition}

\begin{proof}
Define $f\triangleq f(t,Z(t))$ and $f_\epsilon \triangleq f(t,Z_\epsilon(t))$. As in the proof of Lemma\,\ref{lemma:error_convergence} (Eq.\,\eqref{eq:norm_of_Z_in_proof}), use the triangle inequality, the identity $\abs{f+g}^p\leq 2^{p-1}(\abs{f}^p + \abs{g}^p)$, Cauchy-Schwarz inequality, It{\^o} isometry and monotone convergence to obtain
\begin{align*}
    &\|\sup_{0\leq t \leq T}\abs{Z(t)-Z_\epsilon(t)}\|_2^2 \\
    &= E\Bigg[\abs{\sup_{0\leq t \leq T}\abs{\int_{0}^{t}(\tilde{a}-\tilde{a}_\epsilon) ds + \int_{0}^{t}(b-b_\epsilon ) dW(s) + \int_0^t\int_{\abs{z}\geq\epsilon}g(z)(\eta-\eta_\epsilon)\tilde{N}(ds,dz)\\
    & \quad\quad\quad  + \int_0^t G(\epsilon)(\eta -\eta_\epsilon)dB(s) + \int_0^t\int_{\abs{z}<\epsilon}g(z)\eta \tilde{N}(ds,dz) - \int_0^t G(\epsilon)\eta dB(s)}}^2\Bigg] \\
    &\leq 4T\int_0^TE\left[\abs{\tilde{a}-\tilde{a}_\epsilon}^2\right]ds + 4\int_0^TE\left[\abs{b-b_\epsilon}^2\right]ds + 4(G^2(\infty)-G^2(\epsilon))\int_0^TE\left[\abs{\eta-\eta_\epsilon}^2\right]ds \\
    &\quad +8G^2(\epsilon)\int_0^TE\left[\abs{\eta-\eta_\epsilon}^2\right]ds + 8E\left[\sup_{0\leq t \leq T}\abs{S(t)-S_{\epsilon ,Z(t)}(t)}^2\right].
\end{align*}
By Lipschitz continuity and Lemma\,\ref{lemma:error_convergence} we further find
\begin{align*}
    &\|\sup_{0\leq t \leq T}\abs{Z(t)-Z_\epsilon(t)}\|_2^2 \\
    &\leq 8CG^2(\epsilon) + 4K(T+1+G^2(\infty)+G^2(\epsilon))\int_0^TE\left[\abs{Z(s)-Z_\epsilon(s)}^2\right]ds\\
    &= 8CG^2(\epsilon) + C_2\int_0^T\|\sup_{0\leq t \leq s}\abs{Z(t)-Z_\epsilon(t)}\|_2^2 ds,
\end{align*}
where $C_2= 4K(T+1+G^2(\infty)+G^2(\epsilon))$. Define $F(u)\triangleq \|\sup_{0\leq t\leq u}\abs{Z(t)}\|_2^2$. Then, by Gr{\"o}nwall's inequality, we have 
\begin{align*}
    F(T) \leq 8CG^2(\epsilon) + C_2\int_0^T F(s)ds \quad\leftrightarrow\quad F(T)\leq 8CG^2(\epsilon)e^{C_{2}T}.
\end{align*}
This concludes the proof.
\end{proof}
As for the convergence rate in Lemma\,\ref{lemma:error_convergence}, the convergence rate of $Z_\epsilon(t)$ converges to zero as $\epsilon\to 0$. \\

Recall that the aim of this section is to derive an Euler scheme convergence rate for the jump diffusion $Z(t)$ in Eq.\,\eqref{eq:multidimensional_sde_with_small_jumps} with infinite variations and $\nu(U)=\infty$. The convergence rate of the Euler discretization of the approximated process $Z_\epsilon(t)$ is therefore derived next. As we will see in Section\,\ref{sect:the_mcar_model}, this convergence result provides a foundation to estimate MCARMA models driven by NIG-L{\'e}vy processes. \\

\begin{proposition}
    \label{prop:euler_scheme_of_approximated_sde_convergence_rate}
Let $Z(t)$ and $Z_\epsilon(t)$ be as given in Eq.\,\eqref{eq:multidimensional_sde_with_small_jumps} and \eqref{eq:multidimensional_sde_with_small_jumps_approximation} respectively, and let $Z_\epsilon^h(t)\triangleq\{Z_\epsilon^h(t)\}_{t\in[0,T]}$ denote the Euler scheme of $Z_\epsilon(t)$ with equidistant time discretization step size $h\in(0,1]$. The convergence rate of $Z_\epsilon^h(t)$ is given by 
\begin{align*}
    \|\sup_{0\leq t\leq T}\abs{Z(t)-Z_\epsilon^h(t)}\|_2 \leq \sqrt{2}\left(C_\epsilon G^2(\epsilon)+K_3h\right)^{1/2},
\end{align*}
where $K_3$ is a finite positive constant, and $C_\epsilon$ and $G(\epsilon)$ is given in Proposition\,\ref{prop:approx_nig-levy_process_convergence_rate} and Eq.\,\eqref{eq:G_squared_difinition} respectively.
\end{proposition}

\begin{proof}
We assume that $\tilde{a}(t,x)$ (and therefore also $\tilde{a}_\epsilon(t,x)$), see Eq.\,\eqref{eq:multidimensional_sde_with_small_jumps}, satisfy the Lipschitz and linear growth conditions in Eq.\,\eqref{eq:lipschitz_linear_growth_conditions}. Then, since all conditions in Assumption\,\ref{ass:assumption1} are satisfied, Corollary\,6.4.3 in \cite{platen10} (see also \cite{gardon04}) gives
\begin{align}
    \label{eq:euler_convergence_rate_Z_epsilon}
    \|\sup_{0\leq t \leq T}\abs{Z_\epsilon(t)-Z_\epsilon^{h}(t)}\|_2^2 \leq K_3 h,
\end{align}
where $K_3$ is a finite positive constant independent of $h$. Note that Corollary\,6.4.3 in \cite{platen10} holds when the jump part of $Z_\epsilon(t)$ is a compound Poisson process, which is the case for $Z(t)$ and $Z_\epsilon(t)$ when the drift is $\tilde{a}(t,x)$ and $\tilde{a}_\epsilon(t,x)$ respectively. Using the identity $\abs{f+g}^p\leq 2^{p-1}(\abs{f}^p + \abs{g}^p)$, as in the proof of Lemma\,\ref{lemma:error_convergence}, we find
\begin{align*}
    \|\sup_{0\leq t \leq T}\abs{Z(t)-Z_\epsilon^{h}(t)}\|_2^2 &= \|\sup_{0\leq t \leq T}\abs{Z(t) - Z_\epsilon(t) + Z_\epsilon(t) - Z_\epsilon^{h}(t)}\|_2^2 \\
    &\leq 2\|\sup_{0\leq t \leq T}\abs{Z(t)-Z_\epsilon(t)}\|_2^2 + 2\|\sup_{0\leq t \leq T}\abs{Z_\epsilon(t)-Z_\epsilon^{h}(t)}\|_2^2.
\end{align*}
By Proposition\,\ref{prop:approx_nig-levy_process_convergence_rate} and Eq.\,\eqref{eq:euler_convergence_rate_Z_epsilon} the proof is complete.
\end{proof}
Note that the Euler scheme $Z_\epsilon^h(t)$ converges to $Z(t)$ as $\epsilon\to 0$ and $h\to 0$. \\

It is straight forward to show that the results in Lemma\,\ref{lemma:error_convergence}, Proposition\,\ref{prop:approx_nig-levy_process_convergence_rate} and Proposition\,\ref{prop:euler_scheme_of_approximated_sde_convergence_rate} hold for MCARMA processes. A special case is considered in Section\,\ref{subsect:nig_levy_process}. \\ 


\section{\label{sect:the_mcar_model}The NIG-L\'{e}vy-driven MCAR process: A case study}


The remaining of this work focus on a case study for an NIG-L{\'e}vy-driven MCAR process, as an example of how to apply the multivariate ARMA/CARMA transformation relation in model estimation. This MCAR modelling framework is fit to a two-dimensional system of stratospheric temperature and wind variables. More specifically, the dynamical system of dependent variables that is to be considered is one of stratospheric temperature and U wind. As opposed to temperature, wind constitutes a direction in addition to its magnitude. The wind direction is given by two wind components, that is U wind representing the west to east flow, and V wind representing the south to north flow. As explained in, e.g., \cite{karpechko16} and \cite{hitchcock14}, it is particularly interesting to study the U wind component of the Northern hemisphere, as extreme events such as sudden stratospheric warmings has the potential to influence surface weather to a great extent. \\

The initial stratospheric temperature and U wind data are retrieved as European Centre for Medium-Range Weather Forecasts (ECMWF) ERA-Interim atmospheric reanalysis model products (\cite{ecmwf_stratospheric_temp} and \cite{dee2011}), see Appendix\,\ref{app:6}. The data analysed in this work are reprocessed as daily circumpolar mean stratospheric temperature and U wind at $60^{\circ}$N and $10$\;hPa altitude, from 1 January 1979 to 31 December 2018, see Table\,\ref{tab:dataset_specs}. See \cite{eggen2021} for more information about data preparation. \\

In Section\,\ref{subsect:nig_levy_process}, we define the model representation and show that the convergence results in Section\,\ref{subsect:convergence_rates} hold, meaning that the multivariate ARMA/CARMA transformation relation holds for this special case. Further, in Section\,\ref{subsect:distributional_properties}, some statistical considerations of the NIG-L{\'e}vy process are assessed to properly express the error (moving average) coefficients of the NIG-L{\'e}vy-driven MCAR process. Finally, in Section\,\ref{subsect:example},
an extended MCAR process able to describe additive seasonality and heteroscedasticity in dynamical systems is defined. Then, based on the multivariate ARMA/CARMA transformation relation, explicit formulas for estimated autoregressive and error coefficients are derived for a given set of model parameters.
The results are used to fit a two-dimensional NIG-L{\'e}vy-driven MCAR model to stratospheric temperature and U wind data. \\


\subsection{\label{subsect:nig_levy_process}Some considerations on the NIG-L{\'e}vy-driven MCAR(MA) process}


In this section we explicitly define the MCAR process as a state space representation, and show that the convergence results in Section\,\ref{subsect:convergence_rates} hold for an NIG-L{\'e}vy-driven MCAR(MA) process. This means that the multivariate ARMA/CARMA transformation relation holds for the NIG-L{\'e}vy-driven MCAR model. \\

By the state space representation of L{\'e}vy-driven MCARMA processes in Section\,\ref{subsec:varma_mcarma_definition}, the corresponding MCAR process is given by 
\begin{align}
    \label{eq:mcar-model_initial}
    \begin{cases}
    \bs{Y}(t) = \bs{X}_1(t) \\
    d\bs{X}_{l}(t) = \mathbb{1}_d\bs{X}_{l+1}(t)dt \\
    d\bs{X}_{p}(t) = [-A_p\cdots -A_1]\bs{X}(t)dt + B_0 d\bs{L}(t),
    \end{cases}
\end{align}
see Eq.\,\eqref{eq:the_mcarma_model}, \eqref{eq:sde_recursive_equations} and \eqref{eq:sde_lag_dependence_structure}. The multivariate ARMA/CARMA transformation relation in Theorem\,\ref{thm:the_transformation_relation} is considerably simplified for MCAR processes, as the shifted L\'{e}vy-terms disappear. For any practical application, we assume the simpler MCAR($p$) process to be a useful approximation of the MCARMA($p,q$) process. That is, as stated in \cite{gomez19}, for $p$ large enough, the VAR($p$) model approaches a VARMA($p,q$) model. A rigorous proof of this statement for MCAR/MCARMA processes is a topic for further research. The argument for assuming this to hold for MCAR/MCARMA processes as well, is that the direct transformation relation between VARMA/MCARMA processes (see Theorem\,\ref{thm:the_transformation_relation}) is derived by introducing no structural changes going from the discretized MCARMA process to the VARMA process. In particular, we will see in Section\,\ref{subsect:example} that an MCAR model explains well the two dimensional dynamical system of stratospheric temperature and U wind. \\

We will continue to specify the MCAR process in Eq.\,\eqref{eq:mcar-model_initial} further. That is, we will consider the framework when driven by an NIG-L{\'e}vy process. As shown in \cite{barndorff97-2}, the generating triplet of an NIG-L{\'e}vy process is $(\bs{\alpha},0,\nu)$. As the property $\int_{\abs{\bs{z}}\leq 1}\abs{\bs{z}}\nu(dz)=\infty$ holds for this process the small jumps dominate, and the results in Section\,\ref{subsect:convergence_rates} apply. By properties of the L{\'e}vy measure, $\nu((-\epsilon,\epsilon)^c)<\infty$ for all $0<\epsilon \leq 1$, the NIG-L{\'e}vy process might be represented as (see Eq.\,\eqref{eq:ito-levy_decomposition})
\begin{align}
    \label{eq:the_nig-levy_process}
    \bs{L}(t) = \hat{\bs{\alpha}}t + \int_{\abs{\bs{z}}<\epsilon}\bs{z}\tilde{N}(t,d\bs{z}) + \int_{\abs{\bs{z}}\geq\epsilon}\bs{z}N(t,d\bs{z}),
\end{align}
where $\hat{\bs{\alpha}}=\left(\bs{\alpha} - \int_{\epsilon\leq \abs{\bs{z}}< 1}\bs{z}\nu(d\bs{z})\right)$. Further, the MCARMA state space model in Eq.\,\eqref{eq:the_mcarma_model} written terms of the NIG-L{\'e}vy process is given by
\begin{align}
    \label{eq:the_nig-levy_mcarma_process}
    d\bs{X}(t) = \left(A\bs{X}(t) + \beta \hat{\bs{\alpha}}\right)dt + \beta\int_{\abs{\bs{z}}< \epsilon}\bs{z}\tilde{N}(dt,d\bs{z}) + \beta\int_{\abs{\bs{z}}\geq \epsilon}\bs{z}N(dt,d\bs{z}),
\end{align}
which is referred to as the NIG-L{\'e}vy-driven MCARMA process. By Proposition\,\ref{prop:approx_nig-levy_process_convergence_rate} and Proposition\,\ref{prop:euler_scheme_of_approximated_sde_convergence_rate}, it is straight forward to see that $\bs{X}_\epsilon(t)$ (corresponding to $Z_\epsilon(t)$ in Eq.\,\eqref{eq:multidimensional_sde_with_small_jumps_approximation}) converges to $\bs{X}(t)$ when the small jumps approach zero, and that the corresponding Euler scheme $\bs{X}_\epsilon^h(t)$ converges to $\bs{X}(t)$ as $h\to 0$, when all relevant coefficients satisfy Assumption\,\ref{ass:assumption1}. That is, the transformation relation in Theorem\,\ref{thm:the_transformation_relation} is valid for the NIG-L{\'e}vy-driven MCARMA process, as consequently for the NIG-L{\'e}vy-driven MCAR process. \\


\subsection{\label{subsect:distributional_properties}Distributional properties of the NIG-L{\'e}vy process}


The distributional properties of linear combinations of NIG-L{\'e}vy processes are considered in this section. The results will be used in Section\,\ref{subsect:example} to estimate an NIG-L{\'e}vy-driven MCAR model describing stratospheric temperature and U wind dynamics. Note that an extended analysis including lags of the driving NIG-L{\'e}vy process would have to be performed in order to consider estimation of a corresponding MCARMA model. \\

By definition of the MCARMA process in Eq.\,\eqref{eq:the_mcarma_model}, the moving average matrices $\beta_{p-1},\ldots ,\beta_1$ of the MCAR process are given by $\mathbb{0}_m$. That is, the transformation relation in Theorem\,\ref{thm:the_transformation_relation} gives a component-wise stochastic part
\begin{align}
    \label{eq:stoch_part_of_mcar_model}
    \Delta\mathcal{E}_k(t)\triangleq stoch(x_k(t+p)) = \sum_{r=1}^{m}\beta_{r,k}\Delta L_r(t),
\end{align}
where $\beta_{i_r,k}\triangleq \beta_{(k-1)m+i_r}^{(0)}$, and the discretization step size is assumed to be $h=1$ (day). Notice that this indicates that the driving process of the multivariate dynamical system (the MCAR process) is given by $d\bs{\mathcal{E}}(t)=\beta d\bs{L}(t)$ ($\beta\triangleq \beta_p$), with $\bs{\mathcal{E}}(t)$ being a $d$-dimensional L{\'e}vy process. \\

Let the $1$-day increment of each component of the multidimensional L\'{e}vy process, $\bs{L}(t)$, be distributed as univariate NIG random variables, meaning
\begin{align}
    \label{eq:levy_is_nig_distributed}
    \Delta L_{i_r}(t) \overset{d}{\simeq} L_{i_r}(1) \overset{d}{\simeq} NIG(a_{i_r},b_{i_r},\delta_{i_r},\mu_{i_r}),
\end{align}
where $a_{i_r}$, $b_{i_r}$, $\delta_{i_r}$ and $\mu_{i_r}$ is the tail heaviness, asymmetry parameter, scale parameter and distributional location respectively. The NIG distribution is closed under convolution and affine transformations in the following sense (see, e.g., \cite{barndorff-nielsen_book_01}): if $\Delta L_{i_r}$, for $1\leq {i_r} \leq m$, are independent random variables that are NIG distributed such that $a_{i_r}=a$ and $b_{i_r} = b$ for all ${i_r}$, then 
\begin{align*}
    \sum_{r=1}^{m}\Delta L_r(t)\overset{d}{\simeq} NIG\left(a ,b ,\sum_{r=1}^{m}\delta_r, \sum_{r=1}^{m}\mu_r\right)
\end{align*}    
and
\begin{align*}
\beta_{i_r,k}\Delta L_{i_r}(t)\overset{d}{\simeq} NIG\left(\frac{a_{i_r}}{\abs{\beta_{{i_r},k}}},\frac{b_{i_r}}{\beta_{{i_r},k}},\abs{\beta_{{i_r},k}}\delta_{{i_r}},\beta_{{i_r},k}\mu_{{i_r}}\right).
\end{align*}
By these properties, the component-wise distribution generated by the L{\'e}vy process, $\bs{\mathcal{E}}(t)$, is given by 
\begin{align}
 \label{eq:nig_distribution_of_sum}
     \mathcal{E}_k(t) \overset{d}{\simeq} NIG\left(a,b,\sum_{r=1}^{m}\abs{\beta_{r,k}}\delta_{r},\sum_{r=1}^{m}\beta_{r,k}\mu_{r}\right),
\end{align}
with restrictions $a = a_1/\abs{\beta_{1,k}} = \ldots = a_m/\abs{\beta_{m,k}}$ and $b = b_1/\beta_{1,k} = \ldots = b_m/\beta_{m,k}$, for all $1\leq k \leq d$. Note that each independent NIG-L\'{e}vy process component, $dL_{i_r}(t)$, might generate random variables from NIG distributions with distinct parameters $a_i$ and $b_i$, where the components of $\beta$ have to be restricted for Eq.\,\eqref{eq:nig_distribution_of_sum} to hold. \\


\subsection{\label{subsect:example}A two-dimensional MCAR($4$) case study}


Estimation of a two-dimensional MCAR($4$) model driven by a two-dimensional NIG-L{\'e}vy process is performed in this section. First, we will modify the MCAR process, defined in Eq.\,\eqref{eq:mcar-model_initial}, such that data with inherent seasonal behaviour can be represented by the model. Then the transformation relation in Theorem\,\ref{thm:the_transformation_relation} is used to explicitly state the autoregressive MCAR model coefficients as a function of autoregressive VAR model coefficients, and restrictions are set on the NIG distributions representing model residuals, such that formulas for the MCAR model coefficients of the stochastic part can be derived. Finally, empirical estimation of the modified MCAR model is exemplified using daily circumpolar mean stratospheric temperature and U wind data, see Table\,\ref{tab:dataset_specs} for data specifications. \\

In \cite{eggen2021}, daily circumpolar mean stratospheric temperature at $60^{\circ}$\;N and $10$\;hPa altitude were shown to follow a CAR($4$) process on the form
\begin{align}
    \label{eq:car-model}
    \begin{cases}
    S(t) = \Lambda(t) + X_1(t)\\
    d\bs{X}(t) = A_p\bs{X}(t)dt + \bs{e}_p\sigma(t-)dL(t),
    \end{cases}
\end{align}
with statistical significance. Here, $\Lambda(t)\in\R$ is a deterministic, bounded and continuously differentiable seasonality function, $\sigma(t)$ is a deterministic and c{\`a}dl{\`a}g volatility function, and the stochastic process $\bs{X}(t)$ in $\R^p$ represents deseasonalized stratospheric temperature and its autoregressive nature. As indicated, $\bs{X}(t)$
is given by a multidimensional non-Gaussian OU process with time dependent volatility, where $A_p\in \R^p$, $\bs{e}_p$ is the unit vector in $\R^p$, and the driving L{\'e}vy process takes values in $\R$. This model is also used for modelling of temperature and wind in the troposphere. For more details about this model, see, e.g., \cite{eggen2021}, \cite{benth2008}, \cite{benth13}. \\

To represent weather variables independently with an one-dimensional model is a significant simplification, as weather is a complex and non-linear system of dependent variables. The MCAR process allows the improvement of representing several weather variables as a linear cross-correlated system in an autoregressive manner. Using a similar methodology as in \cite{eggen2021}, we see that daily zonal mean stratospheric U wind, at $10$\;hPa altitude, follows a CAR($4$) process with statistical significance, just as stratospheric temperature. This motivates the case study of fitting a two-dimensional MCAR($4$) process to stratospheric temperature and U wind. Note that the exercise of concluding an optimal value of the lag parameter, $p$, for the MCAR process is not part of this work. The choice of using $p=4$ is made because the marginals are well modelled in that case. \\ 

Inspired by the CAR process in Eq.\,\eqref{eq:car-model}, we redefine the MCAR model in Eq.\,\eqref{eq:mcar-model_initial} as 
\begin{align}
    \label{eq:mcar-model_updated}
    \begin{cases}
    \bs{Y}(t) = \bs{\Lambda}(t) + \bs{X}_1(t) \\
    d\bs{X}_{l}(t) = \mathbb{1}_d\bs{X}_{l+1}(t)dt \\
    d\bs{X}_{p}(t) = [-A_p\cdots -A_1]\bs{X}(t)dt + \bs{\sigma}(t-)\circ B_0 d\bs{L}(t),
    \end{cases}
\end{align}
where $\bs{\Lambda}(t)\in{\R^d}$ is a deterministic and continuous seasonality function, $\bs{\sigma}(t)\in{\R^d}$ is a deterministic and c{\`a}dl{\`a}g volatility function, and $\circ$ represents the Hadamard product. This allows modelling data with seasonal and heteroscedastic behaviour. Note that this extension of the MCAR model framework still might be represented as in Eq.\,\eqref{eq:the_nig-levy_mcarma_process}, meaning that the transformation relation in Theorem\,\ref{thm:the_transformation_relation} is valid as long as the relevant model coefficients satisfy Assumption\,\ref{ass:assumption1}. Considerations of possible theoretical challenges regarding the spectral method of deriving the MCARMA model, see \cite{marquardt07}, is beyond the scope of this paper. In what follows, the two-dimensional NIG-L{\'e}vy-driven MCAR($4$) process in Eq.\,\eqref{eq:mcar-model_updated} is fitted to stratospheric temperature and U wind data. Specifications of the data are given in Table\,\ref{tab:dataset_specs}. \\


\subsubsection{\label{subsubsect:parameter_est}The model coefficients}


In this section, we show how autoregressive MCAR model coefficients are found using the multivariate ARMA/CARMA transformation relation in Theorem\,\ref{thm:the_transformation_relation}. Further, the error coefficients are computed through statistical properties of the NIG distribution. This example is easily extended to other dimension and lag parameters. The method of deriving error coefficients only holds for the case of two driving NIG-L{\'e}vy processes. \\

The MCAR model parameters have to be given initially, and we choose parameters as follows. We are considering a dynamical system of two variables, and so the number of model dimensions is $d=2$. Further, it is reasonable to assume that each dimension in the MCAR model is source of a random process, giving rise to the idiosyncratic model error of each dimension separately. Therefore, we set $m=d=2$. As seen in \cite{gomez19}, it is more intricate to find an optimal lag parameter for VARMA models than for the one-dimensional analogue, especially when the driving process is not normal. As this case study is meant as a demonstration of how to apply the transformations in Theorem\,\ref{thm:the_transformation_relation} rather than a detailed assessment of geophysical parameters, we are not concluding an optimal value for $p$ in this work. As argued in the introduction of Section\,\ref{subsect:example}, we set $p=4$. \\

Insert the parameters $p=4$, $d=2$ and $m=d$ into the multivariate ARMA/CARMA transformation relation formula in Theorem\,\ref{thm:the_transformation_relation} (remember that we assume $h=1$ and $q=0$), and rearrange to find the compressed expression
\begin{align}
    \label{eq:MCAR/VAR_transformation_example}
    \bs{x}(t+4) = \Phi_1\bs{x}(t+3) + \Phi_2\bs{x}(t+2) + \Phi_3\bs{x}(t+1) + \Phi_4\bs{x}(t) + \bs{\Phi}(\bs{x}) + \beta\Delta\bs{L}(t),
\end{align}
where $\bs{x}(\cdot) = [x_1(\cdot),x_2(\cdot)]$, $\Delta\bs{L}(t)=[\Delta L_1(t),\Delta L_2(t)]$, and $\Phi_i,\beta\in\R^{2\times 2}$ and $\bs{\Phi}(\bs{x})\in\R^{2}$ are given by 
\begin{align*}
    \Phi_1 = \begin{bmatrix}
    -\alpha_{1,1}^{(1)} \hspace{2mm}&\hspace{2mm} -\alpha_{1,2}^{(1)} \\
    -\alpha_{2,1}^{(1)} \hspace{2mm}&\hspace{2mm} -\alpha_{2,2}^{(1)}
    \end{bmatrix},
    \Phi_2 = \begin{bmatrix}
    -\alpha^{(2)}_{1,1} + 3\alpha_{1,1}^{(1)} \hspace{2mm}&\hspace{2mm} -\alpha^{(2)}_{1,2} + 3\alpha_{1,2}^{(1)} \\
    -\alpha^{(2)}_{2,1} + 3\alpha_{2,1}^{(1)} \hspace{2mm}&\hspace{2mm} -\alpha^{(2)}_{2,2} + 3\alpha_{2,2}^{(1)}
    \end{bmatrix},
\end{align*}
\begin{align*}
    \Phi_3 = \begin{bmatrix}
    -\alpha^{(3)}_{1,1} + 2\alpha_{1,1}^{(2)} - 3\alpha_{1,1}^{(1)} \hspace{2mm}&\hspace{2mm} -\alpha^{(3)}_{1,2} + 2\alpha_{1,2}^{(2)} - 3\alpha_{1,2}^{(1)} \\
    -\alpha^{(3)}_{2,1} + 2\alpha_{2,1}^{(2)} - 3\alpha_{2,1}^{(1)} \hspace{2mm}&\hspace{2mm} -\alpha^{(3)}_{2,2} + 2\alpha_{2,2}^{(2)} - 3\alpha_{2,2}^{(1)}
    \end{bmatrix},
\end{align*}
\begin{align*}
    \Phi_4 = \begin{bmatrix}
    -\alpha^{(4)}_{1,1} + \alpha_{1,1}^{(3)} - \alpha_{1,1}^{(2)} + \alpha^{(1)}_{1,1} \hspace{2mm}&\hspace{2mm} -\alpha^{(4)}_{1,2} + \alpha_{1,2}^{(3)} - \alpha_{1,2}^{(2)} + \alpha^{(1)}_{1,2} \\
    -\alpha^{(4)}_{2,1} + \alpha_{2,1}^{(3)} - \alpha_{2,1}^{(2)} + \alpha^{(1)}_{2,1} \hspace{2mm}&\hspace{2mm} -\alpha^{(4)}_{2,2} + \alpha_{2,2}^{(3)} - \alpha_{2,2}^{(2)} + \alpha^{(1)}_{2,2}
    \end{bmatrix},
\end{align*}
\begin{align*}
    \bs{\Phi}(\bs{x}) = \begin{bmatrix}
    \sum_{i=0}^{3}k(i)x_1(t+i) \\
    \sum_{i=0}^{3}k(i)x_2(t+i)
    \end{bmatrix}\quad\text{and}\quad
    \beta = \begin{bmatrix}
    \beta_{1,1} & \beta_{1,2}\\
    \beta_{2,1} & \beta_{2,2}
    \end{bmatrix},
\end{align*}
where we defined $\alpha_{k,i_s}^{(j)}\triangleq \alpha_{(k-1)d+i_s}^{(p-l+1)}$ and $\beta_{k,i_r}\triangleq \beta_{(k-1)m+i_r}^{(0)}$, and $[k(0),k(1),k(2),k(3)]=[4,-6,4,-1]$. \\

Compare the formula in Eq.\,\eqref{eq:MCAR/VAR_transformation_example} with a two-dimensional VAR($4$) model, generally defined in Eq.\,\eqref{eq:varma_matrix_polynomials}. That is, 
\begin{align}
    \label{eq:var_4_process_example}
    \bs{x}(t+4) = \phi_{1}\bs{x}(t+3) + \phi_{2}\bs{x}(t+2) + \phi_{3}\bs{x}(t+1) + \phi_{4}\bs{x}(t) + \mathbb{1}_2\bs{\mathcal{E}}(t),
\end{align}
where $\bs{\mathcal{E}}(t)\in\R^d$ represents a sequence of serially uncorrelated i.i.d. random vectors and
\begin{align*}
    \phi_j = \begin{bmatrix}
    \phi^{(j)}_{11} & \phi^{(j)}_{12} \\
    \phi^{(j)}_{21} & \phi^{(j)}_{22}
    \end{bmatrix}.
\end{align*}
With the discretized MCAR process and the VAR process expressed as in Eq.\,\eqref{eq:MCAR/VAR_transformation_example} and \eqref{eq:var_4_process_example} respectively, it is straight forward to compute the MCAR model coefficients, $\alpha_{k,i_s}^{(j)}$, as a function of the VAR model coefficients, $\phi^{j}_{\cdot \cdot}$. That is, still with attention on the autoregressive model coefficients, we have for example that 
\begin{align}
    \label{eq:solve_for_alpha}
    \begin{bmatrix}
    \alpha_{1,1}^{(1)} & \alpha_{1,2}^{(1)} \\
    \alpha_{2,1}^{(1)} & \alpha_{2,2}^{(1)}
    \end{bmatrix} = 
    \begin{bmatrix}
    -\phi_{11}^{(1)} - 1 & -\phi_{12}^{(1)} \\
    -\phi_{21}^{(1)} & -\phi_{22}^{(1)} - 1
    \end{bmatrix}.
\end{align}
Continuing with similar approach for $\Phi_2$, $\Phi_3$, and then $\Phi_4$, with some substitutions, we obtain all $16$ autoregressive MCAR model coefficients as a function of the fitted autoregressive VAR model coefficients. The results are given in Appendix\,\ref{app:5}. \\

When fitting a VAR model to data, the distributions of components of the random vector $\bs{\mathcal{E}}(t)$ are easily obtained by fitting an NIG distribution to the model residuals. This is done component-wise by the same methodology as in \cite{eggen2021}. Compare the VAR and MCAR representations in Eq.\,\eqref{eq:MCAR/VAR_transformation_example} and \eqref{eq:var_4_process_example}, respectively, to see that the error coefficients of the MCAR model are given by
\begin{align}
    \label{eq:beta_linear_system}
    \mathcal{E}_k(t)= \beta_{k,1}\Delta L_1(t) + \beta_{k,2}\Delta L_2(t),
\end{align}
where we use the notation $\bs{\mathcal{E}}(t)=[\mathcal{E}_1(t),\mathcal{E}_2(t)]$, and where $\Delta L_1(t)$ and $\Delta L_2(t)$ represents the idiosyncratic error of stratospheric temperature and U wind respectively. Remember that an MCAR model with $m=2$ is considered, and that $\mathcal{E}_k(t)$ and $\Delta L_{i_r}(t)$ ($k,i_r\in\{1,2\}$) are assumed to be NIG distributed random variables. Note that $\beta$ is simply the identity matrix $\mathbb{1}_2$ if the residual datasets represented by $\mathcal{E}(t)$ are independent, as the fitted distributions are given as $\hat{\mathcal{E}}_1(t)\simeq \Delta L_1(t)$ and $\hat{\mathcal{E}}_2(t)\simeq \Delta L_2(t)$ in that case. If the residual datasets are dependent, the question is more intricate. \\

As we will see, to obtain a solution of the linear system in Eq.\,\eqref{eq:beta_linear_system} (if it exists) for $k\in\{1,2\}$, the idiosyncratic error distributions generated by $\Delta L_1(t)$ and $\Delta L_2(t)$ have to be restricted. We assume that the idiosyncratic error distributions are given as in Eq.\,\eqref{eq:levy_is_nig_distributed}, and that the residual datasets for stratospheric temperature and U wind are NIG-distributed as $\mathcal{E}_k(t)\overset{d}{\simeq} NIG(a_k^{\mathcal{E}},b^{\mathcal{E}}_k,\delta^{\mathcal{E}}_k,\mu^{\mathcal{E}}_k)$. The distribution parameters of $\mathcal{E}_k(t)$ are naturally restricted by the explicitly given parameters in Eq.\,\eqref{eq:nig_distribution_of_sum}. \\

Denote the two-dimensional NIG-L{\'e}vy process driving the two-dimensional MCAR($4$) model as $\bs{L}(t)=[L_1(t),L_2(t)]$. Further, assume that the elements of $\bs{L}(t)$ are independent random variables with $E[\bs{L}(t)]=[0,0]$ and $\text{Var}(\bs{L}(t)) = [1,1]$. These assumptions are reasonable because we are considering the MCAR model in Eq.\,\eqref{eq:mcar-model_updated}, meaning that the model is shifted and scaled by the seasonally varying functions $\bs{\Lambda}(t)$ and $\bs{\sigma}(t)$, respectively. The goal from here is to estimate the parameters in $\beta$ from NIG-distributions fitted to the residual datasets, which is represented by $\bs{\mathcal{E}}(t)$. The above assumptions give $\text{Cov}(\bs{L}(t))=\mathbb{1}_2$, meaning that the covariance of $\bs{\mathcal{E}}(t)$ is given as
\begin{align}
    \label{eq:cov_of_epsilon}
    \Sigma \triangleq\text{Cov}(\bs{\mathcal{E}}(t)) = \text{Cov}(\beta\Delta \bs{L}(t)) = \beta\text{Cov}(\Delta \bs{L}(t))\beta' = \beta\beta'=\begin{bmatrix}
\beta_{1,1}^2 + \beta_{1,2}^2  & \beta_{1,1}\beta_{2,1} + \beta_{1,2}\beta_{2,2} \\
 & \beta_{2,1}^2 + \beta_{2,2}^2 
    \end{bmatrix}.
\end{align}
This formula gives a direct relation between $\beta$ and the empirically computed covariance matrix, $\hat{\Sigma}$, of model residuals. As the covariance matrix is symmetric (see Eq.\,\eqref{eq:cov_of_epsilon}), an additional restriction is needed to derive explicit formulas for the components of $\beta$. The additional restriction is set on the idiosyncratic error distributions generated by $\bs{L}(t)$. The goal is to set a restriction that is simple, but still leave the model as flexible as possible. Consider the following assessment of possible restrictions, and how they affect the NIG distribution parameters of $\bs{\mathcal{E}}(t)$:
\begin{enumerate}
    \item Assuming $a_{i_r} = C_a$, $i_r\in\{1,2\}$, for a constant $C_a\geq 0$ gives $\beta_{k,i_r}=C_a/\hat{a}_k^\mathcal{E}$ for all $i_r$, where $\hat{a}_k^\mathcal{E}$ is the estimated tail heaviness for the distribution of $\mathcal{E}_k(t)$. This is not a beneficial restriction as it forces $\beta_{k,1}=\beta_{k,2}$ for each dimension $k$.
    \item A similar argument as in point 1. holds for the idiosyncratic error distribution asymmetry parameters $b_{i_r}$.
    \item The assumption $E[\bs{L}(t)] = [0,0]$ gives $\mu_{i_r} = -\delta_{i_r}b_{i_r}/\sqrt{a_{i_r}^2 - b_{i_r}^2}$, $i_r\in\{1,2\}$, by definition. Therefore, the restriction $\mu_{i_r}=0$ forces $b_{i_r}$ to be zero and vice versa, since $\delta_{i_r} > 0$ by definition. This leads to strict conditions on the idiosyncratic error distributions. 
\end{enumerate}
The restrictions listed above increase the risk of non-existing statistical significant NIG distributions for the elements of $\bs{\mathcal{E}}(t)$. Based on this, there are only two reasonable (and simple enough) choices of restrictions to the idiosyncratic error distributions. Either choose $\mu_{i_r}= C_\mu$ for a constant $\abs{C_\mu}>0$ and $i_r\in\{1,2\}$, or $\delta_{i_r} = C_\delta$ for a constant $C_\delta>0$ and $i_r\in\{1,2\}$. That is, either restrict the scale parameters or the location parameters of the idiosyncratic error distributions to be equal. \\

In this study, we continue with the restriction $\delta_{i_r} = C_\delta$ for a constant $C_\delta>0$ and $i_r\in\{1,2\}$. Thus, by Eq.\,\eqref{eq:nig_distribution_of_sum} we have that
\begin{align}
\label{eq:parameters_delta_restrictions}
    \begin{cases}
    \delta_1^{\mathcal{E}} = \abs{\beta_{1,1}}C_{\delta} + \abs{\beta_{1,2}}C_{\delta}  \\
    \delta_2^{\mathcal{E}} = \abs{\beta_{2,1}}C_{\delta} + \abs{\beta_{2,2}}C_{\delta}
    \end{cases}.
\end{align}
To find the two-dimensional MCAR($4$) model coefficients in $\beta$, the system of equations in Eq.\,\eqref{eq:cov_of_epsilon} and \eqref{eq:parameters_delta_restrictions} has to be solved. Notice that the system of equations are independent of the lag parameter, $p$. \\

Solving the system of equations in Eq.\,\eqref{eq:cov_of_epsilon} and \eqref{eq:parameters_delta_restrictions} gives
\begin{align}
\label{eq:eq_to_find_c_delta_general}
    \pm\sqrt{\Sigma_{11} - \beta_{1,2}^2}\beta_{2,1}  \pm\sqrt{\Sigma_{22} - \beta_{2,1}^2}\beta_{1,2} = \Sigma_{12},
\end{align}
where
\begin{align*}
\beta_{1,2}=\frac{1}{2C_\delta}\left(\pm\delta_1^{\mathcal{E}} \pm \sqrt{2\Sigma_{11}C_\delta^2 - (\delta_1^{\mathcal{E}})^2}\right)\quad\text{and}\quad \beta_{2,1} \in \frac{1}{2C_\delta}\left(\pm\delta_2^{\mathcal{E}} \pm \sqrt{2\Sigma_{22}C_\delta^2 - (\delta_2^{\mathcal{E}})^2}\right),
\end{align*}
for all four possible combinations of signs in each case. For convenience of the reader, the derivations are included in Appendix\,\ref{app:4}. \\

To conclude, all components of the coefficient matrix $\beta$ are dependent on the idiosyncratic error distribution scale parameter, $C_\delta$. Any real-valued solution of components of $\beta$ are allowed as long as Eq.\,\eqref{eq:eq_to_find_c_delta_general} is satisfied, including the restrictions 
\begin{align}
    \label{eq:beta_0_restrictions}
    1)\;C_\delta > 0;\quad 2)\;C_\delta \geq \delta_1^{\mathcal{E}}/\sqrt{2\Sigma_{11}}, C_\delta \geq \delta_2^{\mathcal{E}}/\sqrt{2\Sigma_{22}};\quad 3)\; \Sigma_{11}\geq \beta_{2,1}^2, \Sigma_{22}\geq \beta_{1,2}^2.
\end{align}


\subsubsection{\label{subsubsect:data_and_var_fit}Fit model to stratospheric temperature and U wind data}


Empirical estimation of coefficients of a two-dimensional MCAR($4$) model driven by a two-dimensional NIG-L{\'e}vy process is performed using results from Section\,\ref{subsubsect:parameter_est}. The model is fit to daily zonal mean stratospheric temperature and U wind data. See Table\,\ref{tab:dataset_specs} for data specifications, and note that the datasets are prepared as presented in \cite{eggen2021}. \\

\begin{table}[hbt!]
\centering
\caption{Specifications of stratospheric temperature and U wind datasets. The specifications for the two datasets are similar, except units. Initial datasets are retrieved as ECMWF ERA-Interim reanalysis model products. \label{tab:dataset_specs}}
\begin{tabular}{cccccc}
\toprule
Date & Grid & Pressure level & Time & Area & Unit  \\
\midrule
\makecell{1 January 1979 to\\31 December 2018} \hspace{1mm} & \hspace{1mm} $0.5^{\circ}$ \hspace{1mm} & \hspace{1mm} $10$\;hPa \hspace{1mm} & \hspace{1mm} \makecell{00:00, 06:00,\\12:00, 18:00} \hspace{1mm} & \hspace{1mm} \makecell{$60^{\circ}$N and\\$[-180^{\circ}\text{E},180^{\circ}\text{E})$} \hspace{1mm} & \hspace{1mm} \makecell{Temp.: Kelvin,\\U wind: m/s}\\
\bottomrule  
\end{tabular}
\end{table}

As argued in the introduction of Section\,\ref{subsect:example}, it is reasonable to assume that the dynamical system of stratospheric temperature and U wind follow a VAR/MCAR model. In the following, we will work with the two-dimensional state space model in Eq.\,\eqref{eq:mcar-model_updated}. That is, the dynamical system is given by $\bs{Y}(t)=[Y_1(t),Y_2(t)]$, where $Y_1(t)$ and $Y_2(t)$ represents temperature and U wind respectively. As explained in Section\,\ref{subsubsect:parameter_est}, the parameter values of the MCAR model are initially given as $d=2$, $m=d=2$ and $p=4$ (remember that $q=0$ in a MCAR model). We use a similar methodology as in \cite{eggen2021} to fit a VAR model, equivalent to the modified MCAR model in Eq.\,\eqref{eq:mcar-model_updated}, to temperature and U wind datasets. That is, with all following steps performed component-wise for $Y_1(t)$ and $Y_2(t)$, except when fitting the VAR model, the methodology goes as
\begin{enumerate}
    \item Fit a continuous seasonality function $\bs{\Lambda}(t)=[\Lambda_1(t),\Lambda_2(t)]$ to datasets representing $Y_1(t)$ and $Y_2(t)$. Deseasonalized datasets, represented by $X_1(t)$ and $X_1(t)$, are obtained by subtracting the fitted functions, $\Lambda_1(t)$ and $\Lambda_2(t)$, from the datasets.
    \item Fit a VAR($p$) model to the (two-dimensional) deseasonalized dataset using a standard statistical programming package. Subtract the fitted VAR($p$) model from the deseasonalized datasets (component-wise) to obtain datasets of model residuals. 
    \item Compute empirically the expected values of squared residuals each day over the year (assumed to be $365$ days as each February 29 is removed for convenience) to construct the bivariate volatility function, $\bs{\sigma}(t)=[\sigma_1(t),\sigma_2(t)]$. 
    \item Scale the residual datasets with the constructed volatility function components to obtain datasets of $\bs{\sigma}(t)$-scaled model residuals, to which NIG distributions are fitted. These distributions are said to be generated by the random vector $\bs{\mathcal{E}}(t)$.
    \item The MCAR model coefficients are computed from formulas derived in Section\,\ref{subsubsect:parameter_est}. 
\end{enumerate}

As in \cite{eggen2021}, the seasonality function components of $\bs{\Lambda}(t)$ are assumed to be given by
\begin{align}
    \label{eq:lambda_def}
    \Lambda_k(t) = c_0^{(k)} + c_1^{(k)}t + \sum_{j=1}^{10} \left( c_{2j}^{(k)}\cos(j\pi t/365) + c_{2j + 1}^{(k)}\sin(j \pi t/365) \right),
\end{align}
where $k=1,2$ gives the dimension and $c_0^{(k)},c_1^{(k)},c_2^{(k)},\ldots ,c_{21}^{(k)}$ are constants. Note that $c_0^{(k)}$ captures the initial mean value, $c_1^{(k)}$ captures the slope of the long-term liner changes, and $c_0^{(k)},c_1^{(k)},c_2^{(k)},\ldots ,c_{21}^{(k)}$ describe the yearly cycle as weights in the truncated Fourier series. The ten last years of the stratospheric temperature and U wind datasets are shown with a fitted seasonality function in Figure\,\ref{fig:seasonality_function}, and the corresponding parameter values of Eq.\,\eqref{eq:lambda_def} are listed in Table\,\ref{tab:seasonality_function_pars_temp_uwind}. \\

\begin{figure}[hbt!]
	\centering
	\begin{subfigure}[ht!]{0.49\textwidth}
	\centering
	\includegraphics[scale=0.38]{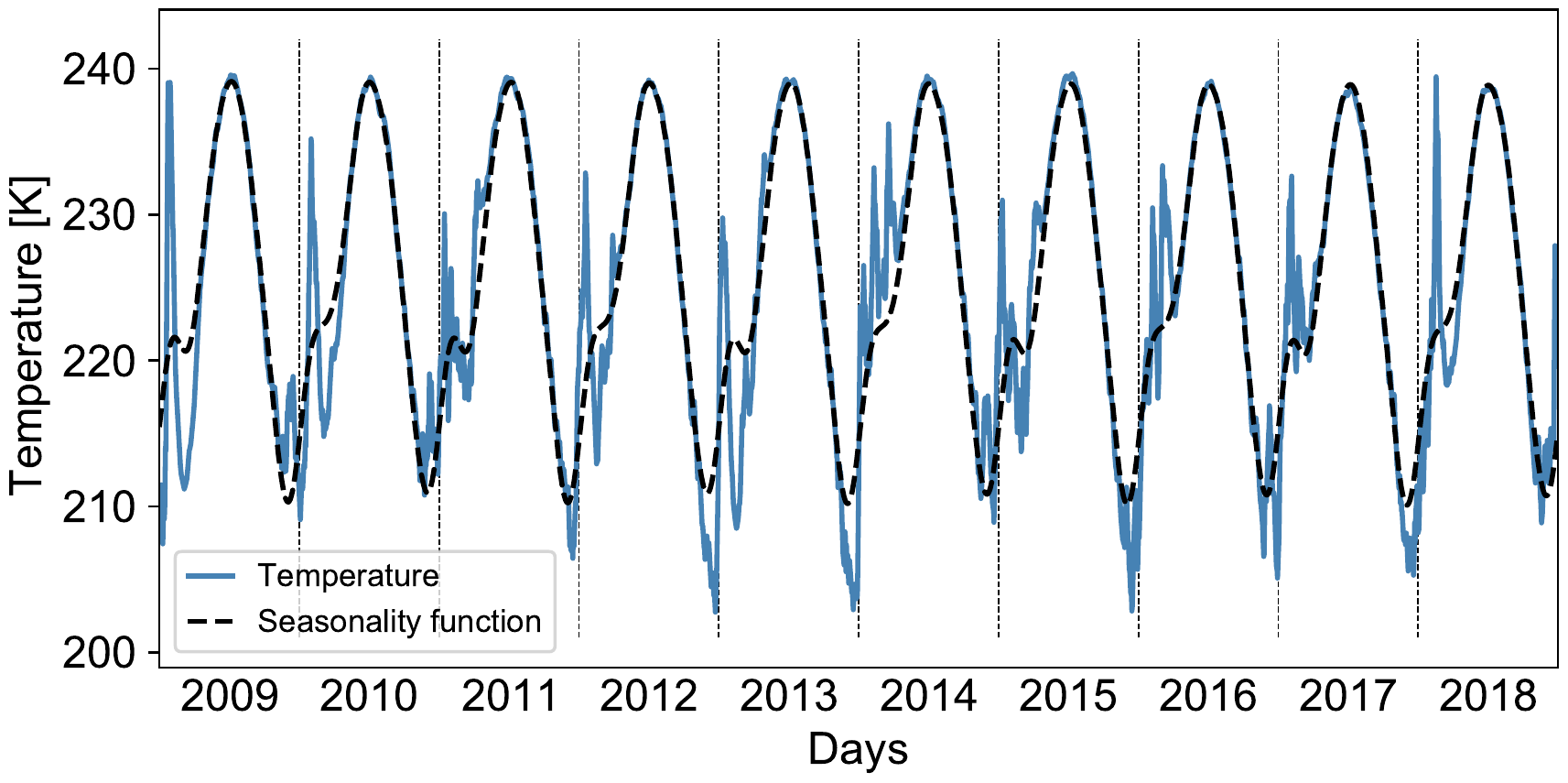}
	\caption{Stratospheric temperature}
	\label{fig:temp_seasonality}
	\end{subfigure}
	\begin{subfigure}[ht!]{0.49\textwidth}
	\centering
	\includegraphics[scale=0.38]{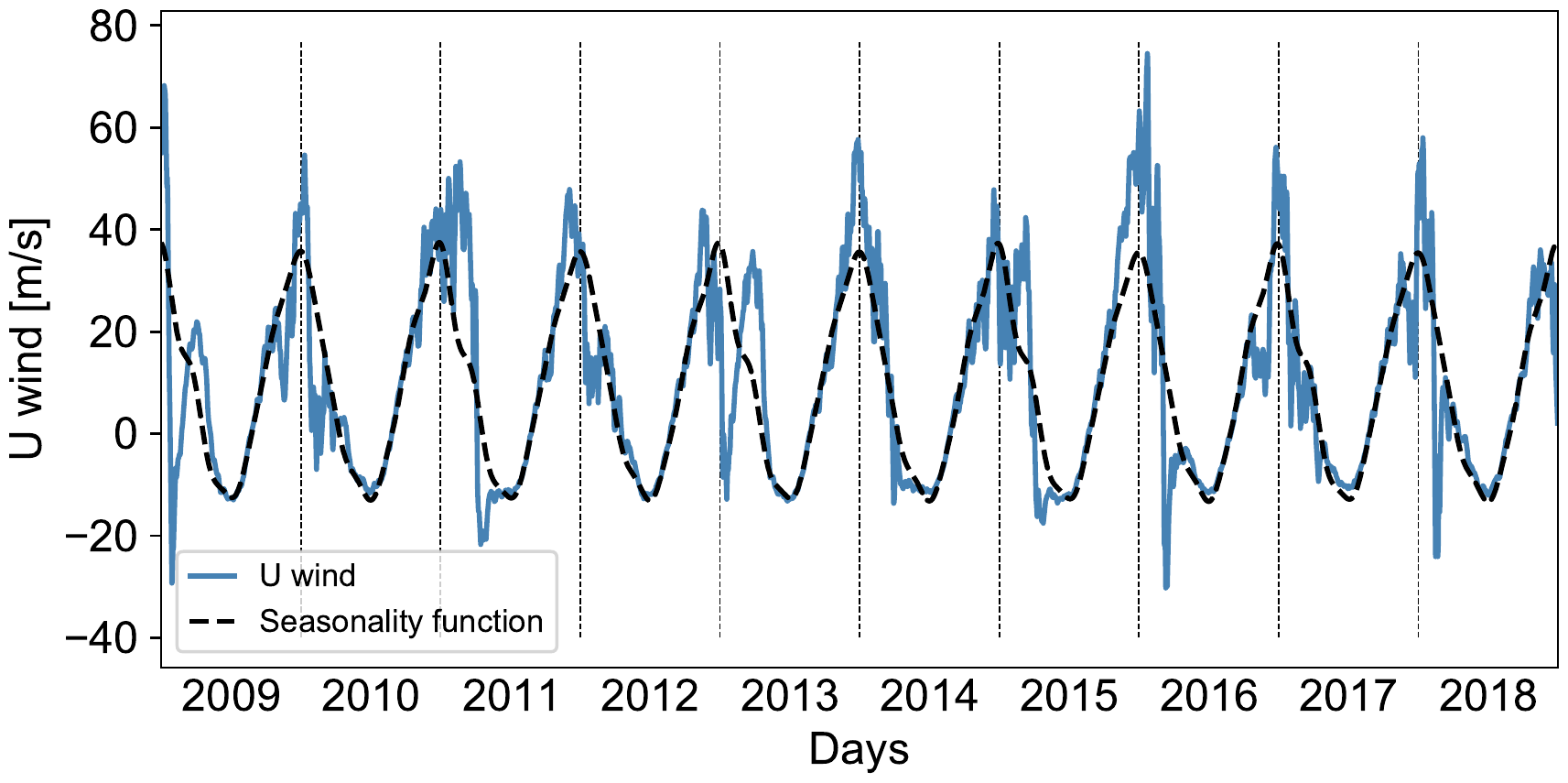}
	\caption{Stratospheric U wind}
	\label{fig:uwind_seasonality}
	\end{subfigure}
    \caption{Daily circumpolar mean stratospheric temperature and U wind from 1 January 2009 to 31 December 2018 with fitted seasonality function $\Lambda_1(t)$ and $\Lambda_{2}(t)$, respectively.}
    \label{fig:seasonality_function}
\end{figure}

\begin{table}[hbt!]
\centering
\caption{Coefficients of seasonality functions $\Lambda_1(t)$ and $\Lambda_2(t)$ fitted to stratospheric temperature and U wind. \label{tab:seasonality_function_pars_temp_uwind}}
\begin{tabular}{ccccccccccc}
\toprule
\multicolumn{11}{c}{Seasonality function parameters: Temperature} \\
\midrule
$c_0$ & $c_2$ & $c_4$ & $c_6$ & $c_8$ & $c_{10}$ & $c_{12}$ & $c_{14}$ & $c_{16}$ & $c_{18}$ & $c_{20}$  \\
$226.15$ & $-0.049$ & $-12.094$ & $0.23$ & $1.88$ & $0.33$            & $0.16$ & $0.13$ & $-0.094$ & $-0.15$ & $-0.014$ \\
\midrule
$c_1$ & $c_3$ & $c_5$ & $c_7$ & $c_9$ & $c_{11}$ & $c_{13}$ & $c_{15}$ & $c_{17}$ & $c_{19}$ & $c_{21}$  \\
$-0.000072$ & $-0.11$ & $1.63$ & $-0.23$ & $2.81$ & $-0.040$            & $1.54$ & $0.14$ & $0.45$ & $0.049$ & $0.11$ \\
\toprule
\multicolumn{11}{c}{Seasonality function parameters: U wind} \\
\midrule
$c_0$ & $c_2$ & $c_4$ & $c_6$ & $c_8$ & $c_{10}$ & $c_{12}$ & $c_{14}$ & $c_{16}$ & $c_{18}$ & $c_{20}$  \\
$11.18$ & $0.19$ & $22.58$ & $-0.40$ & $0.88$ & $-0.29$ & $1.12$ & $0.59$ & $1.06$ & $0.66$ & $0.95$ \\
\midrule
$c_1$ & $c_3$ & $c_5$ & $c_7$ & $c_9$ & $c_{11}$ & $c_{13}$ & $c_{15}$ & $c_{17}$ & $c_{19}$ & $c_{21}$  \\
$-0.00011$ & $0.17$ & $-4.29$ & $0.015$ & $-0.33$ & $-0.69$ & $0.18$ & $-0.57$ & $-0.37$ & $0.19$ & $0.31$ \\
\bottomrule  
\end{tabular}
\end{table}

Studying crosscorrelations of deseasonalized stratospheric temperature and U wind confirms dependence in daily lagged values between the two variates, see Figure\,\ref{fig:deseasonalized_correlations}. This supports the assumption that an MCAR process might represent the two-dimensional system well. A VAR model is fit to the deseasonalized dataset, corresponding to $[Y_1(t)-\Lambda_1(t),Y_2(t)-\Lambda_2(t)]$ (see Eq.\,\eqref{eq:mcar-model_updated}), using the add-on SSMMATLAB in MATLAB, see \cite{gomez_ssmmatlab} and \cite{gomez19}. The fitted VAR coefficients are given in Table\,\ref{tab:var_parameter_fit} with corresponding t-values. Most of the significance levels are acceptable. Note that the constants, $\bs{c}$, of the VAR model are taken as zeroes, as deseasonalized data are considered. \\

\begin{figure}[hbt!]
\centering
\includegraphics[scale=0.5]{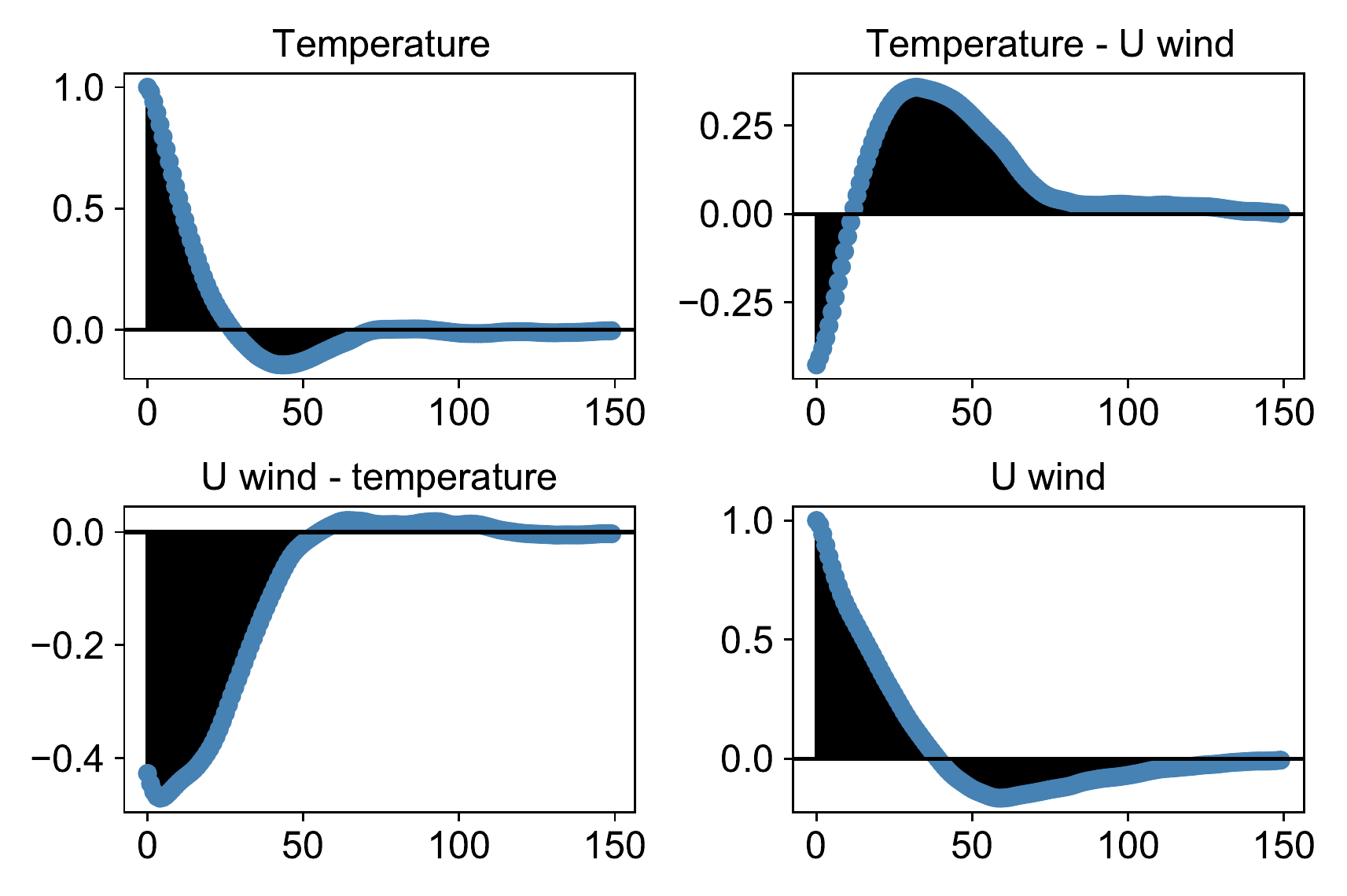}
\caption{Autocorrelation and crosscorrelation of deseasonalized stratospheric temperature and U wind.}
\label{fig:deseasonalized_correlations}
\end{figure}

\begin{table}[hbt!]
\centering
\caption{Coefficients of VAR model fitted to the two-dimensional dataset of deseasonalized stratospheric temperature and U wind, with corresponding t-values. \label{tab:var_parameter_fit}}
\begin{tabular}{ccccc}
\toprule
\multicolumn{5}{c}{VAR model coefficients and corresponding t-values} \\
\midrule
$\bs{c}$ & $\hat{\phi}_1$ & $\hat{\phi}_2$ & $\hat{\phi}_3$ & $\hat{\phi}_4$ \\
$\begin{bmatrix}0.0003 \\ -0.0014\end{bmatrix}$ &
$\begin{bmatrix} 1.53 & 0.0008 \\ -0.10 & 1.73\end{bmatrix}$ & 
$\begin{bmatrix} -0.73 & 0.024 \\ -0.097 & -1.03\end{bmatrix}$ &
$\begin{bmatrix} 0.27 & 0.0029 \\ 0.14 & 0.32\end{bmatrix}$ &
$\begin{bmatrix} -0.10 & 0.023  \\ 0.0026 & -0.041 \end{bmatrix}$ \\
$t_{\bs{c}}$ & $t_{\phi_1}$ & $t_{\phi_2}$ & $t_{\phi_3}$ & $t_{\phi_4}$ \\
$\begin{bmatrix}0.058 \\ -0.12\end{bmatrix}$ &
$\begin{bmatrix} 186.43 & 0.25 \\ -4.94 & 208.67\end{bmatrix}$ & 
$\begin{bmatrix} -48.83 & -3.78 \\ -2.56 & -63.25 \end{bmatrix}$ &
$\begin{bmatrix} 17.88 & 0.45 \\ 3.79 & 19.37 \end{bmatrix}$ &
$\begin{bmatrix} -12.18 & 6.91   \\ 0.13 & -4.92  \end{bmatrix}$ \\
\bottomrule  
\end{tabular}
\end{table}

The yearly varying multivariate volatility function $\bs{\sigma}(t)$ is estimated from the model residuals. As in \cite{eggen2021}, each volatility function, $\sigma_k(t)$ for $k=1,2$, is constructed from three truncated Fourier series on the form\\
\begin{align*}
    w_{f_i,k}^{(n_i)}(t) = d_{0,k}^{(i)} + \sum_{j=1}^{n_i}\left(d_{2j-1,k}^{(i)}\cos(f_i j\pi t/365) + d_{2j,k}^{(i)} \sin(f_i j\pi t/365)\right),
\end{align*}
where $k=1,2$ gives the dimension, $i\in\{1,2,3\}$ represents the series number, $f_i$ adjusts the series frequency, $n_i$ gives the number of terms in each series, and $d_{0,k}^{(i)},d_{2j-1,k}^{(i)},d_{2j,k}^{(i)}$ are constants. Each of the three fitted series, for each $k$, represents seasons winter/spring, summer, and autumn/winter, defined by the intervals [1 January,30 April], [1 May,31 October] and [1 November,31 December], respectively. See \cite{eggen2021} for an explanation of why. The final volatility function $\sigma_k(t)$ is constructed by connecting each of the three fitted functions $w_{f_i,k}^{(n_i)}(t)$ using sigmoid functions. That is, let the sigmoid function, $\omega(x)$, and the connective function, $\xi(x)$, be given by 
\begin{align*}
\omega(x) = \frac{1}{1+\exp\left(-(\frac{x-a}{b})\right)} \quad \text{and}\quad  \xi(x) = \big(1-\omega(x)\big)g_1(x) + \omega(x) g_2(x),
\end{align*}
where $a$ and $b$ are shift and scaling constants respectively, and $g_1(x)$ and $g_2(x)$ are two functions that are to be connected. The numerical results are as follows. For stratospheric temperature, the functions $w_{0.44,1}^{(2)}(t)$ and $w_{0.30,2}^{(2)}(t)$ are connected with $w_{2.0,1}^{(2)}(t)$ and $w_{0.50,2}^{(3)}(t)$ respectively, using parameters $a=120$ and $b=2$. Further, for stratospheric U wind, the functions $w_{2.0,1}^{(2)}(t)$ and $w_{0.50,2}^{(3)}(t)$ are connected with $w_{0.44,1}^{(2)}(t)$ and $w_{0.05,2}^{(4)}(t)$ respectively, using parameters $a=303$ and $b=5$. These connections make up two smooth functions $\sigma_1(t)$ and $\sigma_2(t)$, that are illustrated in Figure\,\ref{fig:estimated_variance} as the variance function $V_k(t)=\sigma^2_k(t)$, together with empirically estimated daily variance values over the year. The corresponding coefficients of the truncated Fourier series are given in Table\,\ref{tab:volatility_parameter_fit_temp} and Table\,\ref{tab:volatility_parameter_fit_uwind} for stratospheric temperature and U wind respectively. \\

\begin{figure}[hbt!]
	\centering
	\begin{subfigure}[ht!]{0.49\textwidth}
	\centering
	\includegraphics[scale=0.38]{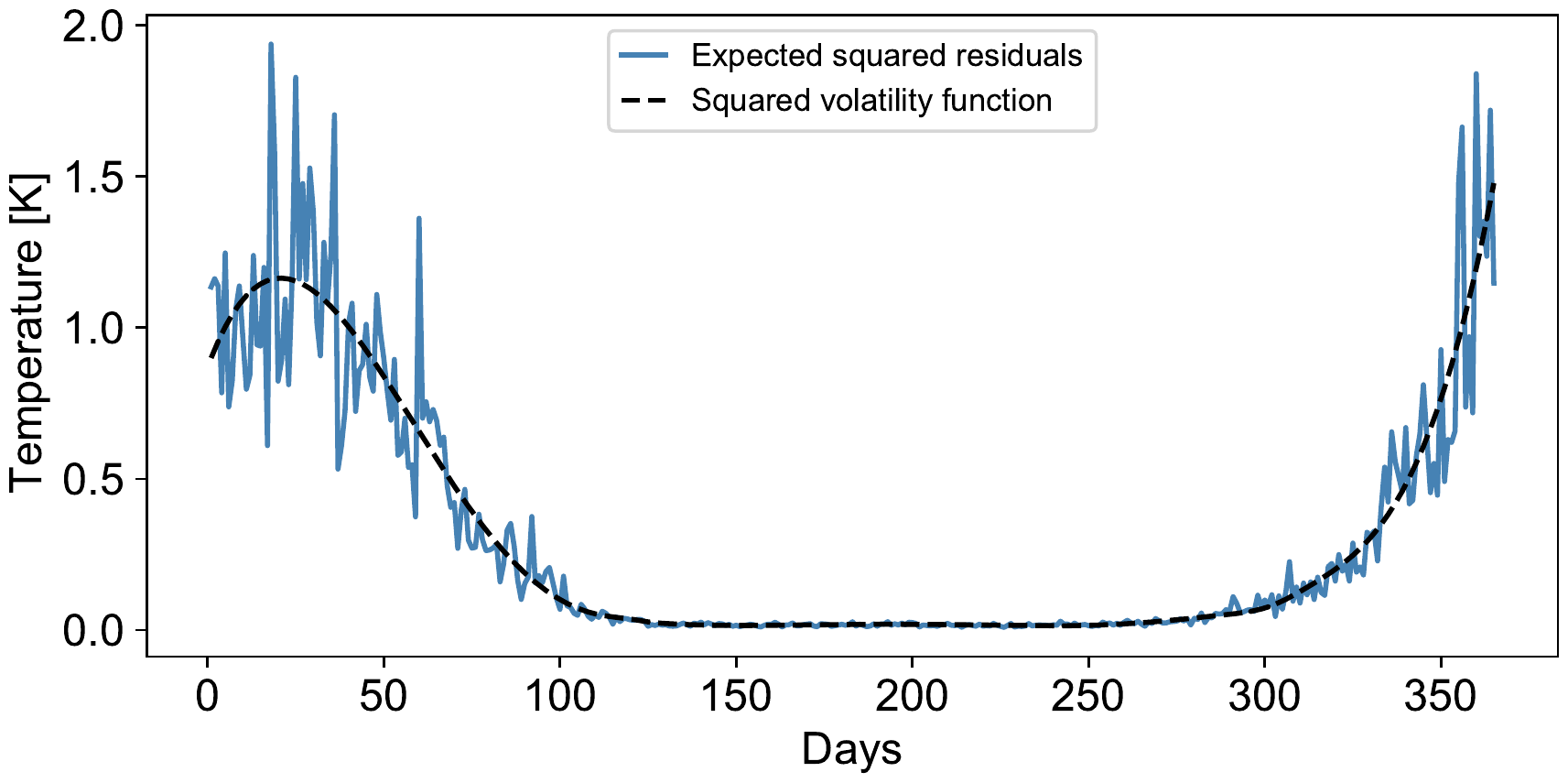}
	\caption{Stratospheric temperature}
	\label{fig:temp_variance}
	\end{subfigure}
	\begin{subfigure}[ht!]{0.49\textwidth}
	\centering
	\includegraphics[scale=0.38]{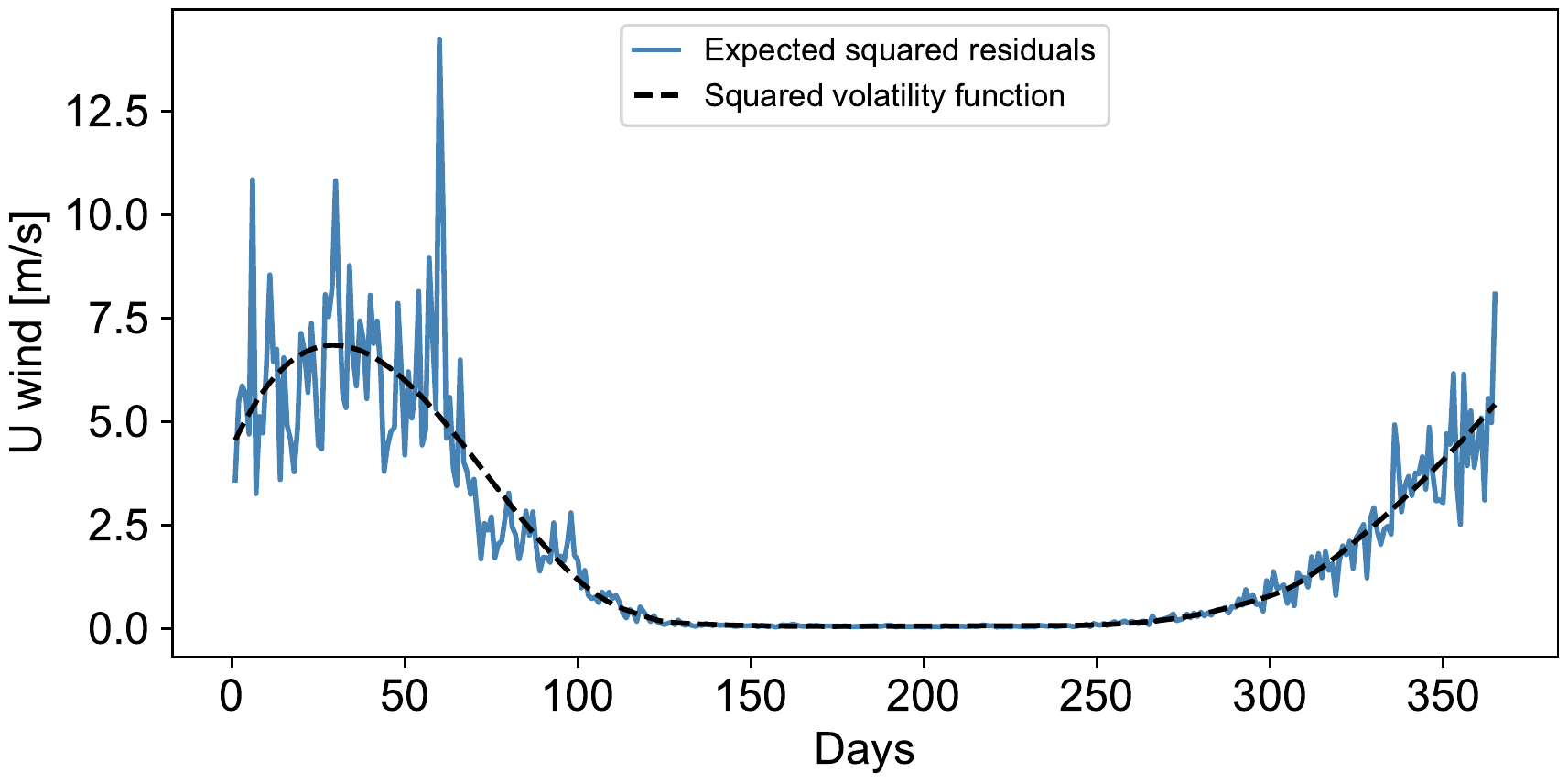}
	\caption{Stratospheric U wind}
	\label{fig:uwind_variance}
	\end{subfigure}
    \caption{Estimated expected squared residuals ($\simeq$\;variance) each day of the year and a fitted variance function.}
    \label{fig:estimated_variance}
\end{figure}

\begin{table}[hbt!]
\centering
\caption{Coefficients of the volatility function $\sigma_1(t)$ fitted to VAR model residuals of stratospheric temperature. \label{tab:volatility_parameter_fit_temp}}
\begin{tabular}{ccccc}
\toprule
\multicolumn{5}{c}{Winter/spring: $w_{0.44,1}^{(2)}$} \\
\midrule
$d_{0,1}^{(1)}$ & $d_{1,1}^{(1)}$ & $d_{2,1}^{(1)}$ & $d_{3,1}^{(1)}$ & $d_{4,1}^{(1)}$ \\
$-595.152$ & $749.200$ & $289.857$ & $-153.179$ & $-140.785$ \\
\midrule
\multicolumn{5}{c}{Summer: $w_{2.0,1}^{(2)}$} \\
\midrule
$d_{0,1}^{(2)}$ & $d_{1,1}^{(2)}$ & $d_{2,1}^{(2)}$ & $d_{3,1}^{(2)}$ & $d_{4,1}^{(2)}$ \\
$0.089$ & $0.103$ & $0.022$ & $0.033$ & $0.014$ \\
\midrule
\multicolumn{5}{c}{Autumn/winter: $w_{0.44,1}^{(2)}$} \\
\midrule
$d_{0,1}^{(3)}$ & $d_{1,1}^{(3)}$ & $d_{2,1}^{(3)}$ & $d_{3,1}^{(3)}$ & $d_{4,1}^{(3)}$ \\
$13.000$ & $-228.138$ & $75.488$ & $80.783$ & $87.432$ \\
\bottomrule  
\end{tabular}
\end{table}

\begin{table}[hbt!]
\centering
\caption{Coefficients of the volatility function $\sigma_2(t)$ fitted to VAR model residuals of stratospheric U wind. \label{tab:volatility_parameter_fit_uwind}}
\begin{tabular}{ccccccccc}
\toprule
\multicolumn{9}{c}{Winter/spring: $w_{0.30,2}^{(2)}$} \\
\midrule
$d_{0,2}^{(1)}$ & $d_{1,2}^{(1)}$ & $d_{2,2}^{(1)}$ & $d_{3,2}^{(1)}$ & $d_{4,2}^{(1)}$ & $-$ & $-$ & $-$ & $-$\\
$-32.767$ & $-333.062$ & $1960.769$ & $370.204$ & $-945.071$ & $-$ & $-$ & $-$ & $-$ \\
\midrule
\multicolumn{9}{c}{Summer: $w_{0.50,2}^{(3)}$} \\
\midrule
$d_{0,2}^{(2)}$ & $d_{1,2}^{(2)}$ & $d_{2,2}^{(2)}$ & $d_{3,2}^{(2)}$ & $d_{4,2}^{(2)}$ & $d_{5,2}^{(2)}$ & $d_{6,2}^{(2)}$ & $-$ & $-$\\
$22.726$ & $-127.933$ & $89.798$ & $126.536$ & $2.515$ & $-24.633$ & $-22.122$ & $-$ & $-$ \\
\midrule
\multicolumn{9}{c}{Autumn/winter: $w_{0.05,2}^{(4)}$} \\
\midrule
$d_{0,2}^{(3)}$ & $d_{1,2}^{(3)}$ & $d_{2,2}^{(3)}$ & $d_{3,2}^{(3)}$ & $d_{4,2}^{(3)}$ & $d_{5,2}^{(3)}$ & $d_{6,2}^{(3)}$ & $d_{7,2}^{(3)}$ & $d_{8,2}^{(3)}$ \\
$113.995$ & $93.302$ & $20.305$ & $33.465$ & $22.639$ & $-58.418$ & $-7.422$ & $-170.865$ & $-83.501$ \\
\bottomrule  
\end{tabular}
\end{table}

The $\bs{\sigma}(t)$-scaled VAR model residual datasets are shown to be NIG-distributed with statistical significance. The NIG parameters for stratospheric temperature and U wind are listed in Table\,\ref{tab:nig_parameter_fit} with corresponding KS test statistics and p-values. The autocorrelations and crosscorrelations of $\bs{\sigma}(t)$-scaled VAR model residuals are shown in Figure\,\ref{fig:residuals_correlations}. It is apparent that the two-dimensional VAR($4$) model explains most of the time lagged dependence. As discussed in \cite{eggen2021}, the remaining memory effects indicate that a stochastic volatility model would provide a more accurate model. \\

\begin{table}[hbt!]
\centering
\caption{Fitted NIG parameters for $\bs{\sigma}(t)$-scaled VAR model residuals and corresponding KS test results. \label{tab:nig_parameter_fit}}
\begin{tabular}{cccccc}
\toprule
\multicolumn{6}{c}{Stratospheric temperature} \\
\midrule
$a_1^{\mathcal{E}}$ & $b_1^{\mathcal{E}}$ & $\delta_1^{\mathcal{E}}$ & $\mu_1^{\mathcal{E}}$ & statistic & p-value \\
$2.93$ & $0.398$ & $1.70$ & $-0.234$ & $0.007$ & $0.47$ \\
\midrule
\multicolumn{6}{c}{Stratospheric U wind} \\
\midrule
$a_2^{\mathcal{E}}$ & $b_2^{\mathcal{E}}$ & $\delta_2^{\mathcal{E}}$ & $\mu_2^{\mathcal{E}}$ & statistic & p-value \\
$3.13$ & $-0.0781$ & $1.77$ & $0.0471$ & $0.005$ & $0.90$ \\
\bottomrule 
\end{tabular}
\end{table}

\begin{figure}[hbt!]
\centering
\includegraphics[scale=0.5]{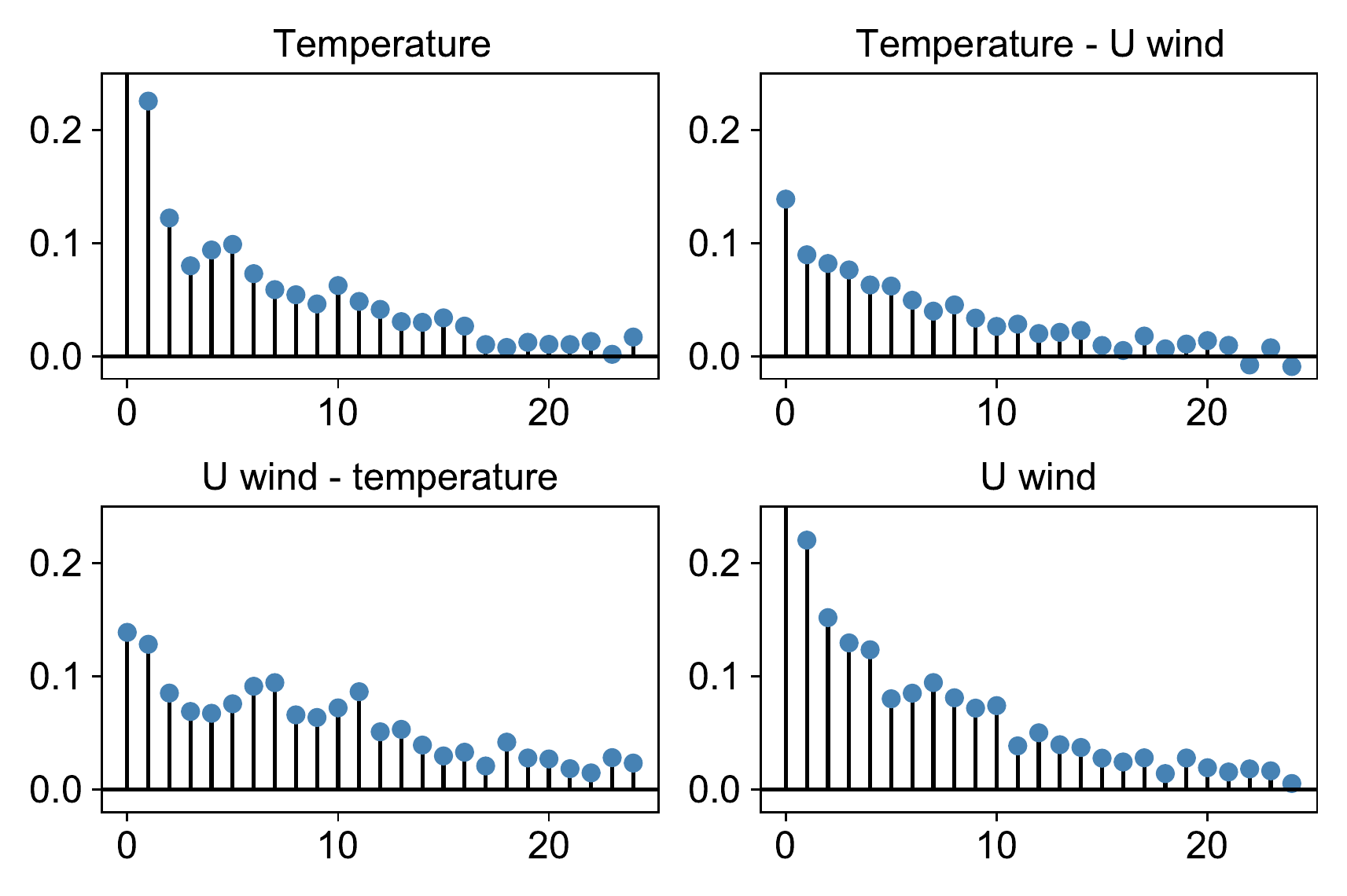}
\caption{Autocorrelation and crosscorrelation of squared $\bs{\sigma}(t)$-scaled model residuals of stratospheric temperature and U wind.}
\label{fig:residuals_correlations}
\end{figure}

The autoregressive MCAR model coefficients are found from the VARMA/MCARMA transformation relation in Theorem\,\ref{thm:the_transformation_relation}, that is explicitly given for our model framework in Section\,\ref{subsubsect:parameter_est}. The empirical results are shown in Table\,\ref{tab:estimated_mcar_coefficients}, where the autoregressive coefficients are displayed as the matrix components of the MCAR model block matrix $A$, see Eq.\,\eqref{eq:the_mcarma_model} and \eqref{eq:matrix_element_definition}. \\

\begin{table}[hbt!]
\centering
\caption{Estimated autoregressive coefficients of the MCAR model fit to stratospheric temperature and U wind. \label{tab:estimated_mcar_coefficients}}
\begin{tabular}{cccc}
\toprule
\multicolumn{4}{c}{Autoregressive MCAR model coefficients} \\
\midrule
$\hat{A}_4$ & $\hat{A}_3$ & $\hat{A}_2$ & $\hat{A}_1$ \\
$\begin{bmatrix} 0.030 & -0.0019 \\ 0.054 & 0.029\end{bmatrix}$ & 
$\begin{bmatrix} -4.41 & 0.043 \\ 0.36 & -4.43\end{bmatrix}$ &
$\begin{bmatrix} -2.87 & 0.022 \\ 0.41 & -3.15 \end{bmatrix}$ &
$\begin{bmatrix} -2.53 & -0.00083  \\ 0.10 & -2.72 \end{bmatrix}$ \\
\bottomrule  
\end{tabular}
\end{table}

As mentioned in Section\,\ref{subsec:varma_mcarma_definition}, the MCAR process has a unique causal stationary solution when the eigenvalues of the block matrix $A$ have negative real parts. As reported in Table\,\ref{tab:block_matrices_eigenvalues}, we see that the eigenvalues $\lambda^{(A)}_{3,4}$ fail having this property. Following \cite{gomez16}, we know that the VAR process has a unique causal stationary solution if and only if the modulus of each root of the polynomial $\det(\phi(z))$ is greater than one. Note that $\phi(z) = \left(I-\phi_1z-\phi_2z^2 -\phi_3z^3-\phi_4z^4\right)$ in this case, see Eq.\,\eqref{eq:varma_matrix_polynomials}. Further, \cite{gomez16} states that when the VAR process is written in a so-called Akaike's state space form, one can show that 
\begin{align*}
    \det(\phi(\lambda)) = \det(\mathbb{1}_d-F\lambda),\quad\text{where}\quad F = \begin{bmatrix}
       \mathbb{0}_d & \mathbb{1}_d & \mathbb{0}_d & \mathbb{0}_d       \\
       \mathbb{0}_d & \mathbb{0}_d & \mathbb{1}_d & \mathbb{0}_d        \\
       \mathbb{0}_d & \mathbb{0}_d & \mathbb{0}_d & \mathbb{1}_d \\
       \phi_4 & \phi_3 & \phi_2  & \phi_1 \\
     \end{bmatrix}.
\end{align*}
This means that the above condition on existence of a proper VAR process solution is equivalent to the modulus of each eigenvalue of $F$ being less than one. The eigenvalues of $F$ are listed in Table\,\ref{tab:block_matrices_eigenvalues} with their respective modulus. All eight eigenvalues have modulus less than one, and thus the fitted VAR process has a unique causal stationary solution. \\

\begin{table}[hbt!]
\centering
\caption{Eigenvalues of the MCAR state space representation matrix, $A$, and of the VAR Akaike's state space form matrix, $F$. \label{tab:block_matrices_eigenvalues}}
\begin{tabular}{ccccc}
\toprule
\multicolumn{5}{c}{Eigenvalues of $A$} \\
\midrule
$\lambda_{1}^{(A)}$ & $\lambda_{2}^{(A)}$ & $\lambda_{3,4}^{(A)}$ & $\lambda_{5,6}^{(A)}$ & $\lambda_{7,8}^{(A)}$ \\
$-2.21$ \hspace{1mm} & \hspace{1mm} $-2.15$ \hspace{1mm} & \hspace{1mm} $0.0067\pm 0.0021i$ \hspace{1mm} & \hspace{1mm} $-0.20\pm 1.42i$ \hspace{1mm} & \hspace{1mm} $-0.25\pm 1.40i$ \\
\midrule
\multicolumn{5}{c}{Eigenvalues of $F$} \\
\midrule
$\lambda_{1}^{(F)}$ & $\lambda_{2}^{(F)}$ & $\lambda_{3,4}^{(F)}$ & $\lambda_{5,6}^{(F)}$ & $\lambda_{7,8}^{(F)}$  \\
$0.18$ & $0.68$ & $-0.024\pm 0.41i$ & $0.28\pm 0.38i$ & $0.94\pm 0.035i$  \\
\midrule
\multicolumn{5}{c}{Modulus of eigenvalues of $F$} \\
\midrule
$\abs{\lambda_{1}^{(F)}}$ & $\abs{\lambda_{2}^{(F)}}$ & $\abs{\lambda_{3,4}^{(F)}}$ & $\abs{\lambda_{5,6}^{(F)}}$ & $\abs{\lambda_{7,8}^{(F)}}$  \\
$0.18$ & $0.68$ & $0.41$ & $0.47$ & $0.94$  \\
\bottomrule  
\end{tabular}
\end{table}

The above test for existence of stationary solutions show that there exist a solution for the the VAR process, however not necessarily for the corresponding MCAR process. This fact leads us to suspect that the real part of $\lambda^{(A)}_{3,4}$ is positive due to approximation errors. That is, the transformation relation transforming autoregressive VAR model coefficients to corresponding MCAR model coefficients is based on discretization, introducing an Euler discretization error to the autoregressive (deterministic) part. Also standard programming packages used to compute eigenvalues of large matrices introduce approximation errors. For example, the perturbation $\rho=-0.03$ of the blocks $A_j$ $(j=1,2,3,4)$ in $A$ results in eigenvalues $\lambda^{(A)}_{3,4}=-0.00032\pm 0.0062i$, where the rest of the eigenvalues have negative real part as well. It is also likely that additional fine-tuning of the seasonality functions, $\bs{\Lambda}(t)$, and yearly volatility functions, $\bs{\sigma}(t)$, give an MCAR model satisfying the existence condition, regardless of approximation errors. \\

Finally, the model error coefficients, $\beta$, are calculated based on the derivations in Section\,\ref{subsubsect:parameter_est}. As shown in Eq.\,\eqref{eq:eq_to_find_c_delta_general}, the coefficients are dependent on the scale parameters, $\delta_1^{\mathcal{E}}$ and $\delta_2^{\mathcal{E}}$, of the NIG distributions $\mathcal{E}_1(t)$ and $\mathcal{E}_2(t)$, as well as the crosscorrelation matrix, $\Sigma$, of $\bs{\mathcal{E}}(t)$. The formula in Eq.\,\eqref{eq:eq_to_find_c_delta_general} is derived by requiring that the scale parameters of the idiosyncratic error distributions, $\Delta L_1(t)$ and $\Delta L_2(t)$ are equal to a constant $C_\delta$. As $\delta_1^{\mathcal{E}}$, $\delta_2^{\mathcal{E}}$ and $\Sigma$ are known parameters (see the procedure of fitting the autoregressive model coefficients), the task of solving the system in Eq.\,\eqref{eq:eq_to_find_c_delta_general} reduces to finding a constant $C_\delta$ satisfying all given restrictions. An analytical solution is not found, and a standard numerical solver has to be used to find an optimal value of $C_\delta$.

The empirical covariance matrix of $\bs{\mathcal{E}}(t)$ is given by 
\begin{align*}
    \hat{\Sigma} = \begin{bmatrix}
    1.018 & -0.02238 \\
    -0.02238 & 1.006
    \end{bmatrix}.
\end{align*}
By definition, the distributions of $\mathcal{E}_1(t)$ and $\mathcal{E}_2(t)$ are nearly independent, meaning that $\beta$ is approximately equal to the identity matrix (see Eq.\,\eqref{eq:beta_linear_system}). To confirm this heuristic result, and to give an illustrative example of the derivations in Section\,\ref{subsubsect:parameter_est}, the system in Eq.\,\eqref{eq:eq_to_find_c_delta_general} is solved, and the results presented. An appropriate solution is found for the case when 
\begin{align}
\label{eq:eq_to_find_c_delta}
    \sqrt{\Sigma_{11} - \beta_{1,2}^2}\beta_{2,1}  +\sqrt{\Sigma_{22} - \beta_{2,1}^2}\beta_{1,2} = \Sigma_{12},
\end{align}
and
\begin{align*}
\beta_{1,2}=\frac{1}{2C_\delta}\left( \sqrt{2\Sigma_{11}C_\delta^2 - (\delta_1^{\mathcal{E}})^2}-\delta_1^{\mathcal{E}}\right)\quad\text{and}\quad \beta_{2,1} \in \frac{1}{2C_\delta}\left( \sqrt{2\Sigma_{22}C_\delta^2 - (\delta_2^{\mathcal{E}})^2}-\delta_2^{\mathcal{E}}\right).
\end{align*}
The estimated constant $\hat{C}_\delta$ with corresponding model error coefficients, $\hat{\beta}$, and final restriction values (see Eq.\,\eqref{eq:beta_0_restrictions}) are given in Table\,\ref{tab:beta_fit_results}. Also the relative error rates between $\delta_{1,2}^{\mathcal{E}}$ and $\hat{\delta}_{1,2}^{\mathcal{E}}$ are given, see Eq.\,\eqref{eq:parameters_delta_restrictions}. The empirically computed results confirm that the linear relationship between the distributions of $\bs{\sigma}(t)$-scaled model residuals of stratospheric temperature and U wind are approximately independent.

\begin{table}[hbt!]
\centering
\caption{Fitted idiosyncratic error distribution scale parameter, $\hat{C}_\delta$, with computed model error coefficients, $\hat{\beta}$, corresponding restriction values, and an error result. \label{tab:beta_fit_results}}
\begin{tabular}{ccccc}
\toprule
$\hat{C}_\delta$ & $\hat{\beta}$ & $\delta_{1}^{\mathcal{E}}/\sqrt{2\Sigma_{11}},\delta_{2}^{\mathcal{E}}/\sqrt{2\Sigma_{22}}$ & $\hat{\beta}_{1,2}^2,\hat{\beta}_{2,1}^2$ & $\abs{\hat{\delta}_{1}^{\mathcal{E}}-\delta_{1}^{\mathcal{E}}}/\delta_{1}^{\mathcal{E}},\abs{\hat{\delta}_{2}^{\mathcal{E}}-\delta_{2}^{\mathcal{E}}}/\delta_{2}^{\mathcal{E}}$ \\
\midrule
$1.709$ & $\begin{bmatrix}1.009 & 0.01174 \\ -0.03310 & 1.002\end{bmatrix}$ & $1.194,1.247$ & $0.0001382,0.001096$ & $0.02358,1.255\cdot 10^{-16}$ \\
\bottomrule  
\end{tabular}
\end{table}


\section{\label{sect:conslusions}Conclusions}
A transformation relation giving a direct relationship between discrete time VARMA processes and continuous time MCARMA processes is derived. The transformation relation is a potential method in future model estimation of MCARMA processes. Convergence results are given as a tool to substantiate validity of the transformation. A demonstration of applying the transformation relation in model estimation of an MCAR model is given through a case study. \\

This work leans on derivations in \cite{marquardt07}, concluding that VARMA processes are the proper discrete time analogue of MCARMA processes. Future work could investigate if the VARMA process remains such proper analogue to the MCARMA process when the framework is extended to include additive seasonality and heteroscedasticity. Further, to develop methods for estimating model error coefficients for an MCAR model with more than two driving L{\'e}vy processes is an important topic. In this work, the transformation relation is demonstrated by fitting an MCAR model to data. A next step would be to investigate the possibility of estimating driving processes from MCARMA models, which includes estimating lagged versions of L{\'e}vy processes. \\


\begin{appendices}


\section{Proof Proposition 1}
\label{app:1}


In this proof, we will see that $Q_i^{(l)}$ corresponds to an SDE number (see Table\,\ref{tab:index_collection_dependence}). The statement $Q_i^{(l)}\in\{\mathcal{C}^{S},\mathcal{C}^{R},\mathcal{C}^{AR}\}$ refer to that the SDE with SDE number $Q_i^{(l)}$ is in one of the collections $\{\mathcal{C}^{S},\mathcal{C}^{R},\mathcal{C}^{AR}\}$. See Eq.\,\eqref{eq:the_recursive_parameter} for definition of the recursive parameter $Q_i^{(l)}$. The result is derived in four stages:\\


\begin{enumerate}
    \item The case $l=p$: Consider first $Q_1^{(p)}$ (i=1). We see that 
\begin{align*}
    Q_1^{(p)} =  Q_2^{(p)} + d \in \mathcal{C}^{AR} \quad\leftrightarrow\quad Q_2^{(p)} \in \mathcal{C}^{R}.
\end{align*}
This means that, by Eq.\,\eqref{eq:sde_recursive_equations}, Eq.\,\eqref{eq:elementwise_p_and_q_terms} and Table\,\ref{tab:index_collection_dependence}, we might write
    \begin{align*}
     X_{Q_2^{(p)}+d}(t)dt = dX_{Q_2^{(p)}}(t) - \sum_{r=1}^{m}\beta_{(k-1)m + r}^{(1)}dL_r(t).
\end{align*}
Similarly,
\begin{align*}
    Q_2^{(p)}=Q_3^{(p)} + d\in \mathcal{C}^{R} \quad\leftrightarrow\quad Q_3^{(p)} \in \mathcal{C}^{R},
\end{align*}
such that, by Eq.\,\eqref{eq:sde_recursive_equations}, Eq.\,\eqref{eq:elementwise_p_and_q_terms} and Table\,\ref{tab:index_collection_dependence}, we might write
\begin{align*}
     X_{Q_3^{(p)}+d}(t)dt = dX_{Q_3^{(p)}}(t) - \sum_{r=1}^{m}\beta_{(k-1)m+r}^{(2)}dL_r(t).
\end{align*}
Continuing in a similar fashion for SDE numbers in the collection $\mathcal{C}^R$, we finally hit 
\begin{align*}
    X_{Q_{p-1}^{(p)}+d}(t)dt = dX_{Q_{p-1}^{(p)}}(t) - \sum_{r=1}^{m}\beta_{(k-1)m+r}^{(p-2)}dL_r(t),
\end{align*}
and since 
\begin{align*}
    Q_{p-1}^{(p)}=Q_{p}^{(p)}+d\in \mathcal{C}^R  \quad\leftrightarrow\quad Q_{p}^{(p)}\in \mathcal{C}^{S},
\end{align*}
this finally gives 
\begin{align*}
    X_{Q_{p}^{(p)}+d}(t)dt = dX_{Q_{p}^{(p)}}(t) - \sum_{r=1}^{m}\beta_{(k-1)m+r}^{(p-1)}dL_r(t),
\end{align*}
by Eq.\,\eqref{eq:sde_recursive_equations}, Eq.\,\eqref{eq:elementwise_p_and_q_terms} and Table\,\ref{tab:index_collection_dependence}. This proves that 
\begin{align*}
     X_{Q_i^{(p)}+d}(t)dt = dX_{Q_i^{(p)}}(t) - \sum_{r=1}^{m}\beta_{(k-1)m+r}^{(p-p+i-1)}dL_r(t),
\end{align*}
holds for $2\leq i \leq p$.

\item The case $l=p-1$: Consider $Q_1^{(p-1)}$ and $Q_2^{(p-1)}$ ($i=1$ and $i=2$ respectively). We see that 
\begin{align*}
Q_1^{(p-1)}\in \mathcal{C}^{R} \quad \leftrightarrow\quad Q_1^{(p-1)} + d\in  \mathcal{C}^{AR},
\end{align*}
and that
\begin{align*}
    Q_1^{(p-1)}=Q_2^{(p-1)} + d\in \mathcal{C}^R  \quad\leftrightarrow\quad Q_2^{(p-1)} \in \mathcal{C}^R.
\end{align*}
By Eq.\,\eqref{eq:sde_recursive_equations}, Eq.\,\eqref{eq:elementwise_p_and_q_terms} and Table\,\ref{tab:index_collection_dependence} we thus might write
\begin{align*}
     X_{Q_1^{(p-1)}+d}(t)dt = dX_{Q_1^{(p-1)}}(t) - \sum_{r=1}^{m}\beta_{(k-1)m+r}^{(1)}dL_r(t).
\end{align*}
and
\begin{align*}
     X_{Q_2^{(p-1)}+d}(t)dt = dX_{Q_2^{(p-1)}}(t) - \sum_{r=1}^{m}\beta_{(k-1)m+r}^{(2)}dL_r(t).
\end{align*}
Continuing in a similar fashion for SDE numbers in the collection $\mathcal{C}^R$, we finally hit
\begin{align*}
    X_{Q_{p-2}^{(p-1)}+d}(t)dt = dX_{Q_{p-2}^{(p-1)}}(t) - \sum_{r=1}^{m}\beta_{(k-1)m+r}^{(p-2)}dL_r(t),
\end{align*}
and since 
\begin{align*}
    Q_{p-2}^{(p-1)} = Q_{p-1}^{(p-1)}+d\in \mathcal{C}^R  \quad\leftrightarrow\quad Q_{p-1}^{(p-1)}\in \mathcal{C}^S,
\end{align*}
this finally gives
\begin{align*}
    X_{Q_{p-1}^{(p-1)}+d}(t)dt = dX_{Q_{p-1}^{(p-1)}}(t) - \sum_{r=1}^{m}\beta_{(k-1)m+r}^{(p-1)}dL_r(t),
\end{align*}
by Eq.\,\eqref{eq:sde_recursive_equations}, Eq.\,\eqref{eq:elementwise_p_and_q_terms} and Table\,\ref{tab:index_collection_dependence}. This proves that
\begin{align*}
     X_{Q_i^{(p-1)}+d}(t)dt = dX_{Q_i^{(p-1)}}(t) - \sum_{r=1}^{m}\beta_{(k-1)m+r}^{(p-(p-1)+i-1)}dL_r(t),
\end{align*}
holds for $1\leq i \leq p-1$.\\

    \item The cases $l\in\{2,\ldots ,p-2\}$: Continuing in a similar fashion as in points 1. and 2. for $Q_i^{(l)}$, $l\in\{2,\ldots ,p-2\}$, with $Q_1^{(l)}$ as the starting point for each $l$, we realize that 
\begin{align}
    \label{app:sde_by_recursive_parameter}
     X_{Q_i^{(l)}+d}(t)dt = dX_{Q_i^{(l)}}(t) - \sum_{r=1}^{m}\beta_{(k-1)m+r}^{(p-l+i-1)}dL_r(t),
\end{align}
holds for for all $l\in\{2,\ldots ,p\}$ and $1\leq i\leq l$ ($i>1$ when $l=p$). 
    
    \item The case $l=1$: The only valid case for $l=1$ is when $i=1$. By Eq.\,\eqref{eq:sde_recursive_equations}, Eq.\,\eqref{eq:elementwise_p_and_q_terms} and Table\,\ref{tab:index_collection_dependence} this trivially gives
    \begin{align*}
     X_{Q_1^{(1)}+d}(t)dt = dX_{Q_1^{(1)}}(t) - \sum_{r=1}^{m}\beta_{(k-1)m+r}^{(p-1)}dL_r(t),
     \end{align*}
     which satisfy Eq.\,\eqref{app:sde_by_recursive_parameter} as well.
\end{enumerate}
This concludes the proof.


\section{Proof Lemma 1}
\label{app:2}


The definition of the recursive parameter, $Q_i^{(l)}$, is used throughout this proof, see Eq.\,\ref{eq:the_recursive_parameter}. The proof is performed in four stages:
\begin{enumerate}
    \item The backwards recursive procedure from $x_{Q_1^{(l)}}$ (SDEs in $\mathcal{C}^R$) to the solution SDEs, $x_{Q_l^{(l)}}(t)=x_k(t)$, in $\mathcal{C}^S$ goes as follows. We know from Eq.\,\eqref{eq:q_stochastic_diff_eq_discretization} and Proposition\,\ref{prop:q_stochastic_diff_eq} that 
\begin{align}
    \label{app_eq:discretized_formula}
    x_{Q_i^{(l)}+d}(t) = \frac{1}{h}\left(x_{Q_i^{(l)}}(t+h) - x_{Q_i^{(l)}}(t) - \sum_{r=1}^m\beta_{(k-1)m+r}^{(p-l+i-1)}\Delta L_r (t)\right),
\end{align}
for $1\leq l \leq p$ and $1\leq i\leq l$ ($i>1$ when $l=p$). Thus, for the case $1\leq l < p$, the backwards recursive procedure starting point ($i=1$) is given by
\begin{align}
    \label{app_eq:base_equations}
    x_{Q_1^{(l)} + d}(t) = \frac{1}{h}\left( x_{Q_1^{(l)}}(t + h) - x_{Q_1^{(l)}}(t) - \sum_{r=1}^{m}\beta_{(k-1)m+r}^{(p-l)}\Delta L_r(t) \right).
\end{align}
By definition of $Q_i^{(l)}$ we have that $x_{Q_1^{(l)}}(t) = x_{Q_2^{(l)} + d}(t)$, and by Eq.\,\eqref{app_eq:discretized_formula} we might write 
\begin{align}
    \label{app_eq:base_equations_plus_1}
    x_{Q_2^{(l)} + d}(t) = \frac{1}{h}\left( x_{Q_2^{(l)}}(t + h) - x_{Q_2^{(l)}}(t) - \sum_{r=1}^{m}\beta_{(k-1)m+r}^{(p-l+1)}\Delta L_r(t) \right).
\end{align}
That is, combining Eq.\,\eqref{app_eq:base_equations} and \eqref{app_eq:base_equations_plus_1}, we find
\begin{align}
    \begin{split}
    \label{app_eq:first_recursive_step_equation}
    x_{Q_1^{(l)} + d} =& \frac{1}{h}\Bigg\{ \frac{1}{h}\left(x_{Q_2^{(l)}}(t+2h) - x_{Q_2^{(l)}}(t+h)- \sum_{r=1}^{m}\beta_{(k-1)m+r}^{(p-l+1)}\Delta L_r(t+h)\right) \\
     & \qquad - \frac{1}{h}\left(x_{Q_2^{(l)}}(t+h) - x_{Q_2^{(l)}}(t)- \sum_{r=1}^{m}\beta_{(k-1)m+r}^{(p-l+1)}\Delta L_r(t)\right) - \sum_{r=1}^{m}\beta_{(k-1)m+r}^{(p-l)}\Delta L_r(t)\Bigg\} \\
     =& \frac{1}{h^2}\Bigg( x_{Q_2^{(l)}}(t+2h) - 2x_{Q_2^{(l)}}(t+h) + x_{Q_2^{(l)}}(t) \Bigg)\\
     &- \frac{1}{h^2}\sum_{r=1}^{m}\beta_{(k-1)m+r}^{(p-l+1)}\left(\Delta L_r(t+h) - \Delta L_r(t)\right) -\frac{1}{h}\sum_{r=1}^{m}\beta_{(k-1)m+r}^{(p-l)}\Delta L_r(t).
     \end{split}
\end{align}
Further, by definition of $Q_i^{(l)}$ we have that $x_{Q_2^{(l)}}(t) = x_{Q_3^{(l)} + d}(t)$, and by Eq.\,\eqref{app_eq:discretized_formula} we might write
\begin{align}
    \label{app_eq:base_equations_plus_2}
    x_{Q_3^{(l)} + d}(t) = \frac{1}{h}\left( x_{Q_3^{(l)}}(t + h) - x_{Q_3^{(l)}}(t) - \sum_{r=1}^{m}\beta_{(k-1)m+r}^{(p-l+2)}\Delta L_r(t) \right).
\end{align}
That is, $x_{Q_1^{(l)}+d}(t)$ in Eq.\,\eqref{app_eq:first_recursive_step_equation} might be rewritten as a function of $x_{Q_3^{(l)}}(t+3h),\ldots ,x_{Q_3^{(l)}}(t),$ $\Delta L_r(t+2h),\Delta L_r(t+h),\Delta L_r(t)$ using Eq.\,\eqref{app_eq:base_equations_plus_2}. Continuing in a similar fashion for $x_{Q_3^{(l)}} = x_{Q_4^{(l)}+d}$, and so on, to recursively rewrite $x_{Q_1^{(1)}+d}$, we finally reach a right hand side of Eq.\,\eqref{app_eq:first_recursive_step_equation} which is given by $x_{Q_l^{(l)}}(t+lh),\ldots ,x_{Q_l^{(l)}}(t),\Delta L_r(t+(l-1)h),\ldots ,\Delta L_r(t)$.

    \item The backwards recursive procedure in point 1. is a analogous to the $1$-dimensional AR/CAR transformation relation formula, see \cite{benth2008}. From \cite{benth2008} (Lemma 10.2) we know that when a discrete variable $x_{\tilde{q}+1}(t)$ ($\tilde{q}\in\mathbb{N}$) satisfies
\begin{align}
    \label{app_eq:1d_discretized_formula}
    x_{\tilde{q}+1}(t) = x_{\tilde{q}}(t+1) - x_{\tilde{q}}(t),\quad 1\leq \tilde{q} \leq p-1,
\end{align}
it holds that 
\begin{align}
    \label{app_eq:benth2008_recursive_formula}
    x_{\tilde{q}+1}(t) = \sum_{n=0}^{\tilde{q}}(-1)^nb_n^{\tilde{q}}x_1(t+\tilde{q}-n),
\end{align}
where the coefficients $b_n^{\tilde{q}}$ ($n\in\mathbb{N}$) are defined recursively as 
\begin{align}
    \label{eq:pascal_triangle}
    b_n^{\tilde{q}} = b_{n-1}^{\tilde{q}-1},\quad 1\leq n \leq p-1,\tilde{q}\geq 2,
\end{align}
and $b_0^{\tilde{q}}=b_{\tilde{q}}^{\tilde{q}}=1$ for $0\leq \tilde{q} \leq p$. As shown in point 1., our goal is to move recursively backwards from $x_{Q_1^{(l)}}(t)$ to $x_{Q_l^{(l)}}(t)=x_k(t)$ ($k\in\{1,\ldots ,d\}$), just as the formula in Eq.\,\eqref{app_eq:benth2008_recursive_formula} achieves for the $1$-dimensional AR/CAR case (however from $x_{\tilde{q}}$ to $x_1$ in that specific case). In our multivariate case, when $1\leq l < p$, the discrete variable $x_{Q_i^{(l)}+d}(t)$ satisfies Eq.\,\eqref{app_eq:discretized_formula} for all $1\leq i \leq l$. Notice that the deterministic part of the right hand side in Eq.\,\eqref{app_eq:discretized_formula} is equivalent to the relationship of the 1-dimensional case in Eq.\,\eqref{app_eq:1d_discretized_formula}. By this observation, we see that the formula in Eq.\,\eqref{app_eq:benth2008_recursive_formula} might be modified to hold for the deterministic part of the $d$-dimensional VARMA/MCARMA case. This is the goal of point 3.

    \item The formula in Eq.\,\eqref{app_eq:benth2008_recursive_formula} will be modified according to the deterministic part of the right hand side, $1/h(x_{Q_i^{l}}(t+h) - x_{Q_i^{(l)}}(t))$, of Eq.\,\eqref{app_eq:discretized_formula}. Consider the four arguments that follows.
    \begin{enumerate}
    
        \item As mentioned in point 2., the final goal of the current multivariate case is to reach $x_{Q_l^{(l)}}(t)=x_k(t)$ ($k\in\{1,\ldots ,d\}$), meaning that the right hand side of Eq.\,\eqref{app_eq:benth2008_recursive_formula} should be stated as 
        \begin{align}
            \label{app_eq:deterministic_recursive_procedure_formula_step_1}
            \sum_{n=0}^{\tilde{q}}(-1)^nb_n^{\tilde{q}}x_k(t+\tilde{q}-n).
        \end{align}
        
    \item In the $1$-dimensional case, the backwards recursive procedure moves from variable $x_{\tilde{q}}(t)$ to $x_1(t)$ by inserting Eq.\,\eqref{app_eq:1d_discretized_formula} repeatedly (as shown in point 1. for the current multivariate case). That is, the variable index number is reduced by $\tilde{q}-1$, and the formula requires $\tilde{q}+1$ terms ($n$ running from $0$ to $\tilde{q}$) in order to do that procedure. In the $d$-dimensional case, the backwards recursive procedure moves from variable $x_{Q_{1}^{(l)}}$ to $x_{Q_{l}^{(l)}}$ in the same way as for the $1$-dimensional case, with a variable index number increased by $l-1$ in the current case (notice that the indexes altered in the 1- and d-dimensional recursive procedures are defined in opposite order). That is, the formula in Eq.\,\eqref{app_eq:deterministic_recursive_procedure_formula_step_1} requires $l+1$ terms ($n$ running from $0$ to $l$), and the updated formula becomes 
        \begin{align}         
            \label{app_eq:deterministic_recursive_procedure_formula_step_2}
            \sum_{n=0}^{l}(-1)^nb_n^{l}x_k(t+l-n).
        \end{align}

    \item In the current multivariate case, an Euler scheme step size $h$ is considered rather than step size equal to $1$, as in the 1-dimensional case. As a consequence, for each step in the backwards recursive procedure (see point 1.), the substituted expression adds an extra factor $1/h$ to the formula in Eq.\,\eqref{app_eq:deterministic_recursive_procedure_formula_step_2}. The starting point of the recursive procedure (see Eq.\,\eqref{app_eq:base_equations}) gives the first factor $1/h$, and each of the $l-1$ recursive steps adds one additional factor. That is, the right hand side of Eq.\,\eqref{app_eq:deterministic_recursive_procedure_formula_step_2} has to be modified as 
    \begin{align}
        \label{app_eq:deterministic_recursive_procedure_formula_step_3}
        \frac{1}{h^l}\sum_{n=0}^{l}(-1)^nb_n^{l}x_k(t+l-n).
    \end{align}
    
    \item An additional consequence of setting the Euler scheme step size to $h$ is that, when the Euler scheme has moved $r$ steps forward, the length of each step is equal to $rh$. That is, the following modification of the right hand side of Eq.\,\eqref{app_eq:benth2008_recursive_formula} has to be done
    \begin{align}
        \label{app_eq:deterministic_recursive_procedure_formula_step_4}
        \frac{1}{h^l}\sum_{n=0}^{l}(-1)^nb_n^{l}x_k(t+(l-n)h).
    \end{align}
    \end{enumerate}
    This concludes the deterministic part of the formula we want to prove. 
    
    \item The result of the backwards recursive procedure for the stochastic part of the right hand side of Eq.\,\eqref{app_eq:discretized_formula} has to be added to the formula in Eq.\,\eqref{app_eq:deterministic_recursive_procedure_formula_step_4}. The illustration of the recursive procedure in point 1. shows that, for each recursive step, incremental L{\'e}vy-terms are added. Continuing the explicit calculations in point 1. by one recursive step adds the stochastic term
    \begin{align*}
        -\frac{1}{h^3}\sum_{r=1}^{m}\beta_{(k-1)m+r}^{(p-j+2)}\left(\Delta L_r(t+2h) -2\Delta L_r(t+h) + \Delta L_r(t)\right),
    \end{align*}
    to the recursive procedure of $x_{Q_1^{(l)}+d}(t)$. It is evident that the recursively added incremental L\'{e}vy-terms follow a similar pattern as the recursively added deterministic terms, however with a delay of one step. Also note that each of the added L\'{e}vy-terms remains in the final formula, rather than being updated at each recursive step as for the deterministic part. By these remarks, the formula for the added incremental L\'{e}vy-terms is found through the following arguments:
    
    \begin{enumerate}
        \item As the recursively added incremental L{\'e}vy-terms follow a similar pattern as the recursively added deterministic terms, we use the formula in Eq.\,\eqref{app_eq:deterministic_recursive_procedure_formula_step_4} as a starting point. As seen in point 1., the variable in the stochastic case is $\sum_{r=1}^m\beta_{(k-1)m+r}^{(p-l+i-1)}\Delta L_r (t)$ rather than $x_k(t)$. That is, the first step towards the recursive formula of the stochastic part on the right hand side of Eq.\,\eqref{app_eq:base_equations} is
        \begin{align}
            \label{app_eq:stochastic_recursive_procedure_formula_step_1}
            \frac{1}{h^l}\sum_{r=1}^m\beta_{(k-1)m+r}^{(p-l+i-1)}\sum_{v=0}^{l}(-1)^vb_v^{l}\Delta L_r(t+(l-v)h).
        \end{align}
        \item The incremental L\'{e}vy-terms are not regressed on in the recursive procedure. As seen in point 1., each added L{\'e}vy-term originate from the recursive steps on $1/h(x_{Q_i^{(l)}}(t+h) - x_{Q_i^{(l)}}(t))$ ($1\leq i < l$). This is why each recursively added stochastic term remains in the final expression. That is, each of the $l-1$ added incremental L\'{e}vy-terms during the recursive procedure are given by the recursive formula in Eq.\,\eqref{app_eq:stochastic_recursive_procedure_formula_step_1}, and the stochastic recursive procedure formula has to be modified as 
        \begin{align}
            \label{app_eq:stochastic_recursive_procedure_formula_step_2}
            \sum_{w=0}^{l-1}\frac{1}{h^{l}}\sum_{r=1}^m\beta_{(k-1)m+r}^{(p-l+i-1)}\sum_{v=0}^{w}(-1)^vb_v^{w}\Delta L_r(t+(w-v)h).
        \end{align}
        
        \item In point 1. we see that each of the added incremental L\'{e}vy-terms brings one extra factor $1/h$ for each recursive step. The index $w$ in Eq.\,\eqref{app_eq:stochastic_recursive_procedure_formula_step_2} might account for this by modifying the formula as 
        \begin{align}
            \label{app_eq:stochastic_recursive_procedure_formula_step_3}
            \sum_{w=0}^{l-1}\frac{1}{h^{w+1}}\sum_{r=1}^m\beta_{(k-1)m+r}^{(p-l+i-1)}\sum_{v=0}^{w}(-1)^vb_v^{w}\Delta L_r(t+(w-v)h).
        \end{align}
        
        \item Finally, from the calculations in point 1. we see that the top index of $\beta$ increases from $p-l$ by one for each added term. The index $w$ in Eq.\,\eqref{app_eq:stochastic_recursive_procedure_formula_step_3} might be used to account for this as well, and we modify the formula as
        \begin{align}
            \label{app_eq:stochastic_recursive_procedure_formula_step_4}
            \sum_{w=0}^{l-1}\frac{1}{h^{w+1}}\sum_{r=1}^m\beta_{r + (k-1)m}^{(p-l+w)}\sum_{v=0}^{w}(-1)^vb_v^{w}\Delta L_r(t+(w-v)h).
        \end{align}
    \end{enumerate}
    
Add the deterministic and stochastic recursive formulas in Eq.\,\eqref{app_eq:deterministic_recursive_procedure_formula_step_4} and \eqref{app_eq:stochastic_recursive_procedure_formula_step_4}, respectively, according to the formula in Eq.\,\eqref{app_eq:base_equations} with regards to signs, to find the backwards recursive formula of $x_{Q_1^{(l)}+d}(t)$ as stated in the lemma, and the proof is complete.
\end{enumerate}


\section{Proof Theorem 1}
\label{app:3}


Combine Eq.\,\eqref{eq:elementwise_p_and_q_terms} and the recursive parameter in Eq.\,\eqref{eq:the_recursive_parameter} to rewrite the SDEs in $\mathcal{C}^{AR}$ (Eq.\,\eqref{eq:sde_lag_dependence_structure}) as the piecewise constant process 
\begin{align}
    \label{eq:theorem_proof_eq_1}
    x_{Q_1^{(p)}}(t+h) - x_{Q_1^{(p)}}(t) = -h\sum_{l=1}^{p}\sum_{s=1}^{d}\alpha^{(p-l+1)} x_{Q_1^{(l)}}(t) + \sum_{r=1}^{m}\beta^{(0)}\Delta L_r(t),
\end{align}
using the Euler scheme (see \cite{kloeden92} and \cite{protter97}). Here, we have defined 
\begin{align*}
\alpha^{(p-l+1)}\triangleq \alpha_{(k-1)d+s}^{(p-l+1)} \quad\text{and}\quad \beta^{(\kappa)} \triangleq \beta^{(\kappa)}_{(k-1)m+r}    
\end{align*}
for notational convenience. Note that we sum $Q_1^{(l)}\triangleq (Q_1^{(l)}\mid s)$ over $s$ as well. Extract the terms from the matrix-vector product $[-A_p\cdots -A_1]\bs{X}(t)dt$ corresponding to equation block $1$ (variables $x_{Q_1^{(1)}}=x_k$), such that we can write $Q_1^{(l)}$ as $Q_1^{(l-1)}+d$. That is, rewrite Eq.\,\eqref{eq:theorem_proof_eq_1} as 
\begin{align}
    \label{eq:theorem_proof_eq_2}
    x_{Q_1^{(p-1)}+d}(t+h) - x_{Q_1^{(p-1)}+d}(t) =& -h\sum_{s=1}^{d}\alpha^{(p)}x_{s}(t)+ \sum_{r=1}^{m}\beta^{(0)}\Delta L_r(t)-h\sum_{l=2}^{p}\sum_{s=1}^{d}\alpha^{(p-l+1)} x_{Q_1^{(l-1)}+d}(t).
\end{align}
By Lemma \ref{lemma:euler_approx_base_equations}, the left hand side of Eq.\,\eqref{eq:theorem_proof_eq_2} can be written as
\begin{align*}
    &\frac{1}{h^{p-1}}\sum_{n=0}^{p-1}(-1)^n b_n^{p-1} \left(x_k(t + (p-n)h) - x_k(t + (p-1-n)h)\right) \\
    & - \sum_{w=0}^{p-2}\frac{1}{h^{w+1}}\sum_{r=1}^{m}\beta^{(w+1)}\sum_{v=0}^{w}(-1)^vb_v^w\left(\Delta L_r(t+(w-v+1)h) - \Delta L_r(t+(w-v)h)\right),
\end{align*}
and the right hand side as
\begin{align*}
    &-h\sum_{s=1}^{d}\alpha^{(p)}x_{s}(t)+ \sum_{r=1}^{m}\beta^{(0)}\Delta L_r(t) \\ &-h\sum_{l=2}^{p}\sum_{s=1}^{d}\alpha^{(p-l+1)}\Bigg(\frac{1}{h^{l-1}}\sum_{n=0}^{l-1}(-1)^n b_n^{l-1} x_s(t + (l-1-n)h)  \\ 
    &\qquad\qquad\qquad\qquad\qquad - \sum_{w=0}^{l-2}\frac{1}{h^{w+1}}\sum_{r=1}^{m}\beta^{(p-l+w+1)}\sum_{v=0}^{w}(-1)^vb_v^w\Delta L_r(t+(w-v)h)\Bigg).
\end{align*}
Remember that $k=s$ in the subscript of $\beta^{(p-l+w+1)}$. That is, Eq.\,\eqref{eq:theorem_proof_eq_2} can be rearranged as
\begin{align*}
&\frac{1}{h^{p-1}}\left(x_k(t+ph) - x_k(t+(p-1)h)\right) = -h\sum_{s=1}^{d}\alpha^{(p)}x_{s}(t)+ \sum_{r=1}^{m}\beta^{(0)}\Delta L_r(t) \\
    &-h\sum_{l=2}^{p}\sum_{s=1}^{d}\alpha^{(p-l+1)}\Bigg(\frac{1}{h^{l-1}}\sum_{n=0}^{l-1}(-1)^n b_n^{l-1} x_s(t + (l-1-n)h)  \\ 
    &\qquad\qquad\qquad\qquad\qquad - \sum_{w=0}^{l-2}\frac{1}{h^{w+1}}\sum_{r=1}^{m}\beta^{(p-l+w+1)}\sum_{v=0}^{w}(-1)^vb_v^w\Delta L_r(t+(w-v)h)\Bigg) \\
    &-\frac{1}{h^{p-1}}\sum_{n=1}^{p-1}(-1)^n b_n^{p-1} \left(x_k(t + (p-n)h) - x_k(t + (p-1-n)h)\right) \\
    & + \sum_{w=0}^{p-2}\frac{1}{h^{w+1}}\sum_{r=1}^{m}\beta^{(w+1)}\sum_{v=0}^{w}(-1)^vb_v^w\left(\Delta L_r(t+(w-v+1)h) - \Delta L_r(t+(w-v)h)\right).    
\end{align*}
where the highest lagged variable, meaning $x_k(t+ph)$, is extracted, and $b_0^{p-1}=1$ is inserted (see Eq.\,\eqref{app_eq:benth2008_recursive_formula}). Solve for $x_k(t+ph)$ to complete the proof.


\section{Computations to find MCAR model coefficients of deterministic part\label{app:5}}


Compare Eq.\,\eqref{eq:MCAR/VAR_transformation_example} and \eqref{eq:var_4_process_example} to see that 
\begin{align*}
    \begin{cases}
    -\alpha_{1,1}^{(1)} - 1 = \phi_{11}^{(1)} \\ 
    -\alpha_{1,2}^{(1)} = \phi_{12}^{(1)}
    \end{cases}\quad\Rightarrow\quad
    \begin{cases}
    \alpha_{1,1}^{(1)} = -\phi_{11}^{(1)} - 1 \\
    \alpha_{1,2}^{(1)} = -\phi_{12}^{(1)}
    \end{cases}
\end{align*}
\begin{align*}
    \begin{cases}
    -\alpha_{1,1}^{(2)} + 3\alpha_{1,1}^{(1)} + 4 = \phi_{11}^{(2)} \\
    -\alpha_{1,2}^{(2)} + 3\alpha_{1,2}^{(1)} = \phi_{12}^{(2)}
    \end{cases}\quad\Rightarrow\quad
    \begin{cases}
    \alpha_{1,1}^{(2)} = -3\phi_{11}^{(1)} - \phi_{11}^{(2)} + 1 \\
    \alpha_{1,2}^{(2)} = -3\phi_{12}^{(1)} - \phi_{12}^{(2)}
    \end{cases}
\end{align*}
\begin{align*}
    \begin{cases}
    -\alpha_{1,1}^{(3)} + 2\alpha_{1,1}^{(2)} - 3\alpha_{1,1}^{(1)} - 6 = \phi_{11}^{(3)} \\
    -\alpha_{1,2}^{(3)} + 2\alpha_{1,2}^{(2)} - 3\alpha_{1,2}^{(1)} = \phi_{12}^{(3)} \\
    \end{cases}\quad\Rightarrow\quad
    \begin{cases}
    \alpha_{1,1}^{(3)} = -3\phi_{11}^{(1)}-2\phi_{11}^{(2)} - \phi_{11}^{(3)} - 1 \\
    \alpha_{1,2}^{(3)} = -3\phi_{12}^{(1)}-2\phi_{12}^{(2)} - \phi_{11}^{(3)} \\
    \end{cases}
\end{align*}
\begin{align*}
    \begin{cases}
    -\alpha_{1,1}^{(4)} + \alpha_{1,1}^{(3)} - \alpha_{1,1}^{(2)} + \alpha_{1,1}^{(1)} + 4 = \phi_{11}^{(4)} \\
    -\alpha_{1,2}^{(4)} + \alpha_{1,2}^{(3)} - \alpha_{1,2}^{(2)} + \alpha_{1,2}^{(1)} = \phi_{12}^{(4)}
    \end{cases}\quad\Rightarrow\quad
    \begin{cases}
    \alpha_{1,1}^{(4)} = -\phi_{11}^{(1)}-\phi_{11}^{(2)}-\phi_{11}^{(3)}-\phi_{11}^{(4)}+1 \\
    \alpha_{1,2}^{(4)} = -\phi_{12}^{(1)}-\phi_{12}^{(2)}-\phi_{12}^{(3)}-\phi_{12}^{(4)}.
    \end{cases}
\end{align*}
In the above solutions, change subscripts $1,1$ and $11$ with $2,2$ and $22$, respectively, to find the solutions of $\alpha_{2,2}^{(1)},\ldots ,\alpha_{2,2}^{(4)}$. Further, change subscripts $1,2$ and $12$ with $2,1$ and $21$, respectively, to find the solutions of $\alpha_{2,1}^{(1)},\ldots ,\alpha_{2,1}^{(4)}$. Now, by the definition $\alpha_{k,i_s}^{(j)}\triangleq \alpha_{(k-1)d + i_s}^{(p-l+1)}$, we find the matrices $A_p,\ldots ,A_1$ (see Eq.\,\eqref{eq:matrix_element_definition}).


\section{Computations to find MCAR model coefficients of stochastic part\label{app:4}}


The system of equations in Eq.\,\eqref{eq:cov_of_epsilon} and \eqref{eq:parameters_delta_restrictions} is given by
\begin{align}
    \label{eq:beta_proof_eq0}
    \begin{cases}
    \delta_1^{\mathcal{E}} = \abs{\beta_{1,1}}C_{\delta} + \abs{\beta_{1,2}}C_{\delta}  \\
    \delta_2^{\mathcal{E}} = \abs{\beta_{2,1}}C_{\delta} + \abs{\beta_{2,2}}C_{\delta}
    \end{cases}\quad\text{and}\quad \begin{bmatrix}
\Sigma_{11}  & \Sigma_{12} \\
 \Sigma_{21} & \Sigma_{22}
    \end{bmatrix} =\begin{bmatrix}
\beta_{1,1}^2 + \beta_{1,2}^2  & \beta_{1,1}\beta_{2,1} + \beta_{1,2}\beta_{2,2} \\
 & \beta_{2,1}^2 + \beta_{2,2}^2 
    \end{bmatrix}.
\end{align}
Solve the expressions for $\delta_1^{\mathcal{E}}$, $\delta_2^{\mathcal{E}}$, $\Sigma_{11}$ and $\Sigma_{22}$ in Eq.\,\eqref{eq:beta_proof_eq0} with respect to $\beta_{1,1}$ and $\beta_{2,2}$, and equate them to find 
\begin{align}
    \label{eq:beta_proof_eq1}
    \begin{cases}
    \pm(\delta_1^{\mathcal{E}}/C_\delta - \abs{\beta_{1,2}}) = \pm\sqrt{\Sigma_{11} - \beta_{1,2}^2} \\
    \pm(\delta_2^{\mathcal{E}}/C_\delta - \abs{\beta_{2,1}}) = \pm\sqrt{\Sigma_{22} - \beta_{2,1}^2}.
    \end{cases}
\end{align}
Solve the system in Eq.\,\eqref{eq:beta_proof_eq1} with respect to $\beta_{1,2}$ and $\beta_{2,1}$. Both equations are solved in a similar way, and we therefore continue the calculations with 
\begin{align*}
\{\beta,\sigma,\delta,C\}\in\{(\beta_{1,2},\beta_{2,1}),(\Sigma_{11},\Sigma_{22}),(\delta_1^{\mathcal{E}},\delta_2^{\mathcal{E}}), (C_\delta,C_\delta)\}.    
\end{align*}
That is,
\begin{align}
    \label{eq:beta_proof_eq2}
    \begin{split}
    (\delta/C - \abs{\beta})^2 &= \sigma - \beta^2\\
    \rightarrow\quad  \beta^2 - \abs{\beta}\delta/C +\frac{1}{2}(\delta^2/C^2 - \sigma) &= 0.
    \end{split}
\end{align}
For $\beta$ positive and $\beta$ negative, Eq.\,\eqref{eq:beta_proof_eq2} gives the following respective solutions by the quadratic equation
\begin{align*}
    \beta = \frac{1}{2C}\left(\delta \pm \sqrt{2\sigma C^2 - \delta^2}\right)\quad\text{and}\quad \beta = -\frac{1}{2C}\left(\delta \mp \sqrt{2\sigma C^2 - \delta^2}\right).
\end{align*}
Finally, combine the expression of $\Sigma_{12}$ in Eq.\,\eqref{eq:beta_proof_eq0} with the right hand side of Eq.\,\eqref{eq:beta_proof_eq1} to find
\begin{align*}
    \pm\sqrt{\Sigma_{11} - \beta_{1,2}^2}\beta_{2,1}  \pm\sqrt{\Sigma_{22} - \beta_{2,1}^2}\beta_{1,2} = \Sigma_{12}.
\end{align*}


\end{appendices}

\vspace{5mm}
Supporting information. Additional information for this article is available online. Raw data used in this work are available from link in Appendix\,\ref{app:6}. 

\begin{appendices}
\renewcommand{\thesection}{S1}
\section{Data availability\label{app:6}}
https://apps.ecmwf.int/datasets/
\end{appendices}

\begin{flushleft}
\textbf{References}\\
\end{flushleft}
\bibliography{arixiv_final_version}

\begin{thebibliography}{44}%
\makeatletter
\providecommand \@ifxundefined [1]{%
 \@ifx{#1\undefined}
}%
\providecommand \@ifnum [1]{%
 \ifnum #1\expandafter \@firstoftwo
 \else \expandafter \@secondoftwo
 \fi
}%
\providecommand \@ifx [1]{%
 \ifx #1\expandafter \@firstoftwo
 \else \expandafter \@secondoftwo
 \fi
}%
\providecommand \natexlab [1]{#1}%
\providecommand \enquote  [1]{``#1''}%
\providecommand \bibnamefont  [1]{#1}%
\providecommand \bibfnamefont [1]{#1}%
\providecommand \citenamefont [1]{#1}%
\providecommand \href@noop [0]{\@secondoftwo}%
\providecommand \href [0]{\begingroup \@sanitize@url \@href}%
\providecommand \@href[1]{\@@startlink{#1}\@@href}%
\providecommand \@@href[1]{\endgroup#1\@@endlink}%
\providecommand \@sanitize@url [0]{\catcode `\\12\catcode `\$12\catcode
  `\&12\catcode `\#12\catcode `\^12\catcode `\_12\catcode `\%12\relax}%
\providecommand \@@startlink[1]{}%
\providecommand \@@endlink[0]{}%
\providecommand \url  [0]{\begingroup\@sanitize@url \@url }%
\providecommand \@url [1]{\endgroup\@href {#1}{\urlprefix }}%
\providecommand \urlprefix  [0]{URL }%
\providecommand \Eprint [0]{\href }%
\providecommand \doibase [0]{https://doi.org/}%
\providecommand \selectlanguage [0]{\@gobble}%
\providecommand \bibinfo  [0]{\@secondoftwo}%
\providecommand \bibfield  [0]{\@secondoftwo}%
\providecommand \translation [1]{[#1]}%
\providecommand \BibitemOpen [0]{}%
\providecommand \bibitemStop [0]{}%
\providecommand \bibitemNoStop [0]{.\EOS\space}%
\providecommand \EOS [0]{\spacefactor3000\relax}%
\providecommand \BibitemShut  [1]{\csname bibitem#1\endcsname}%
\let\auto@bib@innerbib\@empty
\bibitem [{\citenamefont {Applebaum}(2004)}]{applebaum04}%
  \BibitemOpen
  \bibfield  {author} {\bibinfo {author} {\bibnamefont {Applebaum},
  \bibfnamefont {D.}},\ }\href@noop {} {\emph {\bibinfo {title} {L\'{e}vy
  processes and stochastic calculus}}},\ Vol.~\bibinfo {volume} {93}\ (\bibinfo
   {publisher} {Cambridge University Press},\ \bibinfo {year}
  {2004})\BibitemShut {NoStop}%
\bibitem [{\citenamefont {Asmussen}\ and\ \citenamefont
  {Rosiński}(2001)}]{asmussen01}%
  \BibitemOpen
  \bibfield  {author} {\bibinfo {author} {\bibnamefont {Asmussen},
  \bibfnamefont {S.}}and\ \bibinfo {author} {\bibnamefont {Rosiński},
  \bibfnamefont {J.}},\ }\bibfield  {title} {\enquote {\bibinfo {title}
  {Approximations of small jumps of lévy processes with a view towards
  simulation},}\ }\href {https://doi.org/10.1239/jap/996986757} {\bibfield
  {journal} {\bibinfo  {journal} {Journal of Applied Probability}\ }\textbf
  {\bibinfo {volume} {38}},\ \bibinfo {pages} {482–--493} (\bibinfo {year}
  {2001})}\BibitemShut {NoStop}%
\bibitem [{\citenamefont {Barndorff-Nielsen}(1997)}]{barndorff97-2}%
  \BibitemOpen
  \bibfield  {author} {\bibinfo {author} {\bibnamefont {Barndorff-Nielsen},
  \bibfnamefont {O.~E.}},\ }\bibfield  {title} {\enquote {\bibinfo {title}
  {Normal inverse {Gaussian} distributions and stochastic volatility
  modelling},}\ }\href
  {https://doi.org/https://doi.org/10.1111/1467-9469.00045} {\bibfield
  {journal} {\bibinfo  {journal} {Scandinavian Journal of Statistics}\ }\textbf
  {\bibinfo {volume} {24}},\ \bibinfo {pages} {1--13} (\bibinfo {year}
  {1997})}\BibitemShut {NoStop}%
\bibitem [{\citenamefont {Barndorff-Nielsen}, \citenamefont {Mikosch},\ and\
  \citenamefont {Resnick}(2001)}]{barndorff-nielsen_book_01}%
  \BibitemOpen
  \bibfield  {author} {\bibinfo {author} {\bibnamefont {Barndorff-Nielsen},
  \bibfnamefont {O.~E.}}, \bibinfo {author} {\bibnamefont {Mikosch},
  \bibfnamefont {T.}}, and\ \bibinfo {author} {\bibnamefont {Resnick},
  \bibfnamefont {S.~I.}},\ }\href@noop {} {\emph {\bibinfo {title} {L{\'e}vy
  Processes: Theory and Applications}}}\ (\bibinfo  {publisher} {Birkh\"{a}user
  Boston: Imprint: Birkh\"{a}user},\ \bibinfo {year} {2001})\BibitemShut
  {NoStop}%
\bibitem [{\citenamefont {Barndorff-Nielsen}\ and\ \citenamefont
  {Schmiegel}(2007)}]{barndorff-nielsen07}%
  \BibitemOpen
  \bibfield  {author} {\bibinfo {author} {\bibnamefont {Barndorff-Nielsen},
  \bibfnamefont {O.~E.}}and\ \bibinfo {author} {\bibnamefont {Schmiegel},
  \bibfnamefont {J.}},\ }\bibfield  {title} {\enquote {\bibinfo {title} {Ambit
  processes; with applications to turbulence and tumour growth},}\ }in\
  \href@noop {} {\emph {\bibinfo {booktitle} {Stochastic Analysis and
  Applications}}},\ \bibinfo {editor} {edited by\ \bibinfo {editor}
  {\bibfnamefont {F.~E.}\ \bibnamefont {Benth}}, \bibinfo {editor}
  {\bibfnamefont {G.}~\bibnamefont {Di~Nunno}}, \bibinfo {editor}
  {\bibfnamefont {T.}~\bibnamefont {Lindstr{\o}m}}, \bibinfo {editor}
  {\bibfnamefont {B.}~\bibnamefont {{\O}ksendal}}, \ and\ \bibinfo {editor}
  {\bibfnamefont {T.}~\bibnamefont {Zhang}}}\ (\bibinfo  {publisher} {Springer
  Berlin Heidelberg},\ \bibinfo {address} {Berlin, Heidelberg},\ \bibinfo
  {year} {2007})\ pp.\ \bibinfo {pages} {93--124}\BibitemShut {NoStop}%
\bibitem [{\citenamefont {Benth}, \citenamefont {\v{S}altyt\.{e} Benth},\ and\
  \citenamefont {Koekebakker}(2008)}]{benth2008}%
  \BibitemOpen
  \bibfield  {author} {\bibinfo {author} {\bibnamefont {Benth}, \bibfnamefont
  {F.~E.}}, \bibinfo {author} {\bibnamefont {\v{S}altyt\.{e} Benth},
  \bibfnamefont {J.}}, and\ \bibinfo {author} {\bibnamefont {Koekebakker},
  \bibfnamefont {S.}},\ }\href@noop {} {\emph {\bibinfo {title} {Stochastic
  modelling of electricity and related markets}}}\ (\bibinfo  {publisher}
  {(Vol. 11, Advanced series on statistical science \& applied probability).
  Singapore: World Scientific Publishing Pte.},\ \bibinfo {year}
  {2008})\BibitemShut {NoStop}%
\bibitem [{\citenamefont {Benth}, \citenamefont {Di~Nunno},\ and\ \citenamefont
  {Khedher}(2011)}]{benth11}%
  \BibitemOpen
  \bibfield  {author} {\bibinfo {author} {\bibnamefont {Benth}, \bibfnamefont
  {F.~E.}}, \bibinfo {author} {\bibnamefont {Di~Nunno}, \bibfnamefont {G.}},
  and\ \bibinfo {author} {\bibnamefont {Khedher}, \bibfnamefont {A.}},\
  }\bibfield  {title} {\enquote {\bibinfo {title} {Robustness of option prices
  and their deltas in markets modelled by jump-diffusions},}\ }\href
  {https://doi.org/10.31390/cosa.5.2.03} {\bibfield  {journal} {\bibinfo
  {journal} {Communications on Stochastic Analysis}\ }\textbf {\bibinfo
  {volume} {5}} (\bibinfo {year} {2011}),\ 10.31390/cosa.5.2.03}\BibitemShut
  {NoStop}%
\bibitem [{\citenamefont {Benth}\ and\ \citenamefont {Taib}(2013)}]{benth13}%
  \BibitemOpen
  \bibfield  {author} {\bibinfo {author} {\bibnamefont {Benth}, \bibfnamefont
  {F.~E.}}and\ \bibinfo {author} {\bibnamefont {Taib}, \bibfnamefont {C.~M.
  I.~C.}},\ }\bibfield  {title} {\enquote {\bibinfo {title} {On the speed
  towards the mean for continuous time autoregressive moving average processes
  with applications to energy markets},}\ }\href
  {https://doi.org/https://doi.org/10.1016/j.eneco.2013.07.007} {\bibfield
  {journal} {\bibinfo  {journal} {Energy economics}\ }\textbf {\bibinfo
  {volume} {40}},\ \bibinfo {pages} {259--268} (\bibinfo {year}
  {2013})}\BibitemShut {NoStop}%
\bibitem [{\citenamefont {Berrisford}\ \emph {et~al.}(2011)\citenamefont
  {Berrisford}, \citenamefont {Dee}, \citenamefont {Poli}, \citenamefont
  {Brugge}, \citenamefont {Fielding}, \citenamefont {Fuentes}, \citenamefont
  {K{\r a}llberg}, \citenamefont {Kobayashi}, \citenamefont {Uppala},\ and\
  \citenamefont {Simmons}}]{ecmwf_stratospheric_temp}%
  \BibitemOpen
  \bibfield  {author} {\bibinfo {author} {\bibnamefont {Berrisford},
  \bibfnamefont {P.}}, \bibinfo {author} {\bibnamefont {Dee}, \bibfnamefont
  {D.~P.}}, \bibinfo {author} {\bibnamefont {Poli}, \bibfnamefont {P.}},
  \bibinfo {author} {\bibnamefont {Brugge}, \bibfnamefont {R.}}, \bibinfo
  {author} {\bibnamefont {Fielding}, \bibfnamefont {M.}}, \bibinfo {author}
  {\bibnamefont {Fuentes}, \bibfnamefont {M.}}, \bibinfo {author} {\bibnamefont
  {K{\r a}llberg}, \bibfnamefont {P.~W.}}, \bibinfo {author} {\bibnamefont
  {Kobayashi}, \bibfnamefont {S.}}, \bibinfo {author} {\bibnamefont {Uppala},
  \bibfnamefont {S.}}, and\ \bibinfo {author} {\bibnamefont {Simmons},
  \bibfnamefont {A.}},\ }\bibfield  {title} {\enquote {\bibinfo {title} {The
  {ERA}-{Interim} archive {Version} 2.0},}\ }\href
  {https://www.ecmwf.int/node/8174} {\bibfield  {journal} {\bibinfo  {journal}
  {ERA Report Series}\ ,\ \bibinfo {pages} {1--23}} (\bibinfo {year}
  {2011})}\BibitemShut {NoStop}%
\bibitem [{\citenamefont {Brockwell}(2014)}]{brockwell2014recent}%
  \BibitemOpen
  \bibfield  {author} {\bibinfo {author} {\bibnamefont {Brockwell},
  \bibfnamefont {P.~J.}},\ }\bibfield  {title} {\enquote {\bibinfo {title}
  {Recent results in the theory and applications of {CARMA} processes},}\
  }\href {https://doi.org/https://doi.org/10.1007/s10463-014-0468-7} {\bibfield
   {journal} {\bibinfo  {journal} {Annals of the Institute of Statistical
  Mathematics}\ }\textbf {\bibinfo {volume} {66}},\ \bibinfo {pages} {647--685}
  (\bibinfo {year} {2014})}\BibitemShut {NoStop}%
\bibitem [{\citenamefont {Brockwell}\ and\ \citenamefont
  {Davis}(1991)}]{brockwell_davis_book91}%
  \BibitemOpen
  \bibfield  {author} {\bibinfo {author} {\bibnamefont {Brockwell},
  \bibfnamefont {P.~J.}}and\ \bibinfo {author} {\bibnamefont {Davis},
  \bibfnamefont {R.~A.}},\ }\href {https://doi.org/10.1007/978-1-4419-0320-4}
  {\emph {\bibinfo {title} {Time Series: Theory and Methods}}}\ (\bibinfo
  {publisher} {Springer New York},\ \bibinfo {year} {1991})\BibitemShut
  {NoStop}%
\bibitem [{\citenamefont {Brockwell}, \citenamefont {Ferrazzano},\ and\
  \citenamefont {Kl{\"u}ppelberg}(2013)}]{brockwell_etal13_wind}%
  \BibitemOpen
  \bibfield  {author} {\bibinfo {author} {\bibnamefont {Brockwell},
  \bibfnamefont {P.~J.}}, \bibinfo {author} {\bibnamefont {Ferrazzano},
  \bibfnamefont {V.}}, and\ \bibinfo {author} {\bibnamefont {Kl{\"u}ppelberg},
  \bibfnamefont {C.}},\ }\bibfield  {title} {\enquote {\bibinfo {title}
  {High-frequency sampling and kernel estimation for continuous-time moving
  average processes},}\ }\href {https://doi.org/10.1111/jtsa.12022} {\bibfield
  {journal} {\bibinfo  {journal} {Journal of Time Series Analysis}\ }\textbf
  {\bibinfo {volume} {34}},\ \bibinfo {pages} {385--404} (\bibinfo {year}
  {2013})}\BibitemShut {NoStop}%
\bibitem [{\citenamefont {Brockwell}\ and\ \citenamefont
  {Lindner}(2019)}]{brockwell19}%
  \BibitemOpen
  \bibfield  {author} {\bibinfo {author} {\bibnamefont {Brockwell},
  \bibfnamefont {P.~J.}}and\ \bibinfo {author} {\bibnamefont {Lindner},
  \bibfnamefont {A.}},\ }\bibfield  {title} {\enquote {\bibinfo {title}
  {Sampling, embedding and inference for carma processes},}\ }\href
  {https://doi.org/10.1111/jtsa.12433} {\bibfield  {journal} {\bibinfo
  {journal} {Journal of Time Series Analysis}\ }\textbf {\bibinfo {volume}
  {40}},\ \bibinfo {pages} {163--181} (\bibinfo {year} {2019})}\BibitemShut
  {NoStop}%
\bibitem [{\citenamefont {Brockwell}\ and\ \citenamefont
  {Schlemm}(2013)}]{brockwell13}%
  \BibitemOpen
  \bibfield  {author} {\bibinfo {author} {\bibnamefont {Brockwell},
  \bibfnamefont {P.~J.}}and\ \bibinfo {author} {\bibnamefont {Schlemm},
  \bibfnamefont {E.}},\ }\bibfield  {title} {\enquote {\bibinfo {title}
  {Parametric estimation of the driving l\'{e}vy process of multivariate carma
  processes from discrete observations},}\ }\href@noop {} {\bibfield  {journal}
  {\bibinfo  {journal} {Journal of Multivariate Analysis}\ }\textbf {\bibinfo
  {volume} {115}},\ \bibinfo {pages} {217--251} (\bibinfo {year}
  {2013})}\BibitemShut {NoStop}%
\bibitem [{\citenamefont {Broszkiewicz-Suwaj}\ and\ \citenamefont
  {Wy\l{}oma\'{n}ska}(2021)}]{bs21}%
  \BibitemOpen
  \bibfield  {author} {\bibinfo {author} {\bibnamefont {Broszkiewicz-Suwaj},
  \bibfnamefont {E.}}and\ \bibinfo {author} {\bibnamefont {Wy\l{}oma\'{n}ska},
  \bibfnamefont {A.}},\ }\bibfield  {title} {\enquote {\bibinfo {title}
  {Application of non-gaussian multidimensional autoregressive model for
  climate data prediction},}\ }\href
  {https://doi.org/https://doi.org/10.1007/s12572-021-00300-1} {\bibfield
  {journal} {\bibinfo  {journal} {International Journal of Advances in
  Engineering Sciences and Applied Mathematics}\ }\textbf {\bibinfo {volume}
  {13}},\ \bibinfo {pages} {236--247} (\bibinfo {year} {2021})}\BibitemShut
  {NoStop}%
\bibitem [{\citenamefont {Chambers}\ and\ \citenamefont
  {Thornton}(2012)}]{chambers12}%
  \BibitemOpen
  \bibfield  {author} {\bibinfo {author} {\bibnamefont {Chambers},
  \bibfnamefont {M.~J.}}and\ \bibinfo {author} {\bibnamefont {Thornton},
  \bibfnamefont {M.~A.}},\ }\bibfield  {title} {\enquote {\bibinfo {title}
  {Discrete time representation of continuous time arma processes},}\ }\href
  {https://doi.org/10.2307/41426512} {\bibfield  {journal} {\bibinfo  {journal}
  {Econometric Theory}\ }\textbf {\bibinfo {volume} {28}},\ \bibinfo {pages}
  {219--238} (\bibinfo {year} {2012})}\BibitemShut {NoStop}%
\bibitem [{\citenamefont {Dee}\ \emph {et~al.}(2011)\citenamefont {Dee},
  \citenamefont {Uppala}, \citenamefont {Simmons}, \citenamefont {Berrisford},
  \citenamefont {Poli}, \citenamefont {Kobayashi}, \citenamefont {Andrae},
  \citenamefont {Balmaseda}, \citenamefont {Balsamo}, \citenamefont {Bauer},
  \citenamefont {Bechtold}, \citenamefont {Beljaars}, \citenamefont {van~de
  Berg}, \citenamefont {Bidlot}, \citenamefont {Bormann}, \citenamefont
  {Delsol}, \citenamefont {Dragani}, \citenamefont {Fuentes}, \citenamefont
  {Geer}, \citenamefont {Haimberger}, \citenamefont {Healy}, \citenamefont
  {Hersbach}, \citenamefont {Hólm}, \citenamefont {Isaksen}, \citenamefont
  {Kållberg}, \citenamefont {K{\"o}hler}, \citenamefont {Matricardi},
  \citenamefont {McNally}, \citenamefont {Monge‐Sanz}, \citenamefont
  {Morcrette}, \citenamefont {Park}, \citenamefont {Peubey}, \citenamefont
  {de~Rosnay}, \citenamefont {Tavolato}, \citenamefont {Thépaut},\ and\
  \citenamefont {Vitart}}]{dee2011}%
  \BibitemOpen
  \bibfield  {author} {\bibinfo {author} {\bibnamefont {Dee}, \bibfnamefont
  {D.~P.}}, \bibinfo {author} {\bibnamefont {Uppala}, \bibfnamefont {S.~M.}},
  \bibinfo {author} {\bibnamefont {Simmons}, \bibfnamefont {A.~J.}}, \bibinfo
  {author} {\bibnamefont {Berrisford}, \bibfnamefont {P.}}, \bibinfo {author}
  {\bibnamefont {Poli}, \bibfnamefont {P.}}, \bibinfo {author} {\bibnamefont
  {Kobayashi}, \bibfnamefont {S.}}, \bibinfo {author} {\bibnamefont {Andrae},
  \bibfnamefont {U.}}, \bibinfo {author} {\bibnamefont {Balmaseda},
  \bibfnamefont {M.~A.}}, \bibinfo {author} {\bibnamefont {Balsamo},
  \bibfnamefont {G.}}, \bibinfo {author} {\bibnamefont {Bauer}, \bibfnamefont
  {P.}}, \bibinfo {author} {\bibnamefont {Bechtold}, \bibfnamefont {P.}},
  \bibinfo {author} {\bibnamefont {Beljaars}, \bibfnamefont {A.~C.~M.}},
  \bibinfo {author} {\bibnamefont {van~de Berg}, \bibfnamefont {L.}}, \bibinfo
  {author} {\bibnamefont {Bidlot}, \bibfnamefont {J.}}, \bibinfo {author}
  {\bibnamefont {Bormann}, \bibfnamefont {N.}}, \bibinfo {author} {\bibnamefont
  {Delsol}, \bibfnamefont {C.}}, \bibinfo {author} {\bibnamefont {Dragani},
  \bibfnamefont {R.}}, \bibinfo {author} {\bibnamefont {Fuentes}, \bibfnamefont
  {M.}}, \bibinfo {author} {\bibnamefont {Geer}, \bibfnamefont {A.~J.}},
  \bibinfo {author} {\bibnamefont {Haimberger}, \bibfnamefont {L.}}, \bibinfo
  {author} {\bibnamefont {Healy}, \bibfnamefont {S.~B.}}, \bibinfo {author}
  {\bibnamefont {Hersbach}, \bibfnamefont {H.}}, \bibinfo {author}
  {\bibnamefont {Hólm}, \bibfnamefont {E.~V.}}, \bibinfo {author}
  {\bibnamefont {Isaksen}, \bibfnamefont {L.}}, \bibinfo {author} {\bibnamefont
  {Kållberg}, \bibfnamefont {P.}}, \bibinfo {author} {\bibnamefont
  {K{\"o}hler}, \bibfnamefont {M.}}, \bibinfo {author} {\bibnamefont
  {Matricardi}, \bibfnamefont {M.}}, \bibinfo {author} {\bibnamefont {McNally},
  \bibfnamefont {A.~P.}}, \bibinfo {author} {\bibnamefont {Monge‐Sanz},
  \bibfnamefont {B.~M.}}, \bibinfo {author} {\bibnamefont {Morcrette},
  \bibfnamefont {J.~J.}}, \bibinfo {author} {\bibnamefont {Park}, \bibfnamefont
  {B.~K.}}, \bibinfo {author} {\bibnamefont {Peubey}, \bibfnamefont {C.}},
  \bibinfo {author} {\bibnamefont {de~Rosnay}, \bibfnamefont {P.}}, \bibinfo
  {author} {\bibnamefont {Tavolato}, \bibfnamefont {C.}}, \bibinfo {author}
  {\bibnamefont {Thépaut}, \bibfnamefont {J.~N.}}, and\ \bibinfo {author}
  {\bibnamefont {Vitart}, \bibfnamefont {F.}},\ }\bibfield  {title} {\enquote
  {\bibinfo {title} {The {ERA‐Interim} reanalysis: configuration and
  performance of the data assimilation system},}\ }\href
  {https://doi.org/https://doi.org/10.1002/qj.828} {\bibfield  {journal}
  {\bibinfo  {journal} {Quarterly Journal of the Royal Meteorological Society}\
  }\textbf {\bibinfo {volume} {137}},\ \bibinfo {pages} {553--597} (\bibinfo
  {year} {2011})}\BibitemShut {NoStop}%
\bibitem [{\citenamefont {Eggen}\ \emph {et~al.}(2022)\citenamefont {Eggen},
  \citenamefont {Dahl}, \citenamefont {N\"asholm},\ and\ \citenamefont
  {M{\ae}land}}]{eggen2021}%
  \BibitemOpen
  \bibfield  {author} {\bibinfo {author} {\bibnamefont {Eggen}, \bibfnamefont
  {M.~D.}}, \bibinfo {author} {\bibnamefont {Dahl}, \bibfnamefont {K.~R.}},
  \bibinfo {author} {\bibnamefont {N\"asholm}, \bibfnamefont {S.~P.}}, and\
  \bibinfo {author} {\bibnamefont {M{\ae}land}, \bibfnamefont {S.}},\
  }\bibfield  {title} {\enquote {\bibinfo {title} {Stochastic modeling of
  stratospheric temperature},}\ }\href
  {https://doi.org/10.1007/s11004-021-09990-6} {\bibfield  {journal} {\bibinfo
  {journal} {Mathematical Geosciences}\ } (\bibinfo {year} {2022}),\
  10.1007/s11004-021-09990-6}\BibitemShut {NoStop}%
\bibitem [{\citenamefont {Fasen-Hartmann}\ and\ \citenamefont
  {Mayer}(2021)}]{mayer21}%
  \BibitemOpen
  \bibfield  {author} {\bibinfo {author} {\bibnamefont {Fasen-Hartmann},
  \bibfnamefont {V.}}and\ \bibinfo {author} {\bibnamefont {Mayer},
  \bibfnamefont {C.}},\ }\bibfield  {title} {\enquote {\bibinfo {title}
  {Whittle estimation for continuous-time stationary state space models with
  finite second moments},}\ }\href {https://doi.org/10.1007/s10463-021-00802-6}
  {\bibfield  {journal} {\bibinfo  {journal} {Annals of the Institute of
  Statistical Mathematics}\ } (\bibinfo {year} {2021}),\
  10.1007/s10463-021-00802-6}\BibitemShut {NoStop}%
\bibitem [{\citenamefont {Fasen-Hartmann}\ and\ \citenamefont
  {Scholz}(2021)}]{fasenhartmann21}%
  \BibitemOpen
  \bibfield  {author} {\bibinfo {author} {\bibnamefont {Fasen-Hartmann},
  \bibfnamefont {V.}}and\ \bibinfo {author} {\bibnamefont {Scholz},
  \bibfnamefont {M.}},\ }\href
  {https://doi.org/https://doi.org/10.48550/arXiv.2102.11681} {\enquote
  {\bibinfo {title} {Factorization and discrete-time representation of
  multivariate carma processes},}\ } (\bibinfo {year} {2021})\BibitemShut
  {NoStop}%
\bibitem [{\citenamefont {Garc{\'i}a}, \citenamefont {Kl{\"u}ppelberg},\ and\
  \citenamefont {M{\"u}ller}(2011)}]{garcia_etal11}%
  \BibitemOpen
  \bibfield  {author} {\bibinfo {author} {\bibnamefont {Garc{\'i}a},
  \bibfnamefont {I.}}, \bibinfo {author} {\bibnamefont {Kl{\"u}ppelberg},
  \bibfnamefont {C.}}, and\ \bibinfo {author} {\bibnamefont {M{\"u}ller},
  \bibfnamefont {G.}},\ }\bibfield  {title} {\enquote {\bibinfo {title}
  {Estimation of stable carma models with an application to electricity spot
  prices},}\ }\href {https://doi.org/10.1177/1471082X1001100504} {\bibfield
  {journal} {\bibinfo  {journal} {Statistical Modelling}\ }\textbf {\bibinfo
  {volume} {11}},\ \bibinfo {pages} {447--470} (\bibinfo {year}
  {2011})}\BibitemShut {NoStop}%
\bibitem [{\citenamefont {Gardo{\'n}}(2004)}]{gardon04}%
  \BibitemOpen
  \bibfield  {author} {\bibinfo {author} {\bibnamefont {Gardo{\'n}},
  \bibfnamefont {A.}},\ }\bibfield  {title} {\enquote {\bibinfo {title} {The
  order of approximations for solutions of it{\^o}-type stochastic differential
  equations with jumps},}\ }\href {https://doi.org/10.1081/SAP-120030451}
  {\bibfield  {journal} {\bibinfo  {journal} {Stochastic Analysis and
  Applications}\ }\textbf {\bibinfo {volume} {22}},\ \bibinfo {pages}
  {679--699} (\bibinfo {year} {2004})}\BibitemShut {NoStop}%
\bibitem [{\citenamefont {G\'{o}mez}(2016)}]{gomez16}%
  \BibitemOpen
  \bibfield  {author} {\bibinfo {author} {\bibnamefont {G\'{o}mez},
  \bibfnamefont {V.}},\ }\href
  {https://doi.org/https://doi.org/10.1007/978-3-319-28599-3} {\emph {\bibinfo
  {title} {Multivariate Time Series With Linear State Space Structure}}}\
  (\bibinfo  {publisher} {Springer International Publishing},\ \bibinfo {year}
  {2016})\BibitemShut {NoStop}%
\bibitem [{\citenamefont {G\'{o}mez}(2019)}]{gomez19}%
  \BibitemOpen
  \bibfield  {author} {\bibinfo {author} {\bibnamefont {G\'{o}mez},
  \bibfnamefont {V.}},\ }\href@noop {} {\emph {\bibinfo {title} {Linear Time
  Series with MATLAB and OCTAVE}}}\ (\bibinfo  {publisher} {Springer
  International Publishing: Imprint: Springer},\ \bibinfo {year}
  {2019})\BibitemShut {NoStop}%
\bibitem [{\citenamefont {G\'{o}mez}(2020)}]{gomez_ssmmatlab}%
  \BibitemOpen
  \bibfield  {author} {\bibinfo {author} {\bibnamefont {G\'{o}mez},
  \bibfnamefont {V.}},\ }\href@noop {} {\enquote {\bibinfo {title}
  {Ssmmatlab},}\ }\bibinfo {howpublished}
  {\url{https://github.com/vgomezenriquez/ssmmatlab}} (\bibinfo {year}
  {2020})\BibitemShut {NoStop}%
\bibitem [{\citenamefont {Hitchcock}\ and\ \citenamefont
  {Simpson}(2014)}]{hitchcock14}%
  \BibitemOpen
  \bibfield  {author} {\bibinfo {author} {\bibnamefont {Hitchcock},
  \bibfnamefont {P.}}and\ \bibinfo {author} {\bibnamefont {Simpson},
  \bibfnamefont {I.~R.}},\ }\bibfield  {title} {\enquote {\bibinfo {title} {The
  downward influence of stratospheric sudden warmings},}\ }\href
  {https://doi.org/10.1175/JAS-D-14-0012.1} {\bibfield  {journal} {\bibinfo
  {journal} {Journal of the Atmospheric Sciences}\ }\textbf {\bibinfo {volume}
  {71}},\ \bibinfo {pages} {3856 -- 3876} (\bibinfo {year} {2014})}\BibitemShut
  {NoStop}%
\bibitem [{\citenamefont {James}, \citenamefont {Koreisha},\ and\ \citenamefont
  {Partch}(1985)}]{james85}%
  \BibitemOpen
  \bibfield  {author} {\bibinfo {author} {\bibnamefont {James}, \bibfnamefont
  {C.}}, \bibinfo {author} {\bibnamefont {Koreisha}, \bibfnamefont {S.}}, and\
  \bibinfo {author} {\bibnamefont {Partch}, \bibfnamefont {M.}},\ }\bibfield
  {title} {\enquote {\bibinfo {title} {A varma analysis of the causal relations
  among stock returns, real output, and nominal interest rates},}\ }\href
  {https://doi.org/https://doi.org/10.1111/j.1540-6261.1985.tb02389.x}
  {\bibfield  {journal} {\bibinfo  {journal} {The Journal of Finance (New
  York)}\ }\textbf {\bibinfo {volume} {40}},\ \bibinfo {pages} {1375--1384}
  (\bibinfo {year} {1985})}\BibitemShut {NoStop}%
\bibitem [{\citenamefont {Karpechko}, \citenamefont {Tummon},\ and\
  \citenamefont {Secretariat}(2016)}]{karpechko16}%
  \BibitemOpen
  \bibfield  {author} {\bibinfo {author} {\bibnamefont {Karpechko},
  \bibfnamefont {A.}}, \bibinfo {author} {\bibnamefont {Tummon}, \bibfnamefont
  {F.}}, and\ \bibinfo {author} {\bibnamefont {Secretariat}, \bibfnamefont
  {W.}},\ }\bibfield  {title} {\enquote {\bibinfo {title} {Climate
  predictability in the stratosphere},}\ }\href@noop {} {\bibfield  {journal}
  {\bibinfo  {journal} {Bulletin of World Meteorological Organization (WMO)}\
  }\textbf {\bibinfo {volume} {65}} (\bibinfo {year} {2016})}\BibitemShut
  {NoStop}%
\bibitem [{\citenamefont {Kloeden}\ and\ \citenamefont
  {Platen}(1992)}]{kloeden92}%
  \BibitemOpen
  \bibfield  {author} {\bibinfo {author} {\bibnamefont {Kloeden}, \bibfnamefont
  {P.~E.}}and\ \bibinfo {author} {\bibnamefont {Platen}, \bibfnamefont {E.}},\
  }\href@noop {} {\emph {\bibinfo {title} {Numerical solution of stochastic
  differential equations}}},\ Vol.~\bibinfo {volume} {23}\ (\bibinfo
  {publisher} {Springer},\ \bibinfo {year} {1992})\BibitemShut {NoStop}%
\bibitem [{\citenamefont {K\"{u}hn}\ and\ \citenamefont
  {Schilling}(2019)}]{kuhn19}%
  \BibitemOpen
  \bibfield  {author} {\bibinfo {author} {\bibnamefont {K\"{u}hn},
  \bibfnamefont {F.}}and\ \bibinfo {author} {\bibnamefont {Schilling},
  \bibfnamefont {R.~L.}},\ }\bibfield  {title} {\enquote {\bibinfo {title}
  {Strong convergence of the euler–maruyama approximation for a class of
  l\'{e}vy-driven sdes},}\ }\href
  {https://doi.org/https://doi.org/10.1016/j.spa.2018.07.018} {\bibfield
  {journal} {\bibinfo  {journal} {Stochastic Processes and Their Applications}\
  }\textbf {\bibinfo {volume} {129}},\ \bibinfo {pages} {2654--2680} (\bibinfo
  {year} {2019})}\BibitemShut {NoStop}%
\bibitem [{\citenamefont {Levendis}(2018)}]{Levendis18}%
  \BibitemOpen
  \bibfield  {author} {\bibinfo {author} {\bibnamefont {Levendis},
  \bibfnamefont {J.~D.}},\ }\href@noop {} {\emph {\bibinfo {title} {Time Series
  Econometrics: {Learning} Through Replication}}}\ (\bibinfo  {publisher}
  {Springer International Publishing: Imprint: Springer},\ \bibinfo {year}
  {2018})\BibitemShut {NoStop}%
\bibitem [{\citenamefont {Love}, \citenamefont {Matthews},\ and\ \citenamefont
  {Janacek}(2008)}]{love08}%
  \BibitemOpen
  \bibfield  {author} {\bibinfo {author} {\bibnamefont {Love}, \bibfnamefont
  {B.~S.}}, \bibinfo {author} {\bibnamefont {Matthews}, \bibfnamefont {A.~J.}},
  and\ \bibinfo {author} {\bibnamefont {Janacek}, \bibfnamefont {G.~J.}},\
  }\bibfield  {title} {\enquote {\bibinfo {title} {Real-time extraction of the
  madden–julian oscillation using empirical mode decomposition and
  statistical forecasting with a varma model},}\ }\href
  {https://doi.org/https://doi.org/10.1175/2008JCLI1977.1} {\bibfield
  {journal} {\bibinfo  {journal} {Journal of Climate}\ }\textbf {\bibinfo
  {volume} {21}},\ \bibinfo {pages} {5318--5335} (\bibinfo {year}
  {2008})}\BibitemShut {NoStop}%
\bibitem [{\citenamefont {Marquardt}\ and\ \citenamefont
  {Stelzer}(2007)}]{marquardt07}%
  \BibitemOpen
  \bibfield  {author} {\bibinfo {author} {\bibnamefont {Marquardt},
  \bibfnamefont {T.}}and\ \bibinfo {author} {\bibnamefont {Stelzer},
  \bibfnamefont {R.}},\ }\bibfield  {title} {\enquote {\bibinfo {title}
  {Multivariate carma processes},}\ }\href
  {https://doi.org/https://doi.org/10.1016/j.spa.2006.05.014} {\bibfield
  {journal} {\bibinfo  {journal} {Stochastic Processes and Their Applications}\
  }\textbf {\bibinfo {volume} {117}},\ \bibinfo {pages} {96--120} (\bibinfo
  {year} {2007})}\BibitemShut {NoStop}%
\bibitem [{\citenamefont {Parlange}\ and\ \citenamefont
  {Katz}(2000)}]{parlange00}%
  \BibitemOpen
  \bibfield  {author} {\bibinfo {author} {\bibnamefont {Parlange},
  \bibfnamefont {M.~B.}}and\ \bibinfo {author} {\bibnamefont {Katz},
  \bibfnamefont {R.~W.}},\ }\bibfield  {title} {\enquote {\bibinfo {title} {An
  extended version of the richardson model for simulating daily weather
  variables},}\ }\href
  {https://doi.org/https://doi.org/10.1175/1520-0450-39.5.610} {\bibfield
  {journal} {\bibinfo  {journal} {Journal of Applied Meteorology (1988)}\
  }\textbf {\bibinfo {volume} {39}},\ \bibinfo {pages} {610--622} (\bibinfo
  {year} {2000})}\BibitemShut {NoStop}%
\bibitem [{\citenamefont {Pham}\ and\ \citenamefont
  {Le~Breton}(1991)}]{pham91}%
  \BibitemOpen
  \bibfield  {author} {\bibinfo {author} {\bibnamefont {Pham}, \bibfnamefont
  {D.~T.}}and\ \bibinfo {author} {\bibnamefont {Le~Breton}, \bibfnamefont
  {A.}},\ }\bibfield  {title} {\enquote {\bibinfo {title} {Levinson-durbin-type
  algorithms for continuous-time autoregressive models and applications},}\
  }\href {https://doi.org/10.1007/BF02551381} {\bibfield  {journal} {\bibinfo
  {journal} {Mathematics of Control, Signals, and Systems}\ }\textbf {\bibinfo
  {volume} {4}},\ \bibinfo {pages} {69--79} (\bibinfo {year}
  {1991})}\BibitemShut {NoStop}%
\bibitem [{\citenamefont {Platen}\ and\ \citenamefont
  {Bruti-Liberati}(2010)}]{platen10}%
  \BibitemOpen
  \bibfield  {author} {\bibinfo {author} {\bibnamefont {Platen}, \bibfnamefont
  {E.}}and\ \bibinfo {author} {\bibnamefont {Bruti-Liberati}, \bibfnamefont
  {N.}},\ }\href@noop {} {\emph {\bibinfo {title} {Numerical Solution of
  Stochastic Differential Equations with Jumps in Finance}}},\ Vol.~\bibinfo
  {volume} {64}\ (\bibinfo  {publisher} {Springer},\ \bibinfo {year}
  {2010})\BibitemShut {NoStop}%
\bibitem [{\citenamefont {Protter}\ and\ \citenamefont
  {Talay}(1997)}]{protter97}%
  \BibitemOpen
  \bibfield  {author} {\bibinfo {author} {\bibnamefont {Protter}, \bibfnamefont
  {P.}}and\ \bibinfo {author} {\bibnamefont {Talay}, \bibfnamefont {D.}},\
  }\bibfield  {title} {\enquote {\bibinfo {title} {{The Euler scheme for Lévy
  driven stochastic differential equations}},}\ }\href
  {https://doi.org/10.1214/aop/1024404293} {\bibfield  {journal} {\bibinfo
  {journal} {The Annals of Probability}\ }\textbf {\bibinfo {volume} {25}},\
  \bibinfo {pages} {393 -- 423} (\bibinfo {year} {1997})}\BibitemShut {NoStop}%
\bibitem [{\citenamefont {Scaife}\ \emph {et~al.}(2022)\citenamefont {Scaife},
  \citenamefont {Baldwin}, \citenamefont {Butler}, \citenamefont
  {Charlton-Perez}, \citenamefont {Domeisen}, \citenamefont {Garfinkel},
  \citenamefont {Hardiman}, \citenamefont {Haynes}, \citenamefont {Karpechko},
  \citenamefont {Lim}, \citenamefont {Noguchi}, \citenamefont {Perlwitz},
  \citenamefont {Polvani}, \citenamefont {Richter}, \citenamefont {Scinocca},
  \citenamefont {Sigmond}, \citenamefont {Shepherd}, \citenamefont {Son},\ and\
  \citenamefont {Thompson}}]{scaife22}%
  \BibitemOpen
  \bibfield  {author} {\bibinfo {author} {\bibnamefont {Scaife}, \bibfnamefont
  {A.~A.}}, \bibinfo {author} {\bibnamefont {Baldwin}, \bibfnamefont {M.~P.}},
  \bibinfo {author} {\bibnamefont {Butler}, \bibfnamefont {A.~H.}}, \bibinfo
  {author} {\bibnamefont {Charlton-Perez}, \bibfnamefont {A.~J.}}, \bibinfo
  {author} {\bibnamefont {Domeisen}, \bibfnamefont {D.~I.~V.}}, \bibinfo
  {author} {\bibnamefont {Garfinkel}, \bibfnamefont {C.~I.}}, \bibinfo {author}
  {\bibnamefont {Hardiman}, \bibfnamefont {S.~C.}}, \bibinfo {author}
  {\bibnamefont {Haynes}, \bibfnamefont {P.}}, \bibinfo {author} {\bibnamefont
  {Karpechko}, \bibfnamefont {A.~Y.}}, \bibinfo {author} {\bibnamefont {Lim},
  \bibfnamefont {E.-P.}}, \bibinfo {author} {\bibnamefont {Noguchi},
  \bibfnamefont {S.}}, \bibinfo {author} {\bibnamefont {Perlwitz},
  \bibfnamefont {J.}}, \bibinfo {author} {\bibnamefont {Polvani}, \bibfnamefont
  {L.}}, \bibinfo {author} {\bibnamefont {Richter}, \bibfnamefont {J.~H.}},
  \bibinfo {author} {\bibnamefont {Scinocca}, \bibfnamefont {J.}}, \bibinfo
  {author} {\bibnamefont {Sigmond}, \bibfnamefont {M.}}, \bibinfo {author}
  {\bibnamefont {Shepherd}, \bibfnamefont {T.~G.}}, \bibinfo {author}
  {\bibnamefont {Son}, \bibfnamefont {S.-W.}}, and\ \bibinfo {author}
  {\bibnamefont {Thompson}, \bibfnamefont {D.~W.~J.}},\ }\bibfield  {title}
  {\enquote {\bibinfo {title} {Long-range prediction and the stratosphere},}\
  }\href@noop {} {\bibfield  {journal} {\bibinfo  {journal} {Atmospheric
  chemistry and physics}\ }\textbf {\bibinfo {volume} {22}},\ \bibinfo {pages}
  {2601--2623} (\bibinfo {year} {2022})}\BibitemShut {NoStop}%
\bibitem [{\citenamefont {Schlemm}\ and\ \citenamefont
  {Stelzer}(2012{\natexlab{a}})}]{schlemm12}%
  \BibitemOpen
  \bibfield  {author} {\bibinfo {author} {\bibnamefont {Schlemm}, \bibfnamefont
  {E.}}and\ \bibinfo {author} {\bibnamefont {Stelzer}, \bibfnamefont {R.}},\
  }\bibfield  {title} {\enquote {\bibinfo {title} {Multivariate carma
  processes, continuous-time state space models and complete regularity of the
  innovations of the sampled processes},}\ }\href@noop {} {\bibfield  {journal}
  {\bibinfo  {journal} {Bernoulli: Official Journal of the Bernoulli Society
  for Mathematical Statistics and Probability}\ }\textbf {\bibinfo {volume}
  {18}},\ \bibinfo {pages} {46--63} (\bibinfo {year}
  {2012}{\natexlab{a}})}\BibitemShut {NoStop}%
\bibitem [{\citenamefont {Schlemm}\ and\ \citenamefont
  {Stelzer}(2012{\natexlab{b}})}]{stelzer12}%
  \BibitemOpen
  \bibfield  {author} {\bibinfo {author} {\bibnamefont {Schlemm}, \bibfnamefont
  {E.}}and\ \bibinfo {author} {\bibnamefont {Stelzer}, \bibfnamefont {R.}},\
  }\bibfield  {title} {\enquote {\bibinfo {title} {Quasi maximum likelihood
  estimation for strongly mixing state space models and multivariate
  lévy-driven carma processes},}\ }\href
  {https://doi.org/https://doi.org/10.1214/12-EJS743} {\bibfield  {journal}
  {\bibinfo  {journal} {Electronic Journal of Statistics}\ }\textbf {\bibinfo
  {volume} {6}},\ \bibinfo {pages} {2185--2234} (\bibinfo {year}
  {2012}{\natexlab{b}})}\BibitemShut {NoStop}%
\bibitem [{\citenamefont {S{\"o}derstr{\"o}m}\ \emph
  {et~al.}(1997)\citenamefont {S{\"o}derstr{\"o}m}, \citenamefont {Fan},
  \citenamefont {Carlsson},\ and\ \citenamefont {Mossberg}}]{soderstrom97}%
  \BibitemOpen
  \bibfield  {author} {\bibinfo {author} {\bibnamefont {S{\"o}derstr{\"o}m},
  \bibfnamefont {T.}}, \bibinfo {author} {\bibnamefont {Fan}, \bibfnamefont
  {H.}}, \bibinfo {author} {\bibnamefont {Carlsson}, \bibfnamefont {B.}}, and\
  \bibinfo {author} {\bibnamefont {Mossberg}, \bibfnamefont {M.}},\ }\bibfield
  {title} {\enquote {\bibinfo {title} {Some approaches on how to use the delta
  operator when identifying continuous-time processes},}\ }in\ \href
  {https://doi.org/10.1109/CDC.1997.650755} {\emph {\bibinfo {booktitle}
  {Proceedings of the 36th IEEE Conference on Decision and Control}}}\
  (\bibinfo {year} {1997})\ pp.\ \bibinfo {pages} {890--895}\BibitemShut
  {NoStop}%
\bibitem [{\citenamefont {Thornton}\ and\ \citenamefont
  {Chambers}(2017)}]{thornton17}%
  \BibitemOpen
  \bibfield  {author} {\bibinfo {author} {\bibnamefont {Thornton},
  \bibfnamefont {M.~A.}}and\ \bibinfo {author} {\bibnamefont {Chambers},
  \bibfnamefont {M.~J.}},\ }\bibfield  {title} {\enquote {\bibinfo {title}
  {Continuous time arma processes: Discrete time representation and likelihood
  evaluation},}\ }\href {https://doi.org/10.1016/j.jedc.2017.03.012} {\bibfield
   {journal} {\bibinfo  {journal} {Journal of Economic Dynamics and Control}\
  }\textbf {\bibinfo {volume} {79}},\ \bibinfo {pages} {48--65} (\bibinfo
  {year} {2017})}\BibitemShut {NoStop}%
\bibitem [{\citenamefont {Todorov}\ and\ \citenamefont
  {Tauchen}(2006)}]{todorov_etal06}%
  \BibitemOpen
  \bibfield  {author} {\bibinfo {author} {\bibnamefont {Todorov}, \bibfnamefont
  {V.}}and\ \bibinfo {author} {\bibnamefont {Tauchen}, \bibfnamefont {G.}},\
  }\bibfield  {title} {\enquote {\bibinfo {title} {Simulation methods for
  l{\'e}vy-driven continuous-time autoregressive moving average (carma)
  stochastic volatility models},}\ }\href
  {https://doi.org/10.1198/073500106000000260} {\bibfield  {journal} {\bibinfo
  {journal} {Journal of Business \& Economic Statistics}\ }\textbf {\bibinfo
  {volume} {24}},\ \bibinfo {pages} {455--469} (\bibinfo {year}
  {2006})}\BibitemShut {NoStop}%
\bibitem [{\citenamefont {Wei}(2019)}]{wei_19}%
  \BibitemOpen
  \bibfield  {author} {\bibinfo {author} {\bibnamefont {Wei}, \bibfnamefont
  {W.~W.~S.}},\ }\href@noop {} {\emph {\bibinfo {title} {Multivariate time
  series analysis and applications}}}\ (\bibinfo  {publisher} {Wiley},\
  \bibinfo {year} {2019})\BibitemShut {NoStop}%
\end{thebibliography}%

\end{document}